\documentclass[12pt]{amsart} 
\usepackage{amsmath,amssymb,a4wide} 
\usepackage{dsfont}
\usepackage{mathabx}
\usepackage{color}
\usepackage{caption}
\usepackage{subcaption}
\usepackage{mathrsfs}
\usepackage{tikz,pgfplots}
\usepackage{color}
\usepackage{ulem}

\usepackage[utf8]{inputenc}

\newtheorem{theorem}{Theorem} 
\newtheorem{lemma}[theorem]{Lemma} 
\newtheorem{proposition}[theorem]{Proposition}

\newcommand{\Z}{{\mathbb Z}} 
\newcommand{\R}{{\mathbb R}} 
\newcommand{\C}{{\mathbb C}} 
\newcommand{\N}{{\mathbb N}}

\renewcommand{\L}{{\mathbb L}}
\renewcommand{\H}{{\mathbb H}}
\newcommand{\E}{{\mathbb E}}
\newcommand{\PP}{{\mathbb P}}
\newcommand{\U}{{\mathcal U}}
\newcommand{\dd}{{\rm d}}

\providecommand{\norm}[1]{\left\lVert#1\right\rVert}

\DeclareMathOperator{\sech}{sech}

\title[]{Lie--Trotter splitting for the nonlinear stochastic Manakov system}

\date{\today}

\author{Andr\'e Berg}
\address{Department of Mathematics and Mathematical Statistics,
  Ume{\aa} University, SE--901~87~Ume{\aa}, Sweden}
\email{andre.berglund@umu.se}

\author{David Cohen}
\address{University of Gothenburg, Chalmers, Sweden}
\email{david.cohen@chalmers.se}

\author{Guillaume Dujardin}
\address{Inria, Univ. Lille, CNRS, UMR 8524 - Laboratoire Paul Painlev{\'e} F-59000}
\email{guillaume.dujardin@inria.fr}

\begin{document}
	
	%%% ABSTRACT
	
	\begin{abstract}
          This article analyses the convergence of the Lie--Trotter splitting
          scheme for the 
		stochastic Manakov equation, a system arising in the study of pulse propagation 
		in randomly birefringent optical fibers. 
		First, we prove that the strong order of the numerical approximation is $1/2$ 
		if the nonlinear term in the system is globally Lipschitz. 
		Then, we show that the splitting scheme has convergence order $1/2$ 
		in probability and almost sure order $1/2^-~$ in the case of a cubic nonlinearity. 
		We provide several numerical experiments illustrating
                the aforementioned results and the efficiency
                  of the Lie--Trotter splitting scheme.
                  Finally, we numerically investigate the possible blowup of solutions for 
                  some power-law nonlinearities.
	\end{abstract}
	
	%%% TITLE AND AMS CLASSIFICATION
	
	\maketitle
		
		Stochastic Manakov equation. 
		Coupled system of stochastic nonlinear Schr\"odinger equations. 
		Numerical schemes. Splitting scheme. Lie--Trotter scheme. 
		Strong convergence. Convergence in probability. Almost sure convergence. Convergence rates. Blowup.  
		
		%%%% SECTION
		\section{Introduction}\label{sec-intro}
		
		% The most used words in the stored article introductions are:
		% PMD Manakov Equation (these three very linked), birefringence, length, optical fiber, polarization, pulse, dispersion
		% 
		% In order to stress the importance, they use:
		% The internet -> Optical communications -> Better / faster / rate / bandwidth
		% Dispersive effects
		
		% Order, wide to narrow.
		The Internet and its many areas of applications and dependencies create 
		a huge demand for faster optical communication systems. 
		One of the current limiting factors of high bit rate transmissions   
		is dispersive effects which accumulate over long distances \cite{MR3166967}.
		One of these limiting factors is due to polarization mode dispersion (PMD) which 
		follows from birefringence in the optical fibers. 
		This effect in turn may vary due to {\it e.\,g.} 
		core geometry, non-uniform anisotropy, 
		or mechanical distortions from point-like pressure or twisting. 
		These restrictive factors can together be modeled as random influences 
		leading to the Manakov PMD equation and its limiting 
		equation, the stochastic Manakov equation,   
		see for instance \cite{WaiMenyak,MR3024974} for details. 
		A precise definition of the stochastic Manakov equation is given below. 
		This stochastic partial differential equation (SPDE) thus serves as a model to 
		study long distance light propagation in random optical fibers.
		
		%These factors can together be modeled as random influences, and the stochastic Manakov PMD equation follows as a result from asymptotic dynamics (\cite{MR3024974}).
		
		Let us now discuss recent literature on the numerical analysis of the stochastic Manakov equation. 
		The work \cite{Gazeau:13} (see also \cite{GazeauPhd})
		numerically studies the impact of noise on Manakov solitons and soliton 
		wave-train propagation by the following time integrators: 
		the nonlinearly implicit Crank--Nicolson scheme, 
		the linearly implicit relaxation scheme, 
		and an explicit split-step scheme (the Lie--Trotter scheme). 
		The paper \cite{MR3166967} (see also \cite{GazeauPhd}) proves that the order of convergence in probability of 
		the Crank--Nicolson scheme is $1/2$. In addition, it is shown that this numerical integrator 
		preserves the $\L^2$-norm as does the exact solution 
        to the stochastic Manakov equation (see below for details). 
		Furthermore, it is numerically observed in the reference \cite{MR3166967} that 
		the almost-sure order of convergence of the relaxation scheme and the split-step 
		scheme is $1/2^-$. To the best of our knowledge, no proofs for these orders 
		of convergence exist. 
		Finally, the recent reference \cite{bcd20} proves, 
		among other things, that the order of convergence in probability 
		of an exponential integrator is $1/2$. 
		
		The main goal of this article is to analyse a
		linearly implicit version of the Lie--Trotter integrator for an efficient  
		time integration of the stochastic Manakov system. This numerical integrator is an application of the 
		classical deterministic Lie--Trotter splitting from \cite{MR0108732} 
		to the present stochastic setting. The outline of the paper is as follows. 
		The numerical integrator is described in Section~\ref{sec-lt}.
                In Section~\ref{sec-conv}, we theoretically confirm that this time integrator applied
                to the Manakov system with a truncated Lipschitz nonlinear term 
		has the same order of convergence as that of the nonlinearly implicit Crank--Nicolson 
		scheme from \cite{MR3166967} and that of the exponential integrator from \cite{bcd20}.
                 This is achieved in Theorem~\ref{thm:LTSplConvStrong}, which
                  is the main theoretical result of this paper.
                 As a consequence, we prove in Section~\ref{sec:convprob}
                  that the order of convergence of the scheme applied to the Manakov system
                  with untruncated nonlinearity is $1/2^-$ in probability
                  (see Proposition \ref{propProba}) and almost surely (see Proposition~\ref{prop-as}).
		    Finally, Section~\ref{sec-numexp} is devoted to numerical experiments.
              In particular, we illustrate numerically the strong order of convergence of the scheme
              applied to the Manakov system,
              its order of convergence in probability and its order of almost-sure convergence.
              In addition to comparing its qualitative properties ($\L^2$-norm preservation)
              and its computational cost with other numerical methods from the literature
              applied to the stochastic Manakov equation,
              we use this new scheme to investigate the existence of a critical power-law exponent
              for the stochastic Manakov system, which is a theoretical open problem \cite{MR3166967} at the
              time of writing. 
              
		%%%%%%%%%%%%%%%%%%%%%%%%%%%%%%%%%%%%%%%%%%%%%%%%%%%%%%%%%%%%%%%%%%%%%%%%%%%%%%%%%%%%%%%%%%%%
            \section{A Lie--Trotter scheme for the nonlinear stochastic Manakov system}\label{sec-lt}
            In this section, we set notation, we introduce the stochastic Manakov equation
              and the Lie--Trotter splitting scheme that we analyse and use in the
            next sections.
            
		Let $(\Omega, \mathcal{F},\mathbb{P})$ be a probability space on which a three-dimensional 
		standard Brownian motion $W(t):=(W_1(t), W_2(t), W_3(t))$ is defined. 
		We endow this probability space with the 
		complete filtration $\mathcal{F}_t$ generated by $W(t)$. 
		
		Following \cite{MR3166967}, we write the nonlinear stochastic Manakov system as
		\begin{equation}
		\label{eq:Manakov}
		i\text{d}X
		+\partial^2_x X \,\text{d}t
		+ i \sqrt{\gamma} \sum_{k=1}^{3} \sigma_k \partial_x X \circ \text{d} W_k
		+ |X|^2 X \,\text{d}t= 0,
		\end{equation}
		where $X=X(t,x)=(X_1,X_2)$ is the unknown vector-valued function with values in $\C^2$, 
		$\circ$ denotes the Stratonovich product, 
		$\gamma\geq 0$ measures the intensity of the noise, 
		$|X|^2=|X_1|^2+|X_2|^2$ is the nonlinear coupling, 
		and $\sigma_1$, $\sigma_2$ and $\sigma_3$ are the
		classical Pauli matrices defined by
		\begin{equation*}
		\sigma_1=
		\begin{pmatrix}
		0 & 1\\
		1 & 0
		\end{pmatrix},
		\qquad
		\sigma_2=
		\begin{pmatrix}
		0 & -i\\
		i & 0
		\end{pmatrix},
		\quad
		{\rm and}
		\quad
		\sigma_3=
		\begin{pmatrix}
		1 & 0\\
		0 & -1
		\end{pmatrix}.
		\end{equation*}
		The mild form of the stochastic Manakov equation \eqref{eq:Manakov} reads 
		\begin{equation}
		\label{mildManakov}
		X(t)=U(t,0)X_0+i\int_0^tU(t,s)F(X(s))\,\text{d}s, 
		\end{equation}
		where $X_0 = (X_{0,1},X_{0,2})$ denotes the initial value of the problem,
		$U(t,s)$ for $t\geq s$ with $s,t\in\R_+$ 
		is the random unitary propagator defined as the unique solution 
		to the linear part of \eqref{eq:Manakov}, and $F(X)=|X|^2X$.
		
		Let $p\geq1$. We define $\L^p:=\L^p(\R):=(L^p(\R;\C))^2$ the Lebesgue spaces of functions with values in $\C^2$. 
		We equip $\L^2$ with the real scalar product
		$\displaystyle (u,v)_2=\sum_{j=1}^2\mathrm{Re}\left(\int_\R u_j\overline{v_j}\,\text{d}x\right)$.
		Further, for $m\in\N$, we denote by $\H^m:=\H^m(\R)$ the space 
		of functions in $\L^2$ with their $m$ first derivatives in $\L^2$,
                for which we denote the corresponding 
                norm by $\norm{\cdot}_{\H^m} = \norm{\cdot}_m$.
		
		Just as for the classical cubic Schr\"odinger equation, the $\L^2$-norm of the exact solution to the stochastic Manakov system 
		\eqref{eq:Manakov} is almost surely preserved: 
		$$
		\norm{X(t)}_{\L^2}=\norm{X_0}_{\L^2}
		$$ 
		for all $t\in[0,\tau^*[$, where $\tau^*>0$ is a stopping time, see \cite{MR3024974} for details. 
		This is not the case for the evolution of the Hamiltonian, or total energy 
		$$
		H(X):=\frac12\int_\R\left|\frac{\partial X}{\partial x} \right|^2\,\text{d}x-\frac14\int_\R\left|X\right|^4\,\text{d}x, 
		$$
		where $|X|^4=(|X|^2)^2=(|X_1|^2+|X_2|^2)^2$, 
		as shown in \cite[Lemma~3.1]{MR3024974}. 
		
		For the time-integration of the system \eqref{eq:Manakov}, one has to face two issues.
		First, the linear part of this SPDE generates a stochastic group which is not easy 
		to compute. In particular, since the Pauli matrices do not commute, 
		it is not the product of the stochastic semi-groups associated to each
		Brownian motion with the group generated by $i\partial^2_x$. 
		Second, the nonlinear coupling term $|X|^2 X$ often leads to 
		implicit numerical methods that are costly to solve, see for instance the Crank--Nicolson scheme from \cite{MR3166967}. 
		
		Therefore, we numerically approximate solutions to the stochastic Manakov equation \eqref{eq:Manakov} 
		with the Lie--Trotter splitting scheme 
		\begin{equation}
		\label{ltsplit}
		X^{n+1}=U_{h,n+1}\left(X^n+i\int_{t_n}^{t_{n+1}}F(Y^n(s))\,\text{d}s\right),
		\end{equation}
		where $h>0$ denotes the stepsize, $t_n=nh$ for nonnegative integers $n$, 
		$U_{h,n+1}=\left( Id+\frac12 H_{h,n} \right)^{-1}\left( Id-\frac12 H_{h,n} \right)$, 
		with $Id$ the identity operator and $\displaystyle H_{h,n}=-ihI_2\partial_x^2+\sqrt{\gamma h}\sum_{k=1}^3\sigma_k\chi_k^n\partial_x$. Here, $I_2$ is the $2\times2$ identity matrix 
		and $\sqrt{h}\chi_k^n=W_k((n+1)h)-W_k(nh)$, for $k=1,2,3$, are i.i.d. Wiener increments. Furthermore, 
		$Y^n$ is the exact solution to the nonlinear differential equation $i\,\text{d}Y+F(Y)\,\text{d}t=0$ with initial value $X^n$ at time $t=t_n$.
		Iterating the recurrence given by \eqref{ltsplit}, one obtains the discrete mild form 
		of the Lie--Trotter splitting scheme
		\begin{equation}
		\label{eq:SplRec}
		X^{n} = \mathcal{U}_h^{n,0}X_0 + i\sum_{l=0}^{n-1} \mathcal{U}_h^{n,l} \int_{t_l}^{t_{l+1}}F(Y^l(s))\,\text{d}s, 
		\end{equation}
		where $\U_{h}^{n,l}:=U_{h,n}\cdot\ldots\cdot U_{h,l+1}$ with $\U_h^{0,0}=Id$ if needed.
		
		As the exact solution to the SPDE \eqref{eq:Manakov},
		we have that the Lie--Trotter scheme also preserves the 
		$\L^2$-norm almost surely:
		\begin{lemma}
			\label{thm:LTPres}
			The Lie--Trotter splitting scheme \eqref{ltsplit} preserves the $\L^2$-norm almost surely.
		\end{lemma}
		\begin{proof}
			Choose $n\in\N$ such that the scheme is well-defined at step $n+1$.
			By definition of $Y^n$, since $Y^n_1$ and $Y_2^n$ 
			solve pointwise in $x\in\R$ the following differential equations,
			\begin{equation*}
			\frac{\text{d}}{\text{d}t} Y_1(t) =-i (|Y_1(t)|^2+|Y_2(t)|^2)Y_1(t)
			\qquad \text{and} \qquad
			\frac{\text{d}}{\text{d}t} Y_2(t) =-i (|Y_1(t)|^2+|Y_2(t)|^2)Y_2(t),
			\end{equation*}
			over $[t_n,t_{n+1}]$ with $Y^n(t_n)=X^n$, 
			we have for all $t\in[t_n,t_{n+1}]$,
			\begin{align*}
			\frac{\text{d}}{\text{d}t} \left|Y^{n}(t)\right|^2
			&= \frac{\text{d}}{\text{d}t} \left(|Y^n_1(t)|^2 + |Y^n_2(t)|^2\right)^2\\
			&  = 4 \left(|Y^n_1(t)|^2 + |Y^n_2(t)|^2\right)
			\mathrm{Re}\left(\overline{Y^n_1(t)}\frac{\text{d}}{\text{d}t}Y^n_1(t)+\overline{Y^n_2(t)}\frac{\text{d}}{\text{d}t}Y^n_2(t)\right)
			\\
			& = 0,
			\end{align*}
			pointwise in $x\in\R$. In particular, we have $|Y^n(t_n)|^2=|Y^n(t_{n+1})|^2$.
			Integrating this last identity over $\R$ yields
			$\norm{Y^n(t_n)}_{\L^2}^2=\norm{Y^n(t_{n+1})}_{\L^2}^2$.
			
			Using the above, the fact that $U_{h,n+1}$ is an isometry over $\L^2$, 
			see for instance \cite[Appendix~5]{bcd20}, and the definition of $Y^n$, one then obtains
			$$
			\norm{X^{n+1}}_{\L^2}
			= \norm{X^{n}+i\int_{t_n}^{t_{n+1}}F(Y^n(s))\,\text{d}s}_{\L^2}
			= \norm{Y^{n}(t_{n+1})}_{\L^2}
			= \norm{Y^{n}(t_{n})}_{\L^2}
			= \norm{X^n}_{\L^2}.
			$$
		\end{proof}
		
		%%%%%%%%%%%%%%%%%%%%%%%%%%%%%%%%%%%%%%%%%%%%%%%%%%%%%%%%%%%%%%%%%%%%%%%%%%%%%%
		\section{Convergence analysis of the Lie--Trotter splitting scheme}\label{sec-conv}
		In this section, we consider the convergence analysis of the Lie--Trotter splitting \eqref{ltsplit} 
		where we have a globally Lipschitz continuous and bounded nonlinearity in \eqref{mildManakov}. 
		This is the case for instance, when one introduces 
		a cut-off function for the cubic nonlinearity present in \eqref{eq:Manakov}: 
		Let $R>0$ and $\theta\in\mathcal{C}^\infty(\R_+)$, with $\theta\ge 0$, $\text{supp}(\theta)\subset[0,2]$ and $\theta\equiv1$ on $[0,1]$. 
		For $x\geq0$, we set $\theta_R(x)=\theta(\frac xR)$ and define $F_R(X)=\theta_R(\norm{X}^2_{1})|X|^2X$.

		We next present some properties of the function $F_R$ as well as of the numerical solution, 
		given by the Lie--Trotter splitting \eqref{ltsplit}, 
		of the SPDE \eqref{mildManakov} with the cut-off nonlinearity $F_R$. 
		
		\begin{lemma}
			\label{lemma:bounds}
			%Consider the nonlinearity $F(X) = |X|^2X$ and the cut-off nonlinearity $F_R(X) = \theta_R(\norm{X}^2_{1})F(X)$, 
			%for some $R>0$. 
			%Then for $X\in\H^1$, there exists a positive constant $C$ such that
			%\begin{equation*}
			%\norm{F_R(X)}_1 
			%\le CR^{3/2} \coco{check??}.
			% \end{equation*}
				There exists a positive constant $C$ such that for all $R>0$ and all $X\in\H^1$,
				\begin{equation*}
				\norm{F_R(X)}_{1}\leq CR^{3/2}.
				\end{equation*}
			Furthermore, for all $R>0$, the function $F_R$ is globally Lipschitz continuous in $\H^1$, 
			with corresponding Lipschitz constant $L_R$.
			The map $F_R$ also sends bounded subsets of $\H^2$ to bounded subsets of $\H^2$, resp. $\H^6$ to $\H^6$. 
			Finally, the numerical solution of \eqref{mildManakov}, with the cut-off nonlinearity $F_R$, given by the Lie--Trotter splitting scheme \eqref{ltsplit} 
			is almost surely bounded in $\H^m$ for all $m\in\{1,2,6\}$:
                        For all $X^0\in \H^m$, for all $T>0$, there exists
                          a positive constant $C(\norm{X_0}_m,T,L_R)$ such that 
                        for all integer $N$ large enough, for all $n=1,\ldots,N$, one has
			$$
			\displaystyle\norm{X^n}_{m}+\sup_{t_n\le s\le t_{n+1}} \norm{Y^n(s)}_m  \le C(\norm{X_0}_m,T,L_R) \quad \text{a.s}. 
			$$
		\end{lemma}
		\begin{proof}
			We only highlight parts of the proofs. 
			
			In order to show that $F_R$ is globally Lipschitz continuous from $\H^1$ to $\H^1$, one first observes that $F_R$ is of class 
			$\mathcal C^\infty$ and vanishes outside the ball $B_{2R}^{H^1}\times B^{H^1}_{2R}\subset B_{2R\sqrt{2}}^{\H^1}$ 
			, where $H^1=H^1(\R;\C)$. Moreover, the derivative of $F_R$ is bounded on $\H^1$.
			% $\norm{F_R'(X)}_{\mathcal L(\H^1,\H^1)}\leq C${\color{blue}\sout{R}} uniformly for all $X\in\H^1$.
                        The mean value theorem then implies that $F_R$
                        is globally Lipschitz continuous, and we denote by $L_R$ the
                        corresponding Lipschitz constant.
			
			In order to show that $F_R$ sends bounded subsets of $\H^2$ to bounded subsets 
			of $\H^2$, one uses the definition of $\theta_R$ and the fact that $H^2$ is an algebra to get
			\begin{align*}
			\norm{F_R(X)}_2^2
			&=\norm{F_R((X_1,X_2))}_2^2
			\leq C 
			\left|
			\theta_R(\norm{X}^2_1)
			\right|^2
			\left( 
			\norm{|X_1|^2X_1+|X_2|^2X_1}_{H^2}^2
			+\norm{|X_1|^2X_{2}+|X_2|^2X_2}_{H^2}^2 
			\right)\\
			&\leq C \left(
			\norm{X_1}_{H^2}^6
			+\norm{X_2}_{H^2}^4\norm{X_1}_{H^2}^2
			+\norm{X_1}_{H^2}^4\norm{X_2}_{H^2}^2
			+\norm{X_2}_{H^2}^6
			\right),
			\end{align*}
			%\begin{align*}
			%\norm{F_R(X)}_2^2
			%&=\norm{F_R((X_1,X_2))}_2^2\leq C \left|\theta_R(\norm{X}^2_1)\right|^2\left( \norm{|X_1|^2X_1+|X_2|^2X_1}_2^2+\norm{|X_1|^2X_1+|X_2|^2X_2}_2^2 \right)\\
			%&\leq 1\cdot C^2 \left(
			%\norm{X_1}_2^6+\norm{X_2}_2^4\norm{X_1}_2^2+\norm{X_1}_2^4\norm{X_2}_2^2+\norm{X_2}_2^6
			%\right)
			%\end{align*}
			which is bounded if $\norm{(X_1,X_2)}_2$ is bounded. 
			
			Finally, when considering the SPDE \eqref{mildManakov} with the cut-off Lipschitz nonlinearity $F_R$, 
			it is classical to show, using for instance Picard's iterations and Gr\"onwall's lemma, 
			that the Lie--Trotter splitting scheme is almost surely bounded in $\H^m$, with $m\in\{1,2,6\}$.

		\end{proof}

		With the above preparation, we can now show strong convergence of the Lie--Trotter splitting scheme 
		when applied to the stochastic Manakov equation \eqref{mildManakov} with a cut-off nonlinearity $F_R$. 
		\begin{theorem}
			\label{thm:LTSplConvStrong}
			Let $R>0$, $T\geq0$, $N\in\N$, $h=T/N$, $p\geq1$, and $X_0\in \H^6$.
			Consider the stochastic Manakov equation \eqref{mildManakov} with the cut-off nonlinearity $F_R$. 
			Then, the Lie--Trotter splitting scheme \eqref{ltsplit} has strong order of convergence $1/2$:
			There exists $h_0>0$ such that
			$$
			\forall h\in (0,h_0),\qquad
			\E\left[ \max_{n=0,1,\ldots, N}\norm{X^n-X(t_n)}_{\H^1}^{2p} \right]\leq C h^p,
			$$
			where $C = C(\norm{X_0}_6,T,L_R,p,\gamma)$. 
			% \coco{$L_R$ is enough?? do not need $R$??}
			% To follow the notation used in J^5, which gets a constant from Lemma 5.4, we get an R. We could probably replace L_R with R though.
			% Depending on how the Lipschitz proof goes.
		\end{theorem}
		
		\begin{proof}
                  For ease of presentation, in the proof below, we 
                  remove the subscript $R$ in the stochastic processes $X_R(t)$ and $X_R^n$. 
			
			Let us denote the difference $X^n-X(t_n)$ by $e^n$. Using the mild equations \eqref{mildManakov} and \eqref{eq:SplRec}, 
			one gets 
			\begin{align*}
			\norm{e^n}_1
			=& \norm{\mathcal{U}_h^{n,0}X_0 
				+ i\sum_{l=0}^{n-1} \mathcal{U}_h^{n,l}\int_{t_l}^{t_{l+1}}F_R(Y^l(s))\,\text{d}s
				- U(t_n,0)X_0
				- i\int_{0}^{t_n}U(t_n,s)F_R(X(s))\,\text{d}s
			}_1\\
			\le& \norm{\left(\mathcal{U}^{n,0}_h-U(t_n,0)\right)X_0}_1
			+ \norm{\sum_{l=0}^{n-1} \int_{t_l}^{t_{l+1}} \left( \mathcal{U}_h^{n,l}F_R(Y^l(s)) - U(t_n,s)F_R(X(s)) \right) \,\text{d}s}_1\\
			=:&\, I^n_1 + I^n_2.
			\end{align*}
			We begin by estimating the term $I^n_2$ using the following decomposition
			\begin{align*}
			I^n_2 &= \bigg\lVert
			\sum_{l=0}^{n-1} \int_{t_l}^{t_{l+1}}
			\mathcal{U}_h^{n,l}F_R(Y^l(s))-U(t_n,t_l)F_R(Y^l(s))+U(t_n,t_l)F_R(Y^l(s))-U(t_n,s)F_R(Y^l(s))\\
			&\quad+U(t_n,s)F_R(Y^l(s))-U(t_n,s)F_R(X^l)+U(t_n,s)F_R(X^l)-U(t_n,s)F_R(X(t_l))\\
			&\quad+U(t_n,s)F_R(X(t_l))-U(t_n,s)F_R(X(s)) \,\text{d}s
			\bigg\rVert_1\\
			&\le \norm{\sum_{l=0}^{n-1} \int_{t_l}^{t_{l+1}} \left(\mathcal{U}_h^{n,l}-U(t_n,t_l)\right)F_R(Y^l(s)) \,\text{d}s}_1
			+\norm{\sum_{l=0}^{n-1} \int_{t_l}^{t_{l+1}} \left(U(t_n,t_l)-U(t_n,s)\right)F_R(Y^l(s)) \,\text{d}s}_1\\
			&\quad+\norm{\sum_{l=0}^{n-1} \int_{t_l}^{t_{l+1}} U(t_n,s)\left(F_R(Y^l(s))-F_R(X^l)\right) \,\text{d}s}_1
			+\norm{\sum_{l=0}^{n-1} \int_{t_l}^{t_{l+1}} U(t_n,s)\left(F_R(X^l)-F_R(X(t_l))\right) \,\text{d}s}_1\\
			&\quad+\norm{\sum_{l=0}^{n-1} \int_{t_l}^{t_{l+1}} U(t_n,s)\left(F_R(X(t_l))-F_R(X(s))\right) \,\text{d}s}_1
			=: J_1^n +J_2^n +J_3^n +J_4^n +J_5^n.
			\end{align*}
			In estimating these five terms, we repeatedly make
                        use of the facts that $F_R$ is globally Lipschitz continuous, 
			and that $\mathcal{U}^{n,l}_h$ and $U(t,s)$ are isometries on $\H^1$,
                        see \cite{MR3166967,bcd20}. 
			
			In order to bound the first term, $J_1^n$, we first define 
			$$
			Y^*(t) = \sum_{l = 0}^{N-1} \mathds{1}_{[t_l,t_{l+1})}(t) Y^l(t).
			$$
			We then use 
			\cite[Proposition 2.2]{MR3166967} (strong convergence for linear problems, i.\,e when $F_R\equiv 0$) and Lemma~\ref{lemma:bounds} (the almost sure boundedness of the Lie--Trotter splitting scheme in $\H^6$) 
			to conclude that
			$$\forall t \in[0,T],\quad 
			\E\left[
			\max_{n=0,1,\ldots,N-1}
			\max_{l=0,1,\ldots,n}
			\norm{\left(\mathcal{U}_h^{n,l}-U(t_n,t_l)\right)F_R(Y^*(t))}_1^{2p}
			\right]
			\le Ch^p.
			$$
			With this, we get through H\"older's inequality that 
			\begin{align*}
			\E\left[\max_{n=0,1,\ldots,N}(J^n_1)^{2p}\right]
			&\le \E\left[
			\max_{n=0,1,\ldots,N}\left(
			\sum_{l=0}^{n-1} \int_{t_l}^{t_{l+1}} \norm{\left(\mathcal{U}_h^{n,l}-U(t_n,t_l)\right)F_R(Y^l(s))}_1\,\text{d}s 
			\right)^{2p}
			\right]\\
			&\le \E\left[
			\max_{n=0,1,\ldots,N-1}\left(
			\int_{0}^{T}
			\max_{l=0,1,\ldots,n}
			%			T 
			%			\sup_{t_l\le s\le t_{l+1}} 
			\norm{\left(\mathcal{U}_h^{n,l}-U(t_n,t_l)\right)F_R(Y^*(s))}_1
			\,\text{d}s 
			\right)^{2p}
			\right]\\
			&\le 
			T^{2p-1}
			\int_{0}^{T}
			\E\left[
			%			\sup_{t\in[0,T]} 
			\max_{n=0,1,\ldots,N-1}
			\max_{l=0,1,\ldots,n}
			\norm{\left(\mathcal{U}_h^{n,l}-U(t_n,t_l)\right)F_R(Y^*(t))}_1^{2p}
			\right]
			\,\text{d}t \\
			%			&\le 
			%			\E\left[\max_{n=0,1,\ldots,N}
			%			\left(
			%			\sum_{l=0}^{n-1}
			%			h
			%			\norm{\mathcal{U}_h^{n,l}-U(t_n,t_l)}_{\mathcal{L}(\H^6,\H^1)}
			%			\sup_{t_l\le s\le t_{l+1}}
			%			\norm{F_R(Y^l(s))}_6
			%			\right)^{2p}
			%			\right]\\
			%			&\le
			%			C
			%			\E\left[\max_{n=0,1,\ldots,N}
			%			\max_{l=0,1,\ldots,n-1}
			%			\norm{\mathcal{U}_h^{n,l}-U(t_n,t_l)}_{\mathcal{L}(\H^6,\H^1)}^{2p}
			%			\right]\\
			&\le C_1\left(\norm{X_0}_{6},T,L_R,p,\gamma\right) h^p.
			\end{align*}
			Using the isometry property of $U(t_n,s)$ and H\"older's inequality, we obtain 
			\begin{align}\label{truc2ouf}
			\E&\left[\max_{n=0,1,\ldots,N}(J^n_2)^{2p}\right]
			\le \E\left[
			\max_{n=0,1,\ldots,N}
			\left(
			\sum_{l=0}^{n-1} \int_{t_l}^{t_{l+1}} \norm{U(t_n,s)\left(U(s,t_l)-Id\right)F_R(Y^l(s))}_1\,\text{d}s 
			\right)^{2p}
			\right]\nonumber\\
			&\le Ch^{2p}
			\E\left[ \left( \left(\sum_{l=0}^{N-1}1^{\frac{2p}{2p-1}}\right)^{\frac{2p-1}{2p}}\left(\sum_{l=0}^{N-1}\sup_{t_l\leq s\leq t_{l+1}}\norm{\left(Id-U(s,t_l)\right)F_R(Y^l(s))}^{2p}_1\right)^{\frac1{2p}} \right)^{2p}\right]\nonumber\\
			&\leq Ch^{2p}N^{2p-1}\sum_{l=0}^{N-1}\E\left[ \sup_{t_l\leq s\leq t_{l+1}}\norm{\left(Id-U(s,t_l)\right)F_R(Y^l(s))}^{2p}_1\right].
			\end{align}
			In order to estimate the above expectation, we write this term as 
			\begin{align}\label{truc2ouf2}
			\E&\left[ \sup_{t_l\leq s\leq t_{l+1}}\norm{\left(Id-U(s,t_l)\right)( F_R(Y^l(t_l))-F_R(Y^l(t_l))+F_R(Y^l(s)) )}^{2p}_1\right]\nonumber\\
			&\leq 
			C\E\left[ \sup_{t_l\leq s\leq t_{l+1}}\norm{\left(Id-U(s,t_l)\right) F_R(Y^l(t_l))}^{2p}_1\right]+C\E\left[ \sup_{t_l\leq s\leq t_{l+1}}\norm{\left(Id-U(s,t_l)\right)\left( F_R(Y^l(s))-F_R(Y^l(t_l)) \right)}^{2p}_1\right],
			\end{align}
			using the triangle inequality. 
			
			The first term in the equation above is the exact solution to the linear SPDE 
			$\displaystyle i\text dZ(t)+\frac{\partial^2Z(t)}{\partial x^2}\text dt+i\sqrt{\gamma}\sum_{k=1}^3\sigma_k\frac{\partial Z(t)}{\partial x}\circ\,\text d W_k(t)=0$ with initial value $F_R( Y^l(t_l))$ 
			at initial time $t_l$ which has the mild Ito form 
			$$
			Z(t)-F_R(Y^l(t_l))=\left(S(t-t_l)-Id\right)F_R(Y^l(t_l))+i\sqrt{\gamma}\sum_{k=1}^3\int_{t_l}^tS(t-u)\sigma_k\partial_xZ(u)\,\text dW_k(u),
			$$
			where $S(t)$ is the group solution to the free Schr\"odinger equation. Owning at the regularity property of the group $S$ (see for instance the first inequality in the proof of \cite[Lemma 4.2.1]{GazeauPhd}), 
			the fact that the numerical solution $Y^l$ is bounded, that $F_R$ sends bounded sets from $\H^2$ to $\H^2$, 
			and Burkholder--Davis--Gundy's inequality (for the second term), one obtains the following bound
			\begin{align*}
			\E\left[ \sup_{t_l\leq s\leq t_{l+1}}\norm{\left(Id-U(s,t_l)\right)F_R(Y^l(t_l))}^{2p}_1\right]\leq Ch^p. 
			\end{align*}
			Using the fact that the random propagator $U$ is an isometry, that $F_R$ is Lipschitz continuous, and $Y^l$ is solution to a Lipschitz differential equation, one gets the estimate 
			$$
			\E\left[ \sup_{t_l\leq s\leq t_{l+1}}\norm{\left(Id-U(s,t_l)\right)\left( F_R(Y^l(s))-F_R(Y^l(t_l)) \right)}^{2p}_1\right]\leq \E\left[ \sup_{t_l\leq s\leq t_{l+1}}\norm{F_R(Y^l(s))-F_R(Y^l(t_l))}^{2p}_1\right]\leq Ch^{2p},
			$$
			for the second term in \eqref{truc2ouf2}. 
			
			Combining the estimates above, one finally arrives at
			\begin{align*}
			\E\left[\max_{n=0,1,\ldots,N}(J^n_2)^{2p}\right]
			\leq Ch^{2p}N^{2p-1}Nh^p\leq Ch^p.
			\end{align*}
			For the third term, by definition of $Y^l$ and Lemma~\ref{lemma:bounds}, we get 
			\begin{align*}
			\E\left[\max_{n=0,1,\ldots,N}(J^n_3)^{2p}\right]
			&\le 
			\E\left[
			\max_{n=0,1,\ldots,N}
			\left(
			\sum_{l=0}^{n-1} \int_{t_l}^{t_{l+1}} \norm{U(t_n,s)\left(F_R(Y^l(s))-F_R(X^l)\right)}_1\,\text{d}s 
			\right)^{2p}
			\right]\\
			&\le 
			\E\left[
			\max_{n=0,1,\ldots,N}
			\left(
			L_R
			\sum_{l=0}^{n-1} 
			\int_{t_l}^{t_{l+1}} 
			\norm{
				\int_{t_l}^{s} 
				F_R(Y^l(r))
				\,\text{d}r
			}_1
			\,\text{d}s 
			\right)^{2p}
			\right]\\
			&\le C_3(T,L_R,p) h^{2p}.
			\end{align*}
			For the estimation of the fourth term, one may use the original error term involving $e^n$:
			\begin{align*}
			\E\left[\max_{n=0,1,\ldots,N}(J^n_4)^{2p}\right]
			&\le 
			\E\left[
			\max_{n=0,1,\ldots,N}
			\left(
			\sum_{l=0}^{n-1} \int_{t_l}^{t_{l+1}} 
			\norm{
				U(t_n,s)\left(F_R(X^l)-F_R(X(t_l))\right)
			}_1\,\text{d}s 
			\right)^{2p}
			\right]\\
			&\le 
			\E\left[
			\max_{n=0,1,\ldots,N}
			\left(
			L_R
			\sum_{l=0}^{n-1} 
			\int_{t_l}^{t_{l+1}} 
			\norm{
				X^l-X(t_l)
			}_1
			\,\text{d}s 
			\right)^{2p}
			\right]\\
			&\le 
			(L_RT)^{2p}
			\E\left[
			\max_{n=0,1,\ldots,N}
			\left(
			\norm{
				e^n
			}_1
			\right)^{2p}
			\right]
			.
			\end{align*}
			For the fifth term we use H\"older's inequality and \cite[Lemma 5.4]{MR3166967} (temporal regularity of the mild solution), which yields
			\begin{align*}
			\E\left[\max_{n=0,1,\ldots,N}(J^n_5)^{2p}\right]
			&\le \E\left[
			\max_{n=0,1,\ldots,N}
			\left(
			\sum_{l=0}^{n-1} 
			\int_{t_l}^{t_{l+1}}
			\norm{U(t_n,s)\left(F_R(X(t_l))-F_R(X(s))\right)}_1
			\,\text{d}s 
			\right)^{2p}
			\right]\\
			&\le
			\E\left[
			\max_{n=0,1,\ldots,N}
			\left(
			L_Rh
			\sum_{l=0}^{n-1} 
			\sup_{t_l\le s\le t_{l+1}}
			\norm{
				X(t_l)-X(s)
			}_1
			\right)^{2p}
			\right]\\
			&\le 
			(L_Rh)^{2p}
			\E\left[
			\left(
			\left(
			\sum_{l=0}^{N-1}
			1^{\frac{2p}{2p-1}} 
			\right)^{\frac{2p-1}{2p}}
			\left(
			\sum_{l=0}^{N-1} 
			\sup_{t_l\le s\le t_{l+1}}
			\norm{
				X(t_l)-X(s)
			}_1^{2p}
			\right)^{\frac{1}{2p}}
			\right)^{2p}
			\right]\\
			&=
			(L_Rh)^{2p}
			N^{2p-1}
			\sum_{l=0}^{N-1}
			\E\left[
			\sup_{t_l\le s\le t_{l+1}}
			\norm{
				X(t_l)-X(s)
			}_1^{2p}
			\right]\\
			&\le
			(L_Rh)^{2p}
			N^{2p}
			C(\norm{X_0}_2,T,p,R,\gamma)
			h^p
			= C_5(\norm{X_0}_2,T,L_R,p,\gamma)h^p.
			\end{align*}
			%	\coco{I think that one has $\norm{X_0}_6$??}
			%	No, it's Lemma 5.4 which we're using. Not Lemma 5.3. Otherwise you would be completely correct.
			All together, with another use of \cite[Proposition 2.2]{MR3166967} for bounding the term $I_1^n$, 
			we thus obtain
			\begin{align*}
			\E\left[\max_{n=0,1,\ldots,N} \norm{e^n}_1^{2p}\right]
			&\le C \E\left[\max_{n=0,1,\ldots,N} \norm{I_1^n}_1^{2p}\right] 
			+ (L_RT)^{2p}\E\left[\max_{n=0,1,\ldots,N} \norm{e^n}_1^{2p}\right] \\
			&\quad+ C_1h^p 
			+ C_2h^{p}
			+ C_3h^{2p}
			+ C_5h^{p}\\
			&\le Ch^p 
			%+ C(Th)^{2p} 
			+ (L_RT)^{2p} \E\left[\max_{n=0,1,\ldots,N} \norm{e^n}_1^{2p}\right].
			\end{align*}
			Now, as in the proof of \cite[Theorem 2]{bcd20}, for $T=T_1$ small enough, i.\,e. such that $(L_RT)^{2p} < 1$, the inequality above gives
			$$\E\left[\max_{n=0,1,\ldots,N} \norm{e^n}_1^{2p}\right] 
			\le \frac{C
				%(1+T_1^{2p})
			}{1-(L_RT)^{2p}}h^p,$$
			on $[0,T_1]$. In order to iterate this procedure, we impose, if necessary, that $h$ is small enough (or, equivalently, that $N$ is big enough), to ensure that $T_1$ can be chosen as before and as some integer multiple of $h$ (say $T_1 = rh$ for some positive integer $r$), while $T$ is some multiple integer of $T_1$ (say $KT_1 = T$ for some positive integer $K$). To obtain a bound for the error on the longer time interval $[0,T]$, we iterate the procedure above by choosing $T_2 = 2T_1$ and estimate the error on the interval $[T_1,T_2]$. We repeat this procedure, $K$ times, up to the final time $T$. This can be done since the above error estimates are uniform on the intervals $[T_k,T_{k+1}]$ for $k=0,\ldots,K-1$ (with a slight abuse of notation for the time interval):
			$$\E\left[\max_{[T_k,T_{k+1}]}\norm{X^n-X_k(t_n)}^{2p}_1\right] \le C_Eh^p$$
			where $C_E$ is the error constant obtained above, $t_n = nh$ are discrete times in $[T_k,T_{k+1}]$, $X_0(t):=X(t)$ is the exact solution with initial value $X_0$, $X_k(t)$ denotes the exact solution with initial value $\widehat{X}^k$ at time $T_k = kT_1 =(kr)h = t_{kr}$, and $\widehat{X}^k = X_{kr}$ corresponds to numerical solutions at time $T_k$ for $k=0,\ldots,K-1$. For the total error, we thus obtain (details are only written for the first two intervals) 
			\begin{align*}
			\E\left[\max_{n=0,1,\ldots,N} \norm{e^n}_1^{2p}\right]
			=& \E\left[\max_{[0,T]} \norm{X^n - X(t_n)}_1^{2p}\right]
			\leq C 
			\E\left[\max_{[0,T_1]} \norm{X^n - X(t_n)}_1^{2p}\right] \\
			+& C\E\left[\max_{[T_1,T_2]} \norm{X^n - X(t_n)}_1^{2p}\right]
			+\ldots+C
			\E\left[\max_{[T_{K-1},T_K]} \norm{X^n - X(t_n)}_1^{2p}\right]\\
			\leq& C_Eh^p + C_Eh^p + C_L\E\left[\norm{\widehat{X}^1-X(T_1)}_1^{2p}\right]
			+ \ldots
			\leq C_Eh^p + C_LC_Eh^p + \ldots\\
			\leq& C_Eh^p + C_LC_Eh^p + C_L^2C_Eh^p + \cdots + C_L^{K-1}C_Eh^p
			\leq Ch^p,
			\end{align*}
			where $C_L$ is the Lipschitz constant of the exact flow of \eqref{eq:Manakov} from $\mathbb{H}^1$ to itself and the last constant $C$ is independent of $N$ and $h$ with $Nh = T$ for $N$ big enough. This concludes the proof of the theorem.
		\end{proof}
		
		%%%%%%%%%%%%%%%%%%%%%%%%%%%%%%%%%%%%%%%%%%%%%%%%%%%%%%%%%%%%%%%%%%%%%%%%%%%%%%%%%%%%%%
		
		\section{Convergence in probability and almost surely in the non-Lipschitz case}\label{sec:convprob}
		Using the same strategy as in \cite{bbd15,bcd20}, one can show convergence in probability of order $1/2$ 
		and almost sure convergence of order $1/2^-$ for the Lie--Trotter splitting scheme \eqref{ltsplit} when applied to the stochastic Manakov equation \eqref{eq:Manakov} with the original cubic nonlinearity.
		
		\begin{proposition}\label{propProba}
			Let $X^0\in\H^6$ and $T>0$.
			Denote by $\tau^*=\tau^*(X_0,\omega)$ the maximum stopping time for the existence
			of a strong adapted solution,
			denoted by $X(t)$, of the stochastic Manakov equation \eqref{eq:Manakov}.
			For all stopping time $\tau<\tau^*\wedge T$ a.s. there exists $h_0>0$ such that we have 
			\begin{equation*}
			\forall h\in (0,h_0),\qquad
			\lim_{C\to\infty}
			\mathbb{P}\left( \max_{0\leq n\leq N_{\tau}} \norm{X^n-X(t_n)}_{\H^1}\geq C h^{1/2} \right)=0, 
			\end{equation*}
			where $X^n$ denotes the numerical solution given by the Lie--Trotter splitting scheme \eqref{ltsplit} with time step $h$ and $N_\tau=\lceil\frac{\tau}{h}\rfloor$. 
		\end{proposition}
		
		\begin{proof}
			For $R>0$, let us denote by $X_R$, resp. $X_R^n$, the exact, resp. numerical, solutions to the stochastic Manakov equation 
			\eqref{mildManakov} with a truncated nonlinearity $F_R$.
			
			Fix $X^0\in\H^6$, $T>0$, $\varepsilon\in(0,1)$.
			Let $\tau$ be a stopping time such that a.s. $\tau<\tau^*\wedge T$.
			By \cite[Theorem~1.2]{MR3024974} there exists an $R_0>{\color{blue}1}$ such that 
			$\displaystyle\mathbb{P}\left(\sup_{t\in[0,\tau]}\norm{X(t)}_1\geq R_0-1\right)\leq\varepsilon/2$.
			Observe that one has the inclusion 
			\begin{align*}
			\left\{ \max_{0\leq n\leq N_\tau}\norm{X^n-X(t_n)}_1\geq\varepsilon  \right\}&\subset \left\{ \max_{0\leq n\leq N_\tau}\norm{X(t_n)}_1\geq R_0-1 \right\} \\
			&\cup \left( \left\{ \max_{0\leq n\leq N_\tau}\norm{X^n-X(t_n)}_1\geq\varepsilon \right\} \cap \left\{ \max_{0\leq n\leq N_\tau}\norm{X(t_n)}_1<R_0-1 \right\} \right).
			\end{align*}
			Taking probabilities, we obtain
			\begin{align*}
			\PP&\left( \left\{ \max_{0\leq n\leq N_\tau}\norm{X^n-X(t_n)}_1\geq\varepsilon  \right\} \right)\\
			&\leq\varepsilon/2+
			\PP\left( \left\{ \max_{0\leq n\leq N_\tau}\norm{X^n-X(t_n)}_1\geq\varepsilon \right\} \cap \left\{ \max_{0\leq n\leq N_\tau}\norm{X(t_n)}_1<R_0-1 \right\} \right).
			\end{align*}
			In order to estimate the terms on the right-hand side, we define the random variable
			$n_\varepsilon:=\min\{n\in\{0,\dots,N_\tau\}\colon \norm{X^n-X(t_n)}_1\geq\varepsilon\}$,
			with the convention that $n_\varepsilon=N_\tau+1$ if the set is empty.
			If $\displaystyle\max_{0\leq n\leq N_\tau}\norm{X(t_n)}_1<R_0-1$ then we have by the triangle inequality 
			$$
			\max_{0\leq n\leq n_\varepsilon-1}\norm{X^n}_1
			=\max_{0\leq n\leq n_\varepsilon-1}\norm{X^n-X(t_n)+X(t_n)}_1\leq\varepsilon+R_0-1\leq R_0.
			$$
			From the definition of the Lie--Trotter splitting scheme \eqref{ltsplit}, the isometry property of $U_{h,n+1}$, the fact that $F$ below is the cubic nonlinearity, 
			and Lemma~\ref{lemma:bounds} follows 
			\begin{align}\label{lalaland}
			\norm{X^{n_\varepsilon}}_1
			&= \norm{Y^{n_\varepsilon-1}(t_{n_\varepsilon})}_1
			\le \norm{Y^{n_\varepsilon-1}(t_{n_\varepsilon-1})}_1 + \int_{t_{n_\varepsilon-1}}^{t_{n_\varepsilon}}\norm{F(Y^{n_\varepsilon-1}(s))}_1\,\dd s\nonumber\\
			&\le \norm{X^{n_\varepsilon-1}}_1 + C\int_{t_{n_\varepsilon-1}}^{t_{n_\varepsilon}}\norm{Y^{n_\varepsilon-1}(s)}_1^3\,\dd s\quad\text{a.s}.
			\end{align}
			Next, we give an explicit bound on $Y^{n_\varepsilon-1}(s)$
			for $t_{n_\varepsilon-1}\leq s\leq t_{n_\varepsilon}$, in order to
			show that $X^{n_\varepsilon}$ is also explicitly bounded. 
			To do so, we use the explicit form of $Y^{n_\varepsilon-1}$,
			as solution to the ODE system $i\text dY+|Y|^2Y\text dt=0$.
			First, we recall that $|Y_1(t,x)|^2+|Y_2(t,x)|^2$ is pointwise preserved by the flow, see the proof of Lemma~\ref{thm:LTPres}. This implies in particular that, for all $s\in[t_{n_\varepsilon-1},t_{n_\varepsilon}]$,
			$$
			\norm{Y^{n_\varepsilon-1}(s)}_{\L^2}^2=\norm{Y^{n_\varepsilon-1}(t_{n_\varepsilon-1})}_{\L^2}^2\quad{a.s}.
			$$
			Moreover, this preservation property allows for the exact solution of the ODE on
			$(t_{n_\varepsilon-1},t_{n_\varepsilon})$ to be written for all $s\in[t_{n_\varepsilon-1},t_{n_\varepsilon}]$ as
			$$
			Y^{n_\varepsilon-1}(s)=e^{i(s-t_{n_\varepsilon-1})|Y^{n_\varepsilon-1}(t_{n_\varepsilon-1})|^2}\begin{pmatrix}Y^{n_\varepsilon-1}_1(t_{n_\varepsilon-1})\\Y^{n_\varepsilon-1}_2(t_{n_\varepsilon-1})\end{pmatrix}. 
			$$
			Then, the first spatial derivative of $Y^{n_\varepsilon-1}$ can be computed as follows 
			\begin{align*}
			\partial_xY^{n_\varepsilon-1}(s)&=e^{i(s-t_{n_\varepsilon-1})|Y^{n_\varepsilon-1}(t_{n_\varepsilon-1})|^2}\left\{ \begin{pmatrix} \partial_x Y^{n_\varepsilon-1}_1(t_{n_\varepsilon-1}) \\ \partial_x Y^{n_\varepsilon-1}_2(t_{n_\varepsilon-1}) \end{pmatrix} \right.\\
			&\quad\left.+2i(s-t_{n_\varepsilon-1})\text{Re}\left( \overline{Y^{n_\varepsilon-1}_1}(t_{n_\varepsilon-1})\partial_xY^{n_\varepsilon-1}_1(t_{n_\varepsilon-1})+\overline{Y^{n_\varepsilon-1}_2}(t_{n_\varepsilon-1})\partial_xY^{n_\varepsilon-1}_2(t_{n_\varepsilon-1})\right)
			\begin{pmatrix} Y^{n_\varepsilon-1}_1(t_{n_\varepsilon-1}) \\ Y^{n_\varepsilon-1}_2(t_{n_\varepsilon-1}) \end{pmatrix} 
			\right\}.
			\end{align*}
			Using Cauchy--Schwarz inequality and the fact that $H^1$ is an algebra, 
			one obtains that, for some constant $C>0$, 
			$$
% 			\norm{\partial_xY^{n_\varepsilon-1}(s)}^2_{\L^2}\leq \norm{\partial_xY^{n_\varepsilon-1}(t_{n_\varepsilon-1})}_{\L^2}^2+C(s-t_{n_{\varepsilon-1}})\norm{Y^{n_\varepsilon-1}(t_{n_\varepsilon-1})}_{1}^4\left(1+\norm{Y^{n_\varepsilon-1}(t_{n_\varepsilon-1})}_{1}^2\right).
		\norm{\partial_xY^{n_\varepsilon-1}(s)}^2_{\L^2}\leq 2\norm{\partial_xY^{n_\varepsilon-1}(t_{n_\varepsilon-1})}_{\L^2}^2
		\left(1+C^2(s-t_{n_{\varepsilon-1}})^2\norm{Y^{n_\varepsilon-1}(t_{n_\varepsilon-1})}_{1}^4\right).
		$$
        Using the above estimates and the definition of the $\H^1$ norm, one arrives at the following bound
        $$  
        \norm{Y^{n_\varepsilon-1}(s)}^2_{1}\leq 2\norm{Y^{n_\varepsilon-1}(t_{n_\varepsilon-1})}_{1}^2
		\left(1+C^2(s-t_{n_{\varepsilon-1}})^2\norm{Y^{n_\varepsilon-1}(t_{n_\varepsilon-1})}_{1}^4\right).  
        $$
        Taking the square root of the above and using the fact that $\sqrt{a^2+b^2}\leq a+b$ for positive real numbers $a,b$, one gets
			$$
        \norm{Y^{n_\varepsilon-1}(s)}_1\leq \sqrt{2}\norm{Y^{n_\varepsilon-1}(t_{n_\varepsilon-1})}_{1}\left(1+C|s-t_{n_\varepsilon-1}|\norm{Y^{n_\varepsilon-1}(t_{n_\varepsilon-1})}^2_{1}\right).
			$$
			Now, as $Y^{n_\varepsilon-1}(t_{n_\varepsilon-1})=X^{n_\varepsilon-1}$ is bounded in $\H^1$ by $R_0$,
			if we assume that $h$ is small enough to ensure that $C h\left( \sqrt{2}R_0(1+ChR_0^2)\right)^3\leq R_0$, then
			we have $X^n=X_{2R_0}^n$ for $0\leq n\leq n_\varepsilon$ by \eqref{lalaland}.
			
			If $n_\varepsilon\leq N_\tau$, then $\norm{X^{n_\varepsilon}_{2R_0}-X_{2R_0}(t_{n_\varepsilon})}_1\geq\varepsilon$ thanks to the definition of $n_\varepsilon$. 
			Therefore we get $\displaystyle\max_{0\leq n\leq N_\tau}\norm{X^n_{2R_0}-X_{2R_0}(t_n)}_1\geq\varepsilon$. 
			Furthermore, by definition of $n_\varepsilon$, we have 
			$$
			%\displaystyle
			\left\{ \max_{0\leq n\leq N_\tau}\norm{X^n-X(t_n)}_1\geq\varepsilon \right\}\cap\left\{ n_\varepsilon>N_\tau \right\}=\varnothing.$$
			We then deduce that 
			$$
			\left\{ \max_{0\leq n\leq N_\tau}\norm{X^n-X(t_n)}_1\geq\varepsilon \right\}=\left\{ \max_{0\leq n\leq N_\tau}\norm{X^n-X(t_n)}_1\geq\varepsilon \right\}
			\cap\left\{ n_\varepsilon\leq N_\tau \right\}.
			$$
			Combining the above, using Markov's inequality as well as the strong error estimates from Theorem~\ref{thm:LTSplConvStrong}, since $\tau<T$ a.s., there exists $C>0$ such that
			\begin{align*}
			&\PP\left( \max_{0\leq n\leq N_\tau}\norm{X^n-X(t_n)}_1\geq\varepsilon, n_\varepsilon\leq N_\tau,\max_{0\leq n\leq N_\tau}\norm{X(t_n)}_1<R_0-1 \right)\\
			&\leq\PP\left( \max_{0\leq n\leq N_\tau}\norm{X_{2R_0}^n-X_{2R_0}(t_n)}_1\geq\varepsilon \right)
			\leq\frac1{\varepsilon^{2p}}\E\left[ \max_{0\leq n\leq N_\tau}\norm{X^n_{2R_0}-X_{2R_0}(t_n)}_1^{2p} \right]\leq\frac1{\varepsilon^{2p}}Ch^p.
			\end{align*}
			This last term is smaller than $\varepsilon/2$ for $h$ small enough. 
			All together we obtain
			$$
			\PP\left( \max_{0\leq n\leq N_\tau}\norm{X^n-X(t_n)}_1\geq\varepsilon \right)\leq\frac\varepsilon2+\frac\varepsilon2=\varepsilon,
			$$
			and thus convergence in probability. 
			
			To get the order of convergence in probability,
			we choose $R_1\geq R_0-1$ such that for all $h>0$ small enough, 
			$\displaystyle\PP\left( \max_{0\leq n\leq N_\tau}\norm{X^n}_1\geq R_1\right)\leq\frac\varepsilon2$.
			As above, for all positive real number $C$, we have 
			\begin{align*}
			\left\{ \max_{0\leq n\leq N_\tau}\norm{X^n-X(t_n)}_1\geq Ch^{1/2} \right\}&\subset \left\{\max_{0\leq n\leq N_\tau}\norm{X(t_n)}_1\geq R_1 \right\}\\
			&\cup\bigg( \left\{ \max_{0\leq n\leq N_\tau}\norm{X^n-X(t_n)}_1\geq Ch^{1/2} \right\}\\
			&\cap \left\{\max_{0\leq n\leq N_\tau}\norm{X(t_n)}_1<R_1\right\} \bigg).
			\end{align*}
			Taking probabilities and using Markov's inequality as well as the strong error estimate from Theorem~\ref{thm:LTSplConvStrong}, we obtain  
			\begin{align*}
			\PP\left( \left\{ \max_{0\leq n\leq N_\tau}\norm{X^n-X(t_n)}_1\geq Ch^{1/2} \right\} \right)&\leq\frac\varepsilon2+
			\PP\left( \left\{ \max_{0\leq n\leq N_\tau}\norm{X_{4R_1}^n-X_{4R_1}(t_n)}_1\geq Ch^{1/2} \right\} \right)\\
			&\leq\frac\varepsilon2+\frac{K(\norm{X_0}_6,T,4R_1,p,\gamma)}{C^{2p}},
			\end{align*}
			since $\tau\leq T$ almost surely.
			For $C$ large enough, we infer
			$$
			\PP\left( \left\{ \max_{0\leq n\leq N_\tau}\norm{X^n-X(t_n)}_1\geq Ch^{1/2} \right\} \right)\leq\frac\varepsilon2+\frac\varepsilon2=\varepsilon,
			$$
			uniformly for $h<h_0$.
			This shows that the order of convergence in probability of the Lie--Trotter splitting scheme is $1/2$.
		\end{proof}
		Using the results above, one arrives at the following proposition, which establishes
		that the Lie--Trotter splitting scheme has almost sure convergence order $1/2^-$.
		\begin{proposition}\label{prop-as}
			Under the assumptions of Proposition~\ref{propProba}, for all $\delta\in\left(0,\frac12\right)$
			and $T>0$,
			there exists a random variable $K_\delta(T)$ such that for all stopping times $\tau$ with $\tau<\tau^*\wedge T$, we have
			$$
			\max_{n=0,\ldots, N_\tau}\norm{X^n(\omega)-X(t_n,\omega)}_{\H^1}\leq K_\delta(T,\omega)h^\delta\quad\quad\quad\mathbb{P}-a.s.,
			$$
			for $h>0$ small enough.
		\end{proposition}
		\begin{proof}
			The proof uses similar arguments as the corresponding proof in \cite{bcd20}. 
			Let $\tau$ be a stopping time such that $\tau<\tau^*\wedge T$ almost surely.
			Fix $R>0$ and $p>1$.
			Using the strong error estimate from Theorem~\ref{thm:LTSplConvStrong} and Markov's inequality, one gets
			positive $h_0$ and $C$, which depend on $T$ but not on $\tau$ itself, such that
			$$
			\forall h\in (0,h_0),\qquad
			\PP\left(\max_{0\leq n\leq N_\tau}\norm{X_R^n-X_R(t_n)}_1>h^\delta\right)\leq Ch^{p(1-2\delta)}.
			$$
			Using \cite[Lemma~2.8]{MR1873517}, one then obtains that, choosing $p\geq1$ sufficiently
			large to ensure that $p(1-2\delta)>1$,
			there exists a positive random variable $K_\delta(R,\gamma,T,p,\cdot)$ such that
			\begin{equation}\label{asest}
			\mathbb{P}-a.s., \quad \forall h\in(0,h_0),\qquad
			\max_{0\leq n\leq N_\tau}\norm{X_R^n-X_R(t_n)}_1 \leq K_\delta(R,\gamma,T,p,\omega)h^\delta.
			\end{equation}
			After this preliminary observation, we shall proceed as in the proof of Proposition~\ref{propProba}. 
			We know that, since $\tau<\tau^*$ a.s., there exists a random variable $R_0$ such that
			$$
			\sup_{0\leq t\leq\tau}\norm{X(t)}_1\leq R_0(\omega)\quad\mathbb{P}-\text{a.s.}.
			$$
			Let now $\varepsilon\in(0,1)$ and $h$ small enough ($h\leq 3R_0^{-2}(\omega)$). Assume by contradiction that 
			$$
			\max_{0\leq n\leq N_\tau}\norm{X^n-X(t_n)}_1\geq\varepsilon.
			$$
			Define $n_\varepsilon:=\min\{n\colon\norm{X^n-X(t_n)}_1\geq\varepsilon\}$. By definition of $R_0$ and $h$, we have that $\norm{X^n}_1\leq R_0$ a.s. for $0\leq n<n_\varepsilon-1$. 
			Hence, $\norm{X^{n_\varepsilon}}_1\leq 2R_0$ and so the numerical solution equals to the numerical solution of the truncated equation 
			$X^n=X_{2R_0}^n$ for $n=0,1,\ldots,n_\varepsilon$. 
			We thus obtain that $\displaystyle\max_{0\leq n\leq N_\tau}\norm{X^n_{2R_0}-X_{2R_0}(t_n)}_1\geq\varepsilon$ for $h$ small enough. 
			This contradicts \eqref{asest} with $R=2R_0$. Therefore, we have almost sure convergence. 
			
			To get the order of almost sure convergence, we proceed similarly as in the proof of Proposition~\ref{propProba}. 
			From the above, we have for $\omega$ in a set of probability one and all
			$\varepsilon>0$, there exists $h_0>0$ such that for all $h\leq h_0$, 
			$\displaystyle\max_{0\leq n\leq N_\tau}\norm{X^n-X(t_n)}_1\leq\varepsilon$. Thus, there exists $R_1(\omega)>R_0(\omega)$ such that 
			$\displaystyle\max_{0\leq n\leq N_\tau}\norm{X^n}_1\leq R_1(\omega)$. 
			
			If now $h\leq 3R_1^{-2}:=h_0$, we obtain from \eqref{asest} that
			\begin{align*}
			\max_{0\leq n\leq N_\tau}\norm{X^n-X(t_n)}_1=\max_{0\leq n\leq N_\tau}\norm{X^n_{R_1}-X_{R_1}(t_n)}_1\leq K_\delta(R_1,\gamma,T,p,\omega)h^\delta.
			\end{align*}
			%On the other hand if $3R_0^{-2}\geq h>3R_1^{-2}$, we get 
			%\begin{align*}
			%\max_{0\leq n\leq N_\tau}\norm{X^n-X(t_n)}_1\leq 2R_1\leq\frac{2R_1}{h_0^\delta}h^\delta.
			%\end{align*}
			This shows that the order of a.s. convergence of the exponential integrator is $\frac12-$.
		\end{proof}
		
		%%%%%%%%%%%%%%%%%%%%%%%%%%%%%%%%%%%%%%%%%%%%%%%%%%%%%%%%%%%%%%%%%%%%%%%%%%%%%%%
		
		\section{Numerical experiments}\label{sec-numexp}
		% Section introduction elements (in no particular order):
		% Introduce modified equation (factor 1/2) -
		% Introduce numerical schemes -
		% Introduce soliton initial value - 
		% Introduce soliton -
		% Mention change to H and U with factor 1/2 -
		% Introduce standard parameters for experiments -
		In this section, we present various numerical experiments in order
                to illustrate the main properties of the Lie--Trotter scheme \eqref{ltsplit},
                denoted by \textsc{LT} below, and to compare it with other time integrators from the literature  
                for the stochastic Manakov system. 
                We start by numerically illustrating 
		various types of convergence (strong, in probability, and almost-suerly) of various time integrators. 
		Then, we focus on the preservation of the $\L^2$-norm and computational costs. 
		Finally, we consider stochastic evolution of deterministric solitons and we discuss the possible 
		occurrence of blowup of solutions to the stochastic Manakov equation.
		
		Let us note that, in order to be able to consider soliton solutions (see below for precise details), we slightly modify 
		the SPDE \eqref{eq:Manakov} with a factor $1/2$ in front of the second order spatial derivative: 
		\begin{equation}
		\label{eq:halfManakov}
		i\text{d}X
		+ \frac{1}{2}\partial^2_x X \,\text{d}t
		+ i \sqrt{\gamma} \sum_{k=1}^{3} \sigma_k \partial_x X \circ \text{d} W_k
		+ |X|^{2} X \,\text{d}t= 0.
		\end{equation}
		We also introduce the following numerical schemes in order to compare their performance:
		\begin{itemize}
			\item The nonlinearly implicit Crank--Nicolson scheme from \cite{MR3166967}
			\begin{equation}
			\label{CN}\tag{CN}
			X^{n+1}=X^n-H_{h,n}X^{n+1/2}+ihG(X^n,X^{n+1}),
			\end{equation}
			where $G(X^n,X^{n+1})=\frac12\left(|X^n|^2+|X^{n+1}|^2\right)X^{n+1/2}$ and $X^{n+1/2}=\frac{X^{n+1}+X^n}2$.
			\item The exponential integrator studied in \cite{bcd20}
			\begin{equation}
			\label{EXP}\tag{EXP}
			X^{n+1}=U_{h,n}\left(X^n+ihF(X^n)\right).
			\end{equation}
			\item The relaxation scheme presented in \cite{MR3166967}
			\begin{equation}
			\label{relax}\tag{Relax}
			i\left(X^{n+1}-X^n\right)+H_{h,n}\left(\frac{X^{n+1}+X^n}{2} \right)+\Phi^{n+1/2}\left(\frac{X^{n+1}+X^n}{2} \right)=0,
			\end{equation}
			where $\Phi^{n+1/2}=2|X^n|^2-\Phi^{n-1/2}$ with $\Phi^{-1/2}=|X_0|^2$.
		\end{itemize}
		Observe that the definitions of the quantities $H_{h,n}$ and $U_{h,n}$ have to be modified in order to take into account the factor $1/2$ in equation \eqref{eq:halfManakov}.
		
		Equation \eqref{eq:halfManakov} allows the formation of solitons in the deterministic case (i.\,e. for $\gamma=0$). 
		Such solutions are given by the initial value \cite{Gazeau:13,HASEGAWA2004241}
		\begin{equation}
		\label{eq:hasegawa}
		X_0(\eta,\kappa,\alpha,\tau,\theta,\phi_1,\phi_2) =
		X(x,0) =
		\begin{pmatrix}
		\cos(\theta/2)\exp({i\phi_1})\\
		\sin(\theta/2)\exp({i\phi_2})
		\end{pmatrix}
		\eta\sech(\eta x)\exp({-i\kappa(x-\tau)+i\alpha}),
		\end{equation} 
		where $\eta,\kappa,\alpha,\tau,\theta,\phi_1,\phi_2\in\R$,
		and take the form
		\begin{equation}
		\label{eq:soliton}
		X(x,t) =
		\begin{pmatrix}
		\cos(\theta/2)\exp({i\phi_1})\\
		\sin(\theta/2)\exp({i\phi_2})
		\end{pmatrix}
		\eta
		\sech(\eta (x-\tau(t)))
		\exp\left({-i\kappa(x-\tau(t))+
			i\alpha(t)}
		\right),
		\end{equation}
		where
		$\tau(t) = \tau_0 - \kappa t$, 
		$\alpha(t) = \alpha_0 + \frac{1}{2}\left(\eta^2 + \kappa^2\right)t$, and
		$\tau_0,\alpha_0\in\R$.

		For the numerical experiments we follow a few standards that will hold unless stated otherwise.
		We consider the SPDE 
		\eqref{eq:halfManakov} with 
		$\gamma=1$
		on a bounded interval $[-a,a]$ with a sufficiently large $a>0$, with homogeneous Dirichlet boundary conditions,
		and over the time interval $[0,T]$, $T>0$. 
		The spatial discretisation is done by uniform finite differences with mesh size denoted by $\Delta x$. 
		The initial condition for the SPDE is given by equation \eqref{eq:hasegawa} with the parameters $\alpha=\tau=\phi_1=\phi_2=\kappa=0$ and $\theta=\pi/4$, $\eta=1$.
		Lastly, all experiments will per sample, whenever possible, use a common Brownian motion for each numerical scheme and time discretization.

		%		For the numerical experiments we will, unless stated otherwise, consider \eqref{eq:halfManakov} with 
		%		$\gamma=1$
		%		on an interval $[-a,a]$ with a sufficiently large $a>0$ with homogeneous Dirichlet boundary conditions. 
		%		The spatial discretisation is done by centered finite differences with mesh size denoted by $\Delta x$. 
		%		The initial condition for the SPDE is given by equation \eqref{eq:hasegawa} with the parameters $\alpha=\tau=\phi_1=\phi_2=\kappa=0$ and $\theta=\pi/4$, $\eta=1$.
		%		Lastly, all experiments will per sample, whenever possible, use a common Brownian motion for each numerical scheme and time discretization.
		
		%
		%% LTSpl scheme - \eqref{ltsplit} shown convergences:
		% \ref{thm:LTSplConvStrong} - Strong order 1/2 - Cut-off nonlin
		%
		%% CN scheme - \cite{MR3166967} shown convergences:
		% \cite[Theorem 1.3]{MR3166967} - Conv in prob 1/2
		% \cite[Proposition 3.4]{MR3166967} - Strong order 1/2 - Cut-off nonlin
		%  
		%% Exponential integrator - \cite{bcd20} shown converges:
		% \cite[Theorem 2]{bcd20} - Strong order 1/2 - Cut-off nonlin
		% \cite[Proposition 3]{bcd20} - Conv in prob 1/2
		% \cite[Proposition 4]{bcd20} - AS conv 1/2-
		% 
		%
		%% Relaxation scheme - \cite{MR3166967} shown convergences:
		
		\subsection{Strong convergence}
		In this subsection we will numerically demonstrate the mean-square orders of convergence of the four numerical schemes seen above. 
		The strong order of convergence has been shown to be $1/2$, in the case of a cut-off nonlinearity, for 
		the Lie--Trotter splitting scheme, the exponential integrator, and 
		the Crank--Nicolson scheme in Theorem~\ref{thm:LTSplConvStrong} above, 
		\cite[Theorem 2]{bcd20}, and \cite[Proposition 3.4]{MR3166967}, respectively.
		
		To illustrate these strong orders of convergence, we consider discretizations of equation \eqref{eq:halfManakov} with the following parameters: 
		$a = 50$, $\Delta x = 0.05$, $T = 1$, 
		$N = 2^k$, $k=10,11,\ldots,16$, and
		$h = T/N$. 
		We approximate the exact solution using a reference solution,
		denoted by $\widetilde{X}$, simulated using the Lie--Trotter splitting scheme with 
		$N^\text{ref}=2^{18}$ and time step size 
		$h^\text{ref} = T/N^\text{ref}$. 
		We then compute the mean-square errors
		$$
		e_N := \mathbb{E}\left[\max_{n=1,\ldots,N}\norm{X^n - \widetilde{X}(t_n)}_{1}^2\right]
		$$
		at the coarse time grid $\left\{t_n\right\}_{n=1}^N = \left\{nh \right\}_{n=1}^N$. 
		The expectations are approximated using $300$ samples (we have checked that this number of samples was enough).
		The results are presented in Figure~\ref{fig:strongConv}, where it can be observed that  
		all four numerical schemes demonstrate the mean-square rate of convergence $1/2$.
		
		\begin{figure}[h!]
			\begin{center}
				\includegraphics*[width =.4\textwidth,keepaspectratio]{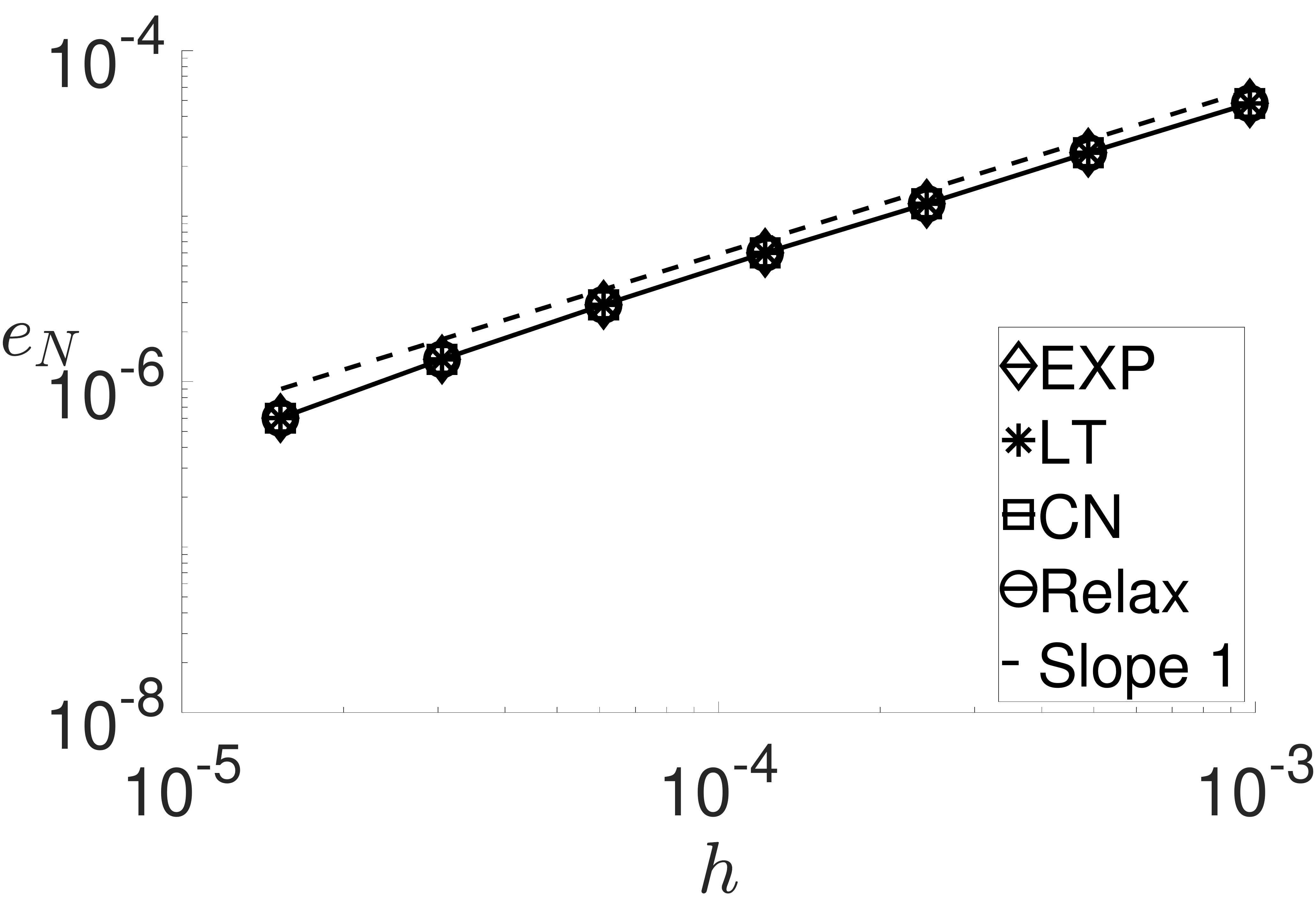}
				\caption{Strong rates of convergence.}
				\label{fig:strongConv}
			\end{center}
		\end{figure}
		
		\subsection{Convergence in probability}
		In this subsection we numerically demonstrate the orders 
		of convergence in probability for the four numerical schemes. 
		This order has been shown to be $1/2$ for 
		the exponential integrator and 
		the Crank--Nicolson scheme in 
		\cite[Proposition 3]{bcd20}, and 
		\cite[Theorem 1.3]{MR3166967}, respectively, and for the Lie--Trotter scheme in Proposition~\ref{propProba} above.
		
		Restating a definition from \cite{bbd15}, 
		we say that a numerical scheme converges in probability with order $\delta$, in $\H^1$, 
		if for all small enough $h>0$,
		$$
		\lim_{C\to\infty}
		\,
		\mathbb{P}\left(\sup_{n\in\{1,2,\ldots,N\}}\norm{X^n - X(t_n)}_{1} \ge C h^\delta\right) = 0,
		$$
		or, equivalently, if for all $\varepsilon \in(0,1]$, there exists
		$C(\varepsilon)>0$ such that
		$$
		\mathbb{P}\left(\sup_{n\in\{1,2,\ldots,N\}}\norm{X^n - X(t_n)}_{1} \ge C(\varepsilon)h^\delta\right) < \varepsilon.
		$$
		Numerically, we investigate the order in probability by using the equation
		\begin{equation}
		\label{eq:probConv}
		\max_{n\in\{1,2,\ldots,N\}} \norm{X^n  - \widetilde{X}(t_n)}_{1} \ge C h^\delta,
		\end{equation}
		where $\widetilde{X}$ denotes a reference solution.
		We then either study the proportion of samples, $P$, fulfilling equation~\eqref{eq:probConv} 
		for given $C$ and $\delta$, or estimate the constant $C$ for given $\delta$ and proportion of samples $P$. This is to say, when estimating $P$ for given $\delta$, $h$, and $C$,
		we then observe whether $P\to 0$ for the given $\delta$ as $h\to 0$ and $C$ increases. 
		Or, when estimating $C$ for given $\delta$, $h$, and the proportion $P$ allowed by the sample size, we then observe whether the range of obtained $C$ is small or not. 
%		In order to allow for comparison, we normalize this range such that $\max_{h}\tilde{C}=1$ for $P=0$, where $\tilde{C}$ denotes the renormalized range. 
		In order to allow for comparison, we normalize this range via
		$$\tilde{C}(\delta, h, P) = \frac{C(\delta, h, P)}{\max_{\hat h} C(\delta,\hat h, 0)},$$
		which forces $\max_{h}\tilde{C}(\delta, h, 0)=1$.
		
		In the first numerical experiment we simulate $36$ samples (on a $12$ core computer) using the Lie--Trotter splitting scheme using a pseudospectral spatial discretization using $2^{10}$ Fourier modes 
		and we take the parameters 
		$\gamma=9$, 
		$a=50$, 
		$T = 1/2$, 
		$N = 2^k$, $k = 12, 14, \ldots, 20$,
		and $h=T/N$. 
		We approximate the exact solution using a reference solution,
		$\widetilde{X}$, simulated using the Lie--Trotter splitting scheme with 
		$N^\text{ref} = 2^{24}$ and time step size 
		$h^\text{ref} = T/N^\text{ref}$.
		We then estimate the proportion $P$ 
		of samples fulfilling \eqref{eq:probConv} for each given $h$, $\delta= 0.4, 0.5, 0.6$, and 
		$C = 10^c$ for $c=0, .5, 1$. %$c=0.75, 1.4, 2$
		The results are presented in Figure~\ref{fig:LieProbConv}. 
		
		\begin{figure}[h!]
			\begin{center}
				\includegraphics*[width =.5\textwidth,keepaspectratio]{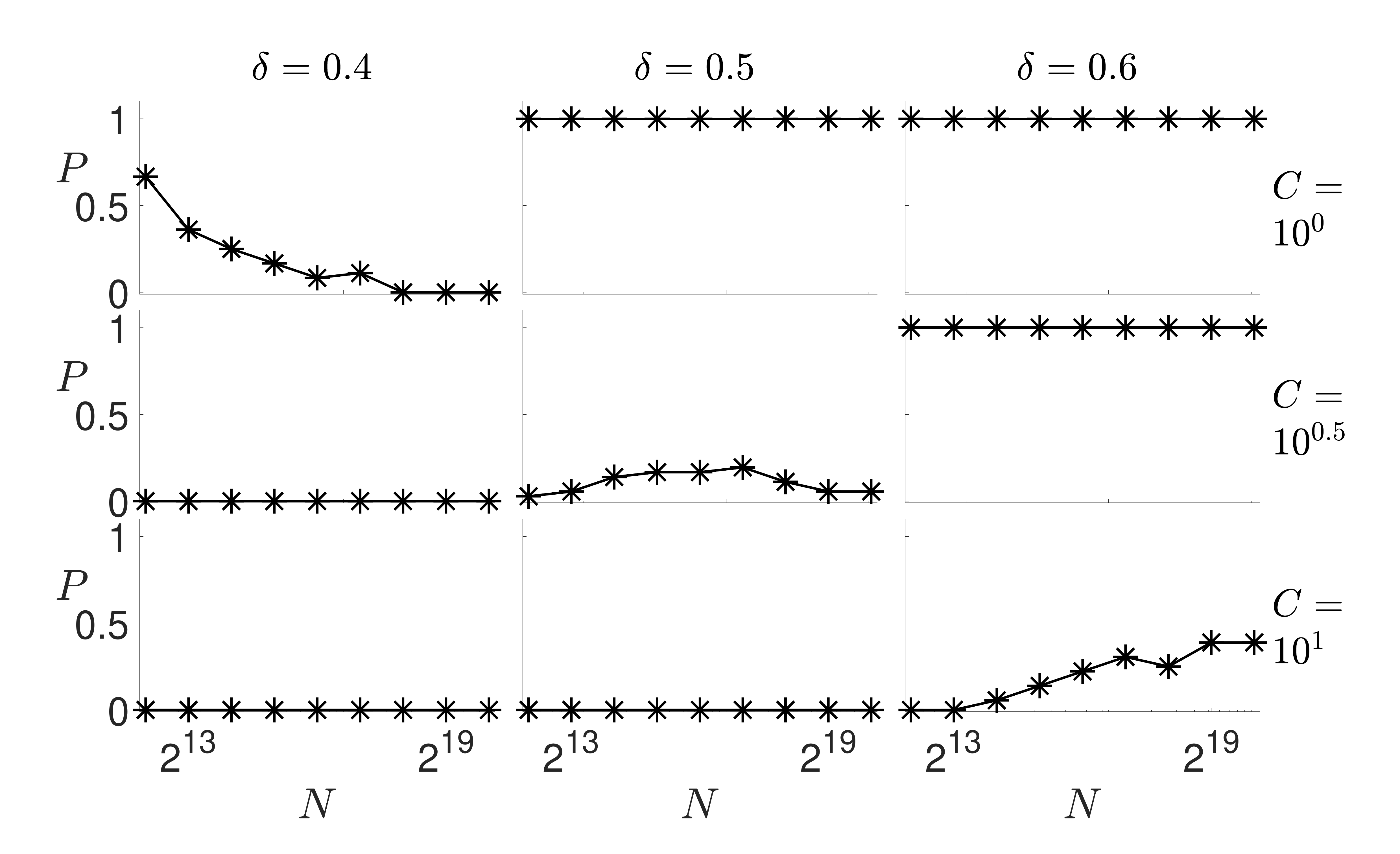}
				\caption{Proportion of samples fulfilling \eqref{eq:probConv} for the Lie--Trotter splitting scheme. 
					\label{fig:LieProbConv}
				}
			\end{center}
		\end{figure}
		
		In this figure, one clearly sees how the proportion of samples $P$ 
		quickly goes to zero for $\delta \le 1/2$ and an increasing $C$. Furthermore, this property 
		does not hold for $\delta > 1/2$. This numerical experiment thus confirms that the order 
		of convergence in probability of the Lie--Trotter scheme is $1/2$. Observe that the reason 
		for choosing a larger value of $\gamma$ in the model \eqref{eq:halfManakov} is for being able to perform 
		such computations in reasonable times with reasonable large values of $C$ and $N$. 
		We have performed similar experiments for the other time integrators and have obtained alike results.
		
		In the second numerical experiment we simulate $300$ samples using all four schemes and 
		considering the parameters  
		$a=50$, 
		$\Delta x = 0.05$, 
		$T = 1$, 
		$N = 2^k$, $k = 10, 11, \ldots, 16$,
		and $h=T/N$. 
		We approximate the exact solution using a reference solution,
		$\widetilde{X}$, simulated using the Lie--Trotter splitting scheme with 
		$N^\text{ref} = 2^{18}$ and time step size 
		$h^\text{ref} = T/N^\text{ref}$. 
		Using the obtained samples, we estimate the normalized ranges $\tilde{C}$ for the given $h$, 
		$\delta = 0.3, 0.4, 0.5$ 
		and $P\in(0,1]$. 
		The results are presented in Figure~\ref{fig:probConvCEst}. 
		
		\begin{figure}[h!]
			\centering
			\begin{subfigure}[t]{0.3\textwidth}
				\centering
				\includegraphics*[width =\textwidth,keepaspectratio,clip]{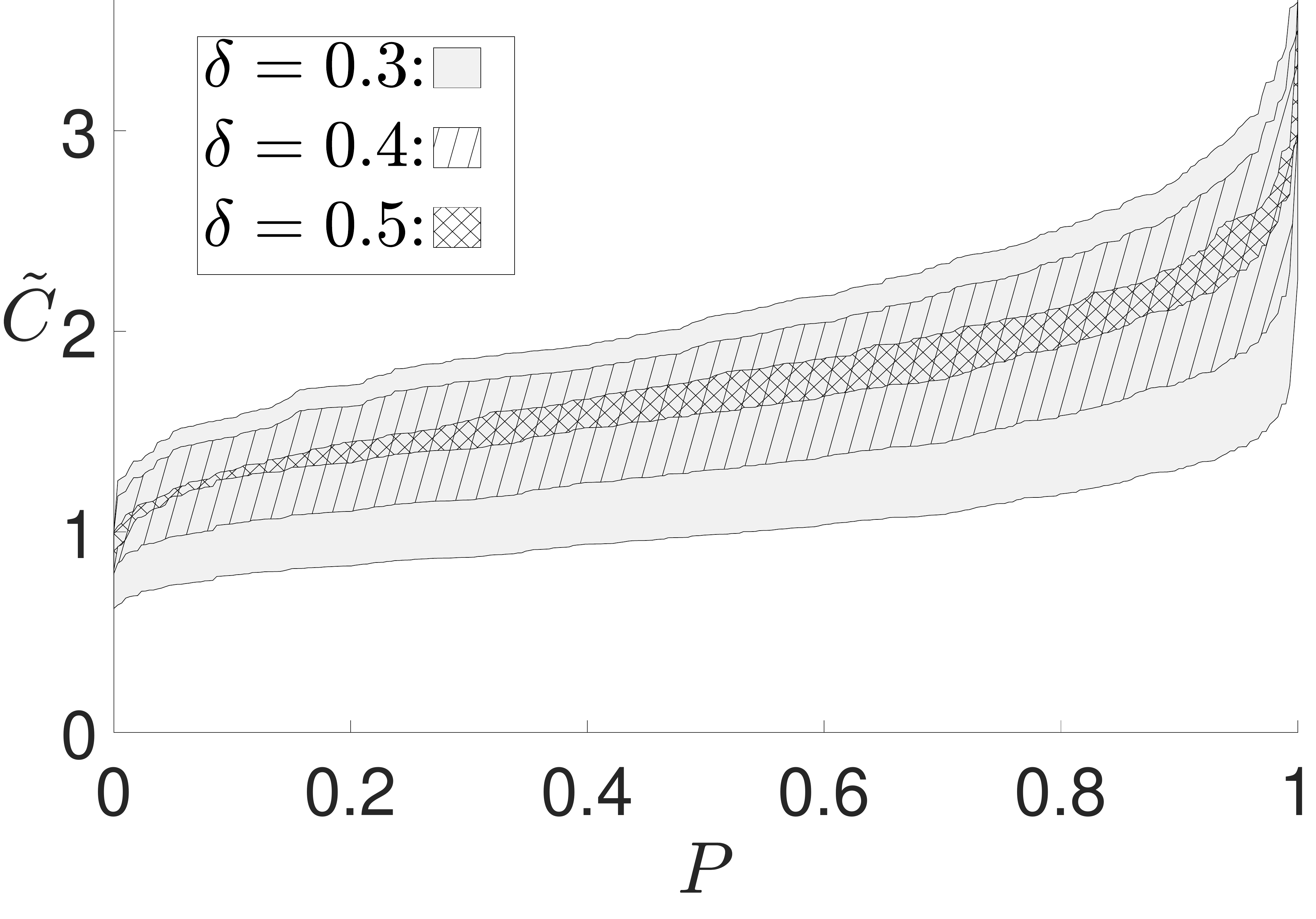}
				\caption{Lie--Trotter splitting scheme.}
				\label{fig:Ceps1}
			\end{subfigure}
			~ 
			\begin{subfigure}[t]{0.3\textwidth}
				\centering
				\includegraphics*[width =\textwidth,keepaspectratio,clip]{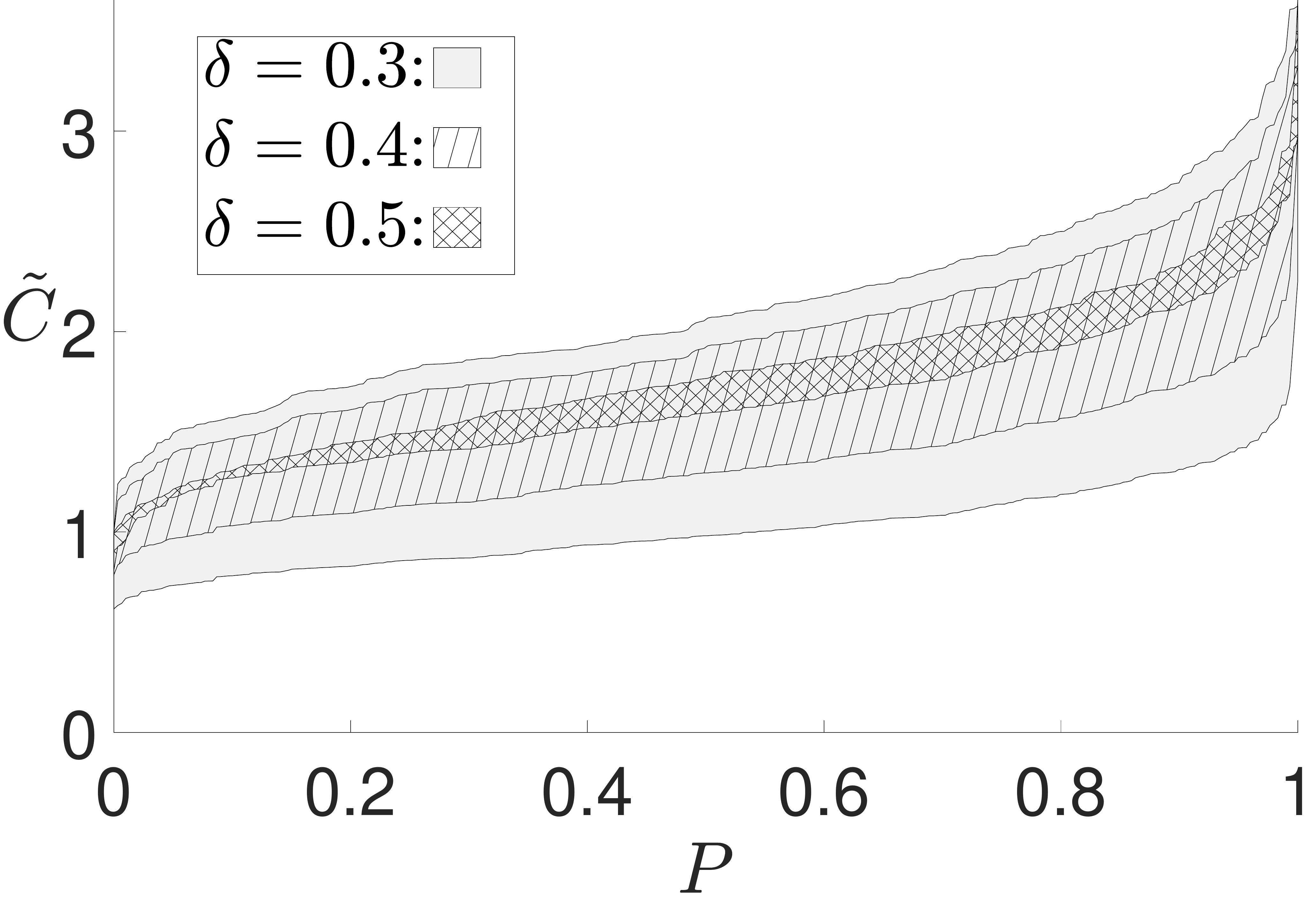}
				\caption{Crank--Nicolson scheme.}
				\label{fig:Ceps2}
			\end{subfigure}
			
			\begin{subfigure}[t]{0.3\textwidth}
				\centering
				\includegraphics*[width =\textwidth,keepaspectratio,clip]{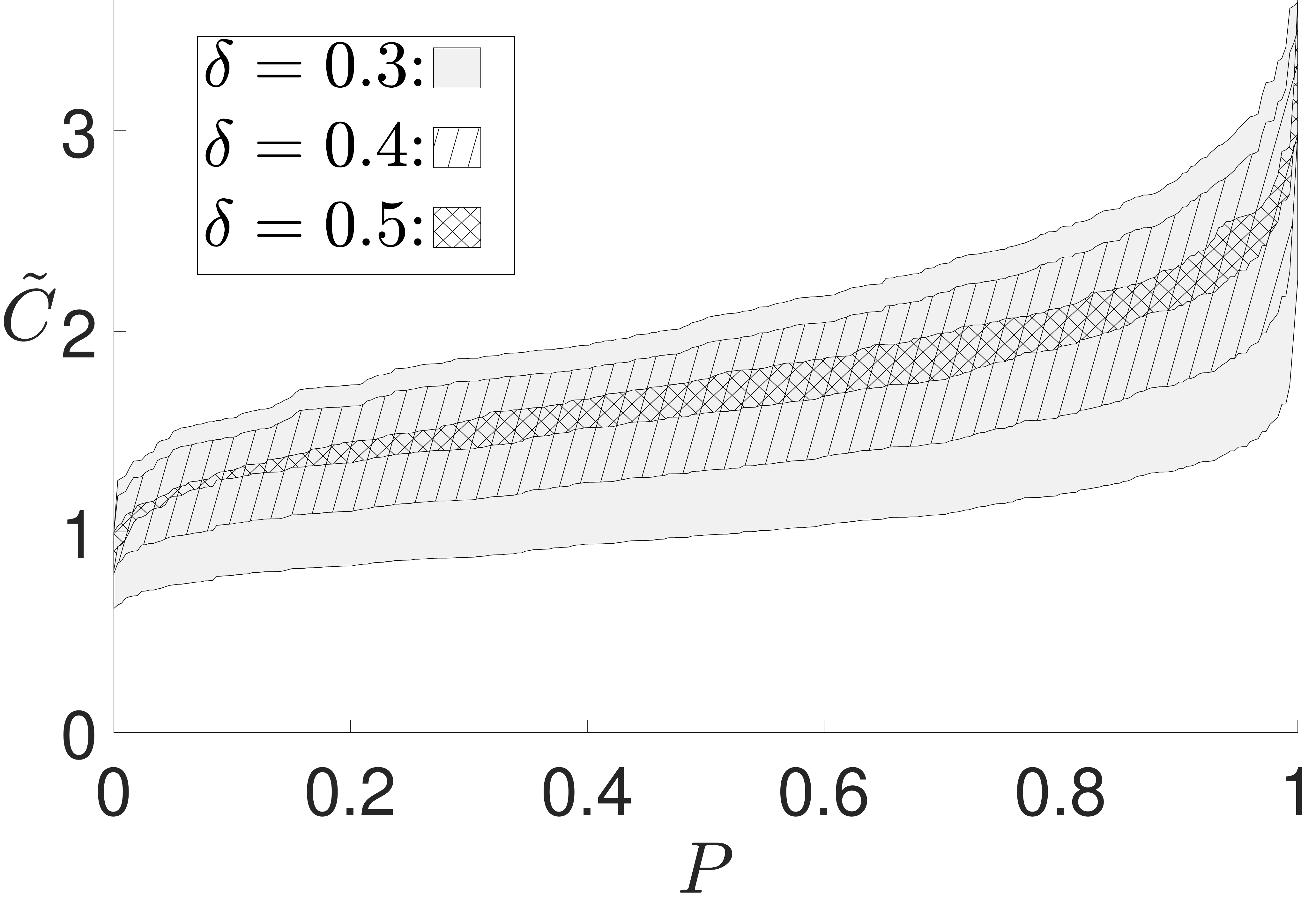}
				\caption{Exponential integrator.}
				\label{fig:Ceps3}
			\end{subfigure}
			~ 
			\begin{subfigure}[t]{0.3\textwidth}
				\centering
				\includegraphics*[width =\textwidth,keepaspectratio,clip]{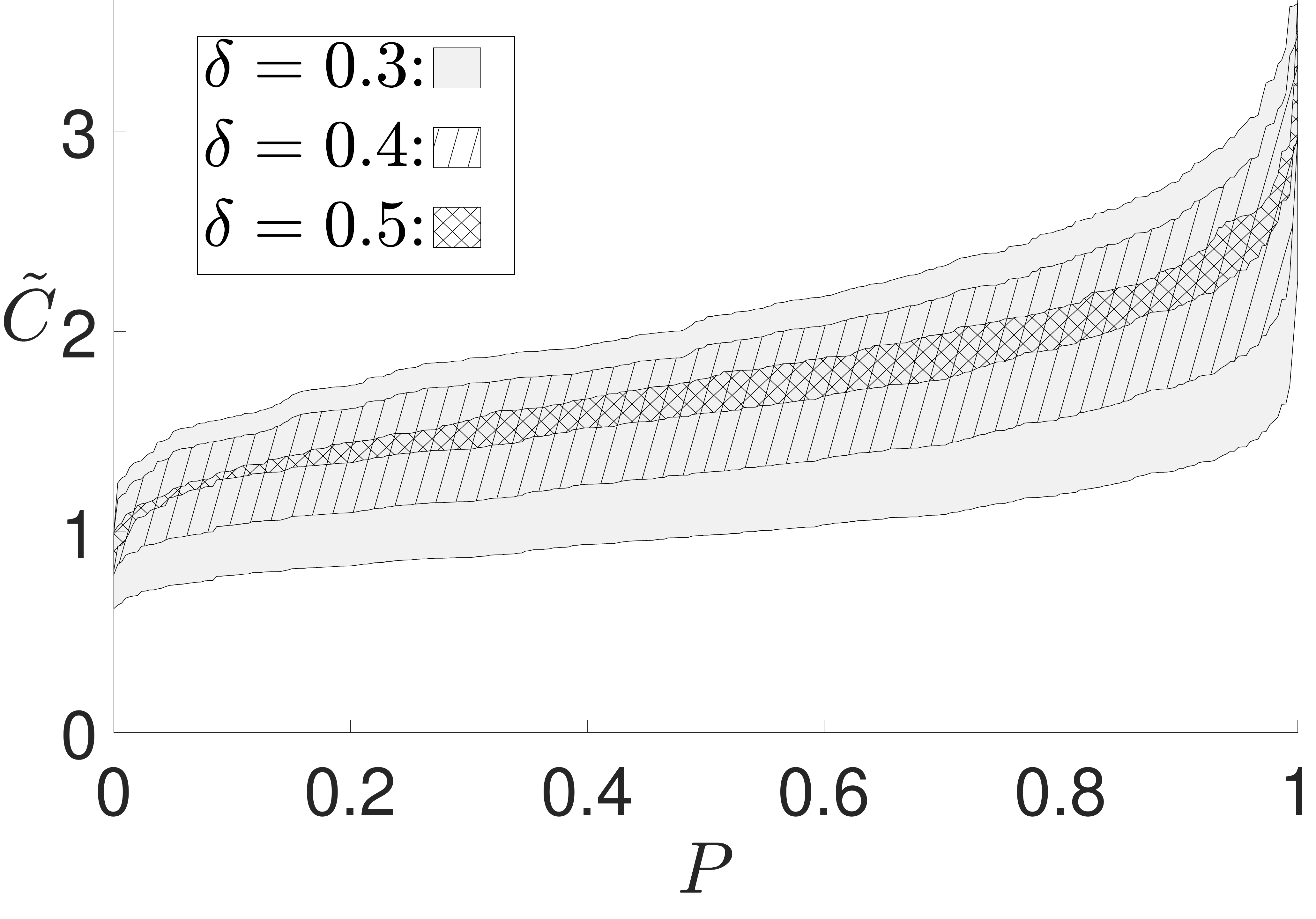}
				\caption{Relaxation scheme.}
				\label{fig:Ceps4}
			\end{subfigure}
			\caption{
				The ranges $\tilde{C}$ of each numerical schemes obtained with $\delta= 0.3,0.4,0.5$. 
				\label{fig:probConvCEst}
			}
		\end{figure}
		
		In this figure, one clearly sees how the ranges $\tilde{C}$ 
		becomes smaller when $\delta\to 1/2$. 
		This figure illustrates that all four numerical schemes 
		converge in probability with at least order $1/2$.
		
		\subsection{Almost-sure convergence}
		In this subsection, with the help of two numerical experiments, 
		we numerically demonstrate the orders 
		of almost-sure convergence of the four numerical schemes.
		This order has been shown to be $1/2^-$ for the exponential integrator and 
		the Crank--Nicolson scheme in \cite[Proposition 4]{bcd20}, and 
		\cite[Theorem 1.3]{MR3166967}, respectively as well as for 
		the Lie--Trotter scheme in Proposition~\ref{prop-as}.
		
		Restating the definition, we say that a numerical scheme converges 
		almost-surely with order $\hat{\delta}$ if for all $\delta\in(0,\hat{\delta})$, 
		there exists a random variable $K_\delta(T)$ such that one has 
		$$
		\max_{n=0,\ldots, N_\tau}\norm{X^n(\omega)-X(t_n,\omega)}_{\H^1}\leq K_\delta(T,\omega)h^\delta\quad\quad\quad\mathbb{P}-a.s.,
		$$
		for $h>0$ small enough, where $\tau$ is a stopping time. 
		For ease of presentation we take $T=1$. Numerically, 
		we can investigate the order of almost-sure convergence 
		by first computing the sample errors, 
		with respect to a reference solution $\widetilde X$, 
		$$
		e_N(\omega) := \max_{n=1,\ldots,N}\norm{X^n(\omega)  - \widetilde{X}(\omega,t_n)}_{1}
		$$ 
		at the coarse time grid 
		$\left\{t_n\right\}_{n=1}^N = \left\{nh \right\}_{n=1}^N$, where $h = 1/N$. 
		For a fixed sample, we can then estimate the constant $K_\delta(T,\omega)$, according to the formula
		$$
		\min\:\{K_\delta(T,\omega)\in \mathbb{R}^+\colon K_\delta(T,\omega) h^\delta \ge e_N(\omega)\}
		$$
		for all considered time step sizes $h$. For a sufficiently large span of $h$, we would then expect that the distance
		\begin{equation}
		\label{eq:ASError}
		e_\delta = \max_{h} |K_\delta(T,\omega) h^\delta - e_N(\omega) |
		\end{equation}
		would be minimized for the correct order $\delta$.
		
		Let us now use the above and investigate $e_\delta$ from \eqref{eq:ASError} for all four numerical schemes. 
		We simulate $300$ samples and consider the following parameters
		$a = 50$, $\Delta x = 0.2$, $T = 1$, 
		$N = 2^k$, $k = 8, 11, \ldots, 16$,
		and 
		$h = T/N$. 
		We approximate the exact solution using a reference solution,
		$\widetilde{X}$, simulated using the Lie--Trotter splitting scheme with 
		$N^\text{ref}=2^{18}$ and time step size $h^\text{ref} = T/N^\text{ref}$.  
		We then estimate the constant $K_\delta(T,\omega)$ for each sample, and compute the mean, median and standard deviation of $e_\delta$ in equation \eqref{eq:ASError}.
		The results are presented in Table~\ref{tab:ASError}. 
		We see that $\delta=1/2$ minimizes the mean, median and the deviation of $e_\delta$ 
		for all four numerical schemes. 
		This confirms the theoretical results that the order of almost-sure convergence is $1/2-$ for the Lie--Trotter scheme, Crank--Nicolson scheme, and exponential integrator.
		%This indicates that all time integrators converge almost-surely with an order close to $\delta=1/2$.
		
		By treating the $e_\delta$ as random variables dependent on the choice of $\delta$ and numerical scheme, we can statistically compare their expected values via one-sided paired t-tests. 
		We do this by pairing $e_{0.5}$ and $e_\delta$, $\delta = 0.4, 0.45, 0.55, 0.6$ for each numerical scheme, and test the hypothesis pairs
		$$
		\begin{cases}
		H_0: \E[e_{0.5}] = \E[e_{\delta}]\\
		H_1: \E[e_{0.5}] < \E[e_{\delta}].
		\end{cases}
		$$
		The p-values obtained by the t-tests is the probability, under the assumption of the null hypothesis $H_0$, to observe a set of observations at least as extreme as those tested. 
		This means that if one obtains a p-value smaller than a chosen significance level, typically $5\%$ or $1\%$, one may reject the null hypothesis $H_0$ in favor of the alternative hypothesis, $H_1$.
		For $12$ of the $16$ combinations of $e_\delta$ and numerical schemes, we obtain p-values which are at most approximately $10^{-11}$. 
		When comparing $e_{0.5}$ with $e_{0.55}$ for the four numerical schemes, we obtain p-values of approximately $0.35$ to $0.37$.
		We may therefore safely reject the null hypotheses in favor of the alternative hypotheses, with the exception of $\E[e_{0.5}] < \E[e_{0.55}]$. 
		This leads to the conclusion that $\delta$ close to $1/2$ minimizes the mean of the chosen $e_\delta$, for all four numerical schemes.
		Combining this with the results in Table~\ref{tab:ASError} thus illustrates that all time integrators converge 
		almost-surely with at least order $\delta=1/2$. 
		
		%Improving on this result, we perform several one-sided paired t-tests for each numerical scheme, with the pairings between the $120$ samples of $e_{0.5}$ and $e_\delta$, $\delta = 0.4, 0.45, 0.55, 0.6$. This corresponds to testing the hypothesis pairs
		%$$
		%\begin{cases}
		%H_0: E[e_{0.5}] = E[e_{\delta}]\\
		%H_1: E[e_{0.5}] < E[e_{\delta}].
		%\end{cases}
		%$$
		%The resulting p-values are at most approximately $10^{-6}$, rejecting the null hypotheses $H_0$ in favor of the alternative hypotheses. 
		%This leads to the conclusion that $\delta = 1/2$ minimizes the mean of $e_\delta$ for all four numerical schemes.
		%Combining this with the results in Table~\ref{tab:ASError}, strongly indicates that all time integrators converge almost-surely with an order close to $\delta=1/2$.
		%%further indicating that all time integrators converge almost-surely with an order close to $\delta=1/2$.
		
		\begin{table}\centering
			\captionsetup{justification=centering}
			\begin{tabular}{|l|l|l|l|l|l|}
				\hline
				Mean & $\delta = 0.4$ & $\delta = 0.45$ & $\delta = 0.5$ & $\delta = 0.55$ & $\delta = 0.6$ \\ \hline
				SEXP & 0.67228 & 0.54360 & 0.47071 & 0.47269 & 0.54750 \\ \hline
				LT & 0.67215 & 0.54351 & 0.47070 & 0.47275 & 0.54763 \\ \hline
				CN & 0.67159 & 0.54316 & 0.47082 & 0.47326 & 0.54838 \\ \hline
				Relax & 0.67179 & 0.54331 & 0.47081 & 0.47311 & 0.54816 \\ \hline
			\end{tabular}
			
			\vspace{.2cm}
			\begin{tabular}{|l|l|l|l|l|l|}
				\hline
				Median & $\delta = 0.4$ & $\delta = 0.45$ & $\delta = 0.5$ & $\delta = 0.55$ & $\delta = 0.6$ \\ \hline
				SEXP & 0.66295 & 0.53609 & 0.45448 & 0.45633 & 0.52218 \\ \hline
				LT & 0.66284 & 0.53583 & 0.45447 & 0.45621 & 0.52256 \\ \hline
				CN & 0.66324 & 0.53674 & 0.45382 & 0.45754 & 0.52234 \\ \hline
				Relax & 0.66211 & 0.53662 & 0.45384 & 0.45759 & 0.52177 \\ \hline
			\end{tabular}
			
			\vspace{.2cm}
			\begin{tabular}{|l|l|l|l|l|l|}
				\hline
				STD & $\delta = 0.4$ & $\delta = 0.45$ & $\delta = 0.5$ & $\delta = 0.55$ & $\delta = 0.6$ \\ \hline
				SEXP & 0.20884 & 0.18380 & 0.16160 & 0.16295 & 0.18948 \\ \hline
				LT & 0.20884 & 0.18378 & 0.16157 & 0.16293 & 0.18947 \\ \hline
				CN & 0.20881 & 0.18377 & 0.16135 & 0.16289 & 0.18958 \\ \hline
				Relax & 0.20882 & 0.18380 & 0.16153 & 0.16302 & 0.18969 \\ \hline
			\end{tabular}
			\caption{
				Mean, median and standard deviation of \eqref{eq:ASError} obtained for each scheme, and $\delta = 0.4,0.45, 0.5,0.55, 0.6$.
				\label{tab:ASError}}
		\end{table}
		
		In the second numerical experiment, we illustrate the behavior of how each individual sample of the Lie--Trotter splitting scheme converge{\color{blue}s} in $\H^1$ as $h\to 0$. 
		To do this, we take the parameters $a = 50$, $\Delta x = 1/16$, $T = 1$, 
		$N = 2^k$, $k = 4, 6, \ldots, 18$, and $h = T/N$. 
		The real part of the obtained numerical solutions are presented in Figure~\ref{fig:LieAS}. 
		Solutions computed with larger time steps $h$ are displayed with lighter gray, the red color is used for the reference solution $\widetilde{X}$. In this figure, one can clearly observe that not only the heights of the oscillations will converge properly, 
		but also that the offsets of the peaks, caused by coarse time stepping, will be lessen as the number of time steps increases. 
		
		\begin{figure}
			\begin{center}
				\includegraphics*[width =.5\textwidth,keepaspectratio]{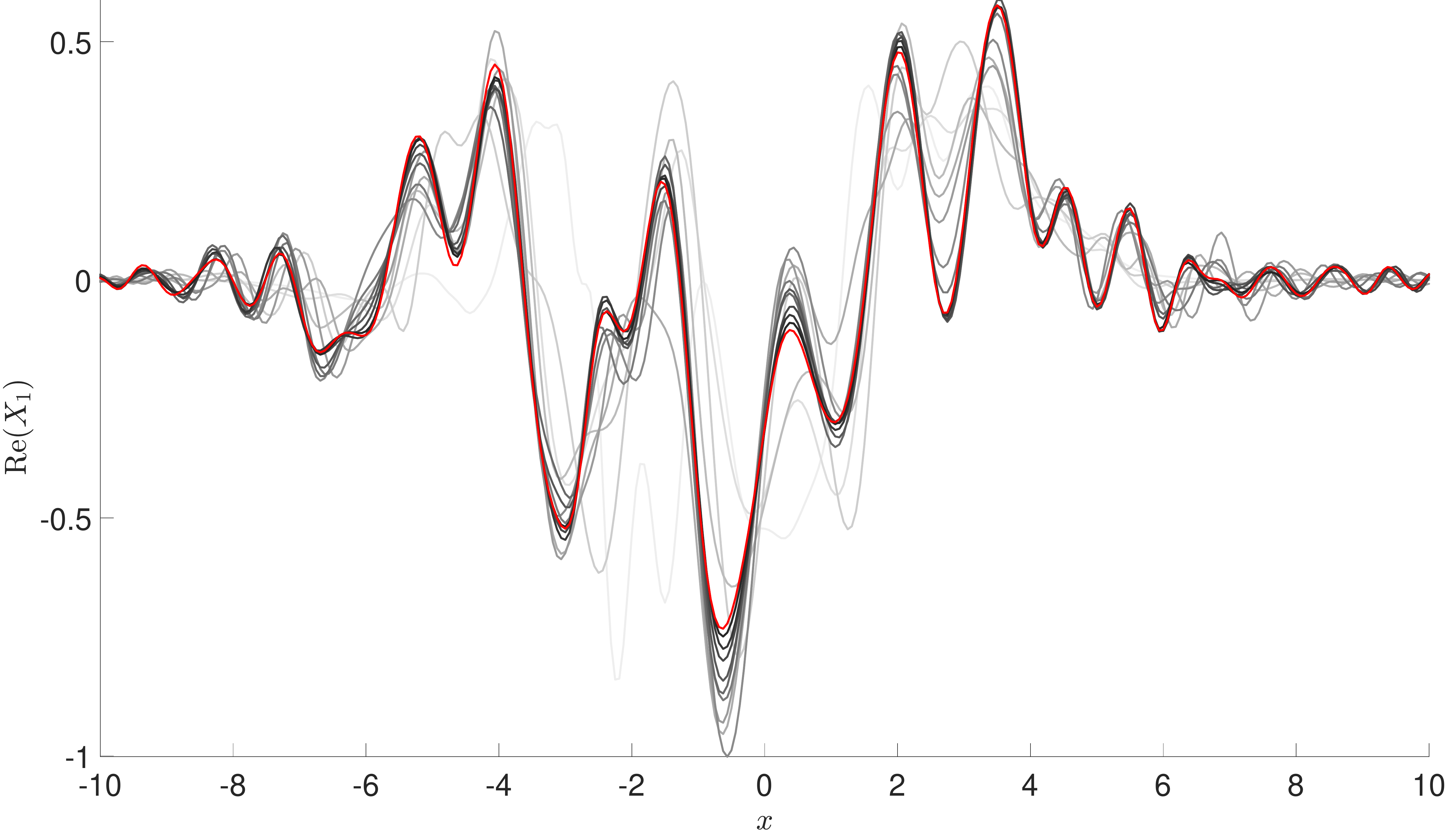}
				\caption{Almost sure convergence with the Lie--Trotter splitting scheme. 
					\label{fig:LieAS}
				}
			\end{center}
		\end{figure}
		
		\subsection{Preservation of the $\L^2$-norm}
		In this subsection we numerically illustrate the preservation of the $\L^2$-norm for the above time integrators. 
		The numerical schemes which have been shown to preserve the $\L^2$-norm are the
		Lie--Trotter splitting scheme (in Lemma~\ref{thm:LTPres} above) 
		and the Crank--Nicolson scheme (in \cite[Proposition 3.1]{MR3166967}).
		In contrast, the exponential integrator does not preserve the $\L^2$-norm, see \cite[Ch. 4]{bcd20}.
		
		For this numerical experiments, we consider the parameters $a = 50$, $\Delta x = 0.2$, $T = 2$, $N=2^{14}$, and $h = T/N$.
		We simulate $100$ samples and for each sample and each scheme compute the maximum drift in the $\L^2$-norm,
		$$
		\max_{n=1,\ldots,N} \log_{10} \left|\, \norm{X^n}_{\L^2} - \norm{X^0}_{\L^2} \, \right|.
		$$
		The maximum drifts of these samples are presented in Figure~\ref{fig:L2Cons}. 
		\begin{figure}[h!]
			\begin{center}
				\includegraphics*[width =.4\textwidth,keepaspectratio,clip]{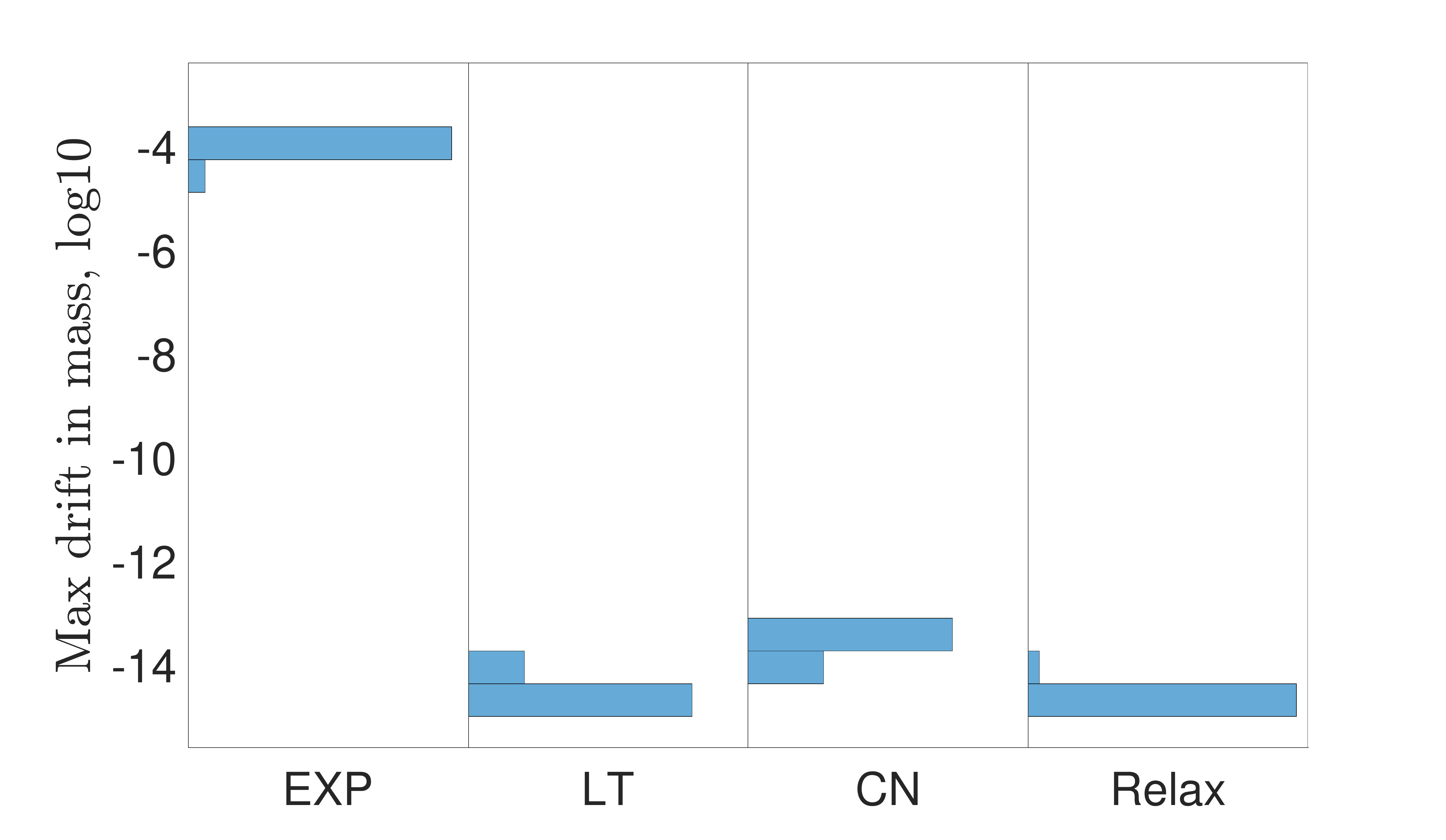}
				\caption{Log of the maximum drift in the $\L^2$-norm for each time integrators.}
				\label{fig:L2Cons}
			\end{center}
		\end{figure}
		In this figure, it can be observed that one has a preservation up to the $14$'th decimal 
		for all numerical schemes but the exponential integrator.
		
		%		{\color{blue} G.D. 04/04/20 : Maybe we should mention that the other 3 are known to be
		%			$L^2$-preserving (one because of the proof above) and that there is no reasong for the SEXP integrator
		%			to be $L^2$-preserving since the explicit Euler method used is known not to be $L^2$-preserving ?}
		
		\subsection{Computational costs}
		The goal of the present numerical experiment is to compare the computational costs of the above time integrators. 
		To do this, we consider the following parameters 
		$a = 50$, 
		$\gamma = 2$, 
		$\Delta x = 0.25$, 
		$T = 1$, 
		$N = 2^k$, $k = 8, 9, \ldots, 14$, and 
		$h = T/N$. 
		We approximate the exact solution using a reference solution, denoted by 
		$\widetilde{X}$ and simulated using the Lie--Trotter splitting scheme with $N^{\text{ref}}=2^{16}$ and time step size 
		$h^\text{ref}=T/N^{\text{ref}}$, and  
		compute the mean-square errors
		$$
		e_N := \mathbb{E}\left[\norm{X^N - \widetilde{X}(T)}_{1}^2\right].
		$$
		The expectations present in $e_N$ are approximated using $100$ samples. 
		The mean computational times for all time integrators are presented in Figure~\ref{fig:compTime}. 
		
		\begin{figure}[h!]
			\begin{center}
				\includegraphics*[width =.4\textwidth,keepaspectratio]{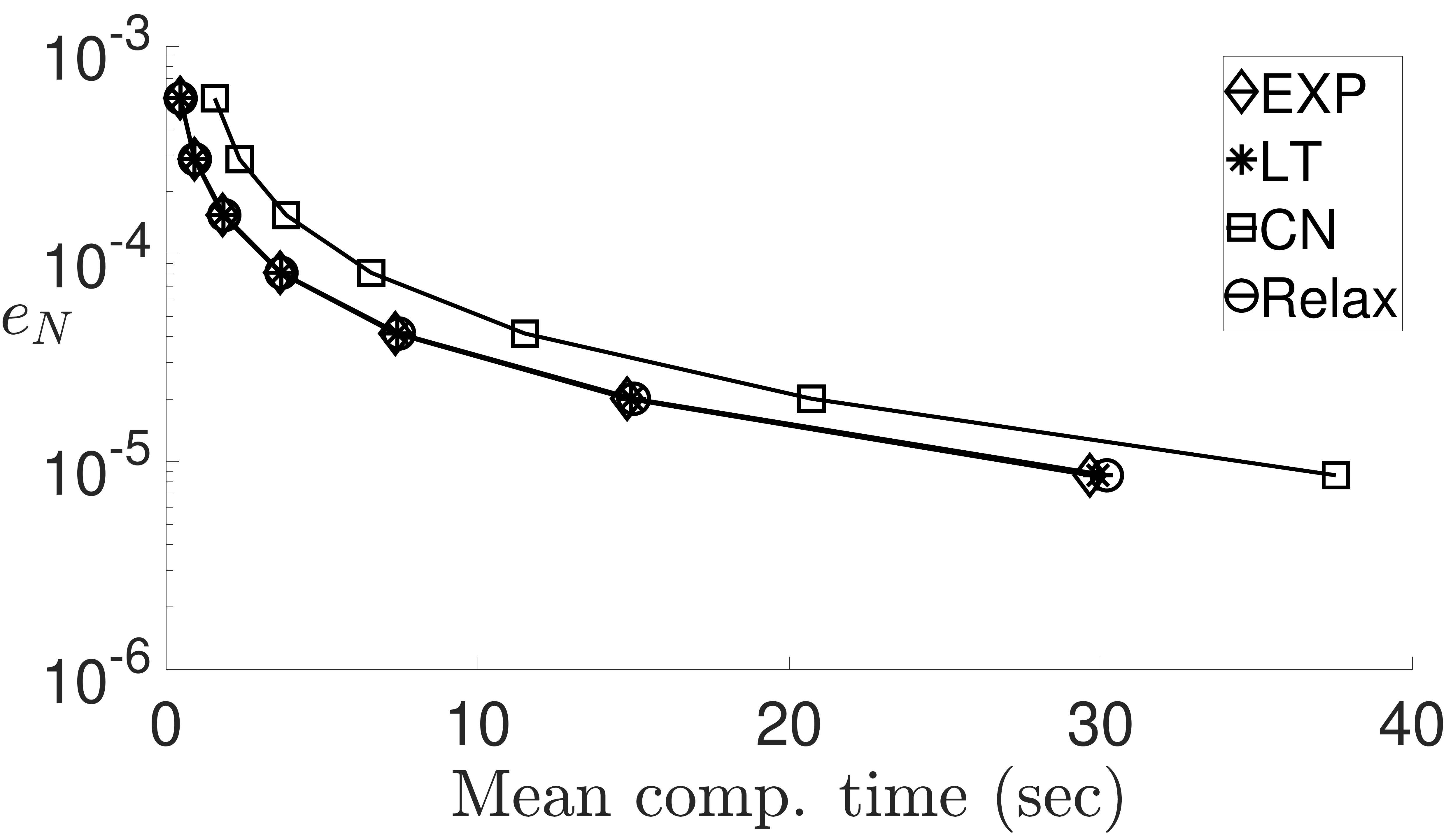}
				\caption{Mean-square errors at time $T=1$ compared to mean computational times.
					\label{fig:compTime}
				}
			\end{center}
		\end{figure}
		
		In this figure, one can see that the Crank--Nicolson scheme is significantly slower than the other three schemes, which is to be expected due to the implicit calculations. 
		In addition, it seems that the exponential integrator has a slight advantage when it comes to computational time. 
		%		Finally, it can be observed that for $h = 2^{-8}$ there are non-convergent samples for the CN scheme, 
		%		indicating some possible CFL condition for this numerical scheme.
		
		\subsection{Comparison of the evolution of solitons}
		\label{sec:solitonEvol}
		% State purpose: Compare behaviour of solitons and how they don't hold when gamma > 0.
		% Mention viriel, barycenter, and preserved properties.
		In this subsection, we investigate the impact of noise on soliton solutions to the deterministic Manakov system. 
		We do this by simulating equation \eqref{eq:halfManakov} with the Lie--Trotter splitting scheme \eqref{ltsplit} 
		for different levels of noise. 
		We observe the evolutions of the $\H^1$-norm, Hamiltonian, the mass center, and the pulse width.
		We also inspect the evolution of the profiles of the solitons. The mass center (or barycenter, or time displacement) and the pulse width (or viriel), see e.g. \cite{garn2002,GazeauPhd}, are respectively defined as
		$$T_c(t) = \langle x \rangle(t) = \frac{\displaystyle\int_\R x |X(t,x)|^2 \text{d}\,x}{\displaystyle\int_\R |X(t,x)|^2 \text{d}\,x}$$
		and 
		$$V(t) = \langle x^2 \rangle(t) = \frac{\displaystyle\int_\R \left(x-\langle x \rangle(t)\right)^2 |X(t,x)|^2 \text{d}\,x}{\displaystyle\int_\R |X(t,x)|^2 \text{d}\,x}.$$
		
		In the deterministic case, the soliton \eqref{eq:soliton} preserves a number of properties, including the $\H^1$-norm, the Hamiltonian, and the pulse width. Some of the preserved properties can be seen in the lemma below. 
		\begin{lemma}
			\label{lemma:cons}
			The soliton \eqref{eq:soliton}, obtained by equation \eqref{eq:halfManakov} with $\gamma = 0$ and initial value \eqref{eq:hasegawa}, preserves the following quantities:
			\begin{enumerate}
				%				\item $\norm{X(\cdot,t)}_{\L^2}^2$.
				\item $H(X(\cdot,t))$, where $H$ is the Hamiltonian
				\begin{equation*}
				H(X) = \frac{1}{2} \int_\R \left|\frac{\partial}{\partial x}X\right|^2 \text{d}\, x 
				- \frac{1}{4} \int_\R \left|X\right|^4 \text{d}\, x .
				\end{equation*}
				\item The norm $\norm{X(\cdot,t)}_{\H^1}^2$.
				\item $\norm{\left(\frac{\partial}{\partial x}X(x,t)\right)(\cdot)}_{\L^2}^2$.
				\item $\norm{\left|X(\cdot,t)\right|^{\sigma}}_{\L^2}^2$ for $\sigma\in\Z^+$.
				%\item $\norm{\left(\left|X_1(\cdot,t)\right|^2+\left|X_2(\cdot,t)\right|^2\right)^\sigma}_{L^2}^2$ for $\sigma\in\Z^+$.
				\item The pulse width $V(t)$.
			\end{enumerate}
		\end{lemma}
		\begin{proof}
            To show that the first four quantities are preserved along the soliton \eqref{eq:soliton}, one 
            inserts the definition of the soliton into these quantities and uses the 
            translation invariance of the Lebesgue measure.  
            
% 			Point $1$ is known since earlier, see for instance \cite{give ref??}. We start with point $4$:
% 			\begin{align*}
% 			\norm{\left|X(\cdot,t)\right|^{\sigma}}_{\L^2}^2
% 			&= \sum_{i = 1}^{2} \norm{\left|X_i(\cdot,t)\right|^{\sigma}}_{L^2}^2
% 			= \sum_{i = 1}^{2} \int_\R \left|X_i(x,t)\right|^{2\sigma} \text{d}\, x\\
% 			&= \left(\cos^{2\sigma}(\theta/2) + \sin^{2\sigma}(\theta/2)\right) \int_\R 
% 			\left|\eta
% 			\sech(\eta (x-\tau_0+\kappa t))
% 			\right|^{2\sigma} \text{d}\, x,
% 			\end{align*}
% 			and this does not depend on $t$ using the translation invariance of the Lebesgue measure. 
% 			This integral is solved through the substitutions $z = \eta (x-\tau_0+\kappa t)$ and $y = e^{2z}+1$:
% 			\begin{align*}
% 			\int_\R \left|\eta
% 			\sech(\eta (x-\tau_0+\kappa t))\right|^{2\sigma} \text{d}\, x
% 			&=
% 			\eta^{2\sigma-1}
% 			\int_\R \left|
% 			\sech(z)\right|^{2\sigma} \text{d}\, z
% 			=
% 			(2\eta)^{2\sigma-1}
% 			\int_1^\infty 
% 			\frac{(y-1)^{\sigma-1}}{y^{2\sigma}}
% 			\text{d}\, y\\
% 			&= 
% 			(2\eta)^{2\sigma-1}\sum_{i=0}^{\sigma-1} {{\sigma-1}\choose i}\frac{(-1)^{\sigma-1-i}}{2\sigma-1-i}.
% 			\end{align*}
% 			This shows point $4$. Point $5$ follows similarly.
% 			Points $2$ and $3$ follow from points $1$ and $4$.
			
			Along the soliton \eqref{eq:soliton}, the evolution of the mass center $T_c$ is given by  
			\begin{align*}
			T_c(t) = \frac{\displaystyle\int_\R x |X(t,x)|^2 \text{d}\,x}{\displaystyle\int_\R |X(t,x)|^2 \text{d}\,x}
			= \frac{\displaystyle\int_\R (\eta z + \tau_0 - \kappa t)|\sech(z)|^2 \text{d}\,z}{2} = \tau_0 - \kappa t.
			\end{align*}
			It then follows that the pulse width satisfies
			\begin{align*}
			V (t)
			= \frac{\displaystyle\int_\R \left(x- T_c(t)\right)^2 |X(t,x)|^2 \text{d}\,x}{\displaystyle\int_\R |X(t,x)|^2 \text{d}\,x}
			= \frac{\pi^2}{12\eta^2}.
			\end{align*}
		\end{proof}
		With the following experiments we highlight how some of these quantities evolve
                in the stochastic case considered in this paper ($\gamma>0$),
                and present a visual comparison, using the Lie--Trotter splitting scheme
                and a pseudospectral spatial discretization.
		We consider three instances of equation \eqref{eq:halfManakov} 
		with periodic boundary conditions 
		and the noise coefficients 
		$\gamma = 0$,
		$\gamma = 1$, and 
		$\gamma = 1/20$.
		We use three initial values for each choice of $\gamma$ in the form of equation \eqref{eq:hasegawa}, with the coefficients seen in Table~\ref{table:hasegawaCoeff}.
		\begin{table}[]
			\centering
			\begin{tabular}{|l|l|l|l|}
				\hline
				Set & $1$    & $2$   & $3$     \\ \hline \hline
				%		$\gamma$ & $1$    & $1/8$   & $0$     \\ \hline \hline
				$\eta$   & $1$    & $1.2$   & $1.5$   \\ \hline
				$\theta$ & $\pi/3$ & $\pi/2$  & $-\pi/2$ \\ \hline
				$\phi_1$ & $\pi/4$ & $-\pi/4$ & $4\pi/5$ \\ \hline
				$\phi_2$ & $\pi/2$ & $\pi/4$  & $-\pi/2$ \\ \hline
				$\kappa$ & $2$    & $3$     & $4$     \\ \hline
			\end{tabular}
			\caption{\label{table:hasegawaCoeff}
				Initial value coefficients.
			}
		\end{table}
		Further, we use
		$a = 20\pi$,
		$2^{14}$ Fourier modes,
		$T = 10$, $N=2^{10}$
		and $h = T/N$.
		In the stochastic cases ($\gamma = 1$ and $\gamma=1/20$) we simulate one sample each using one common Brownian motion.
% 		{\color{blue} It is a bit difficult to express in a concise manner. The six ($6$) stochastic numerical solutions refers to the numerical solutions obtained by the three ($3$) initial values (see table) and the two ($2$) noise intensities ($1$ and $1/20$ resp.)}
		%We track the $H^1$-norm and the Hamiltonian of the numerical solutions as they evolve over time.
		The $\H^1$-norm, Hamiltonian, the mass center, and the pulse width can be seen in Figure~\ref{fig:H1HamilSoliton}. 
		In addition to this, the evolution of the first component of some of the numerical solutions can 
		be seen in Figure~\ref{fig:SolitonWaterfall}.
		
		\begin{figure}[h!]
			\centering
			\begin{subfigure}[t]{0.49\textwidth}
				\centering
				\includegraphics*[width =\textwidth,keepaspectratio,clip]{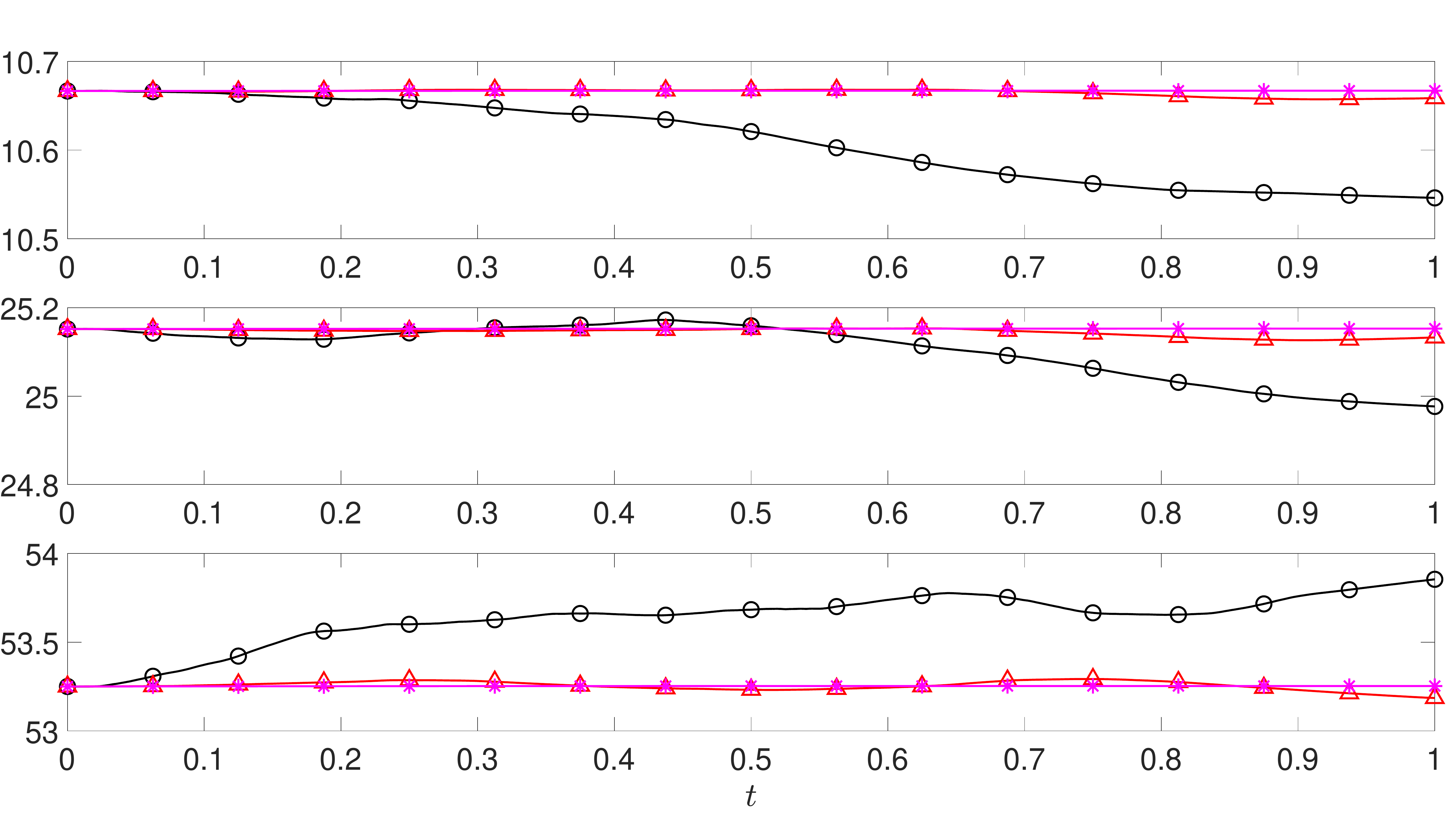}
				\caption{Evolution of $\H^1$-norm.}
			\end{subfigure}
			~ 
			\begin{subfigure}[t]{0.49\textwidth}
				\centering
				\includegraphics*[width =\textwidth,keepaspectratio,clip]{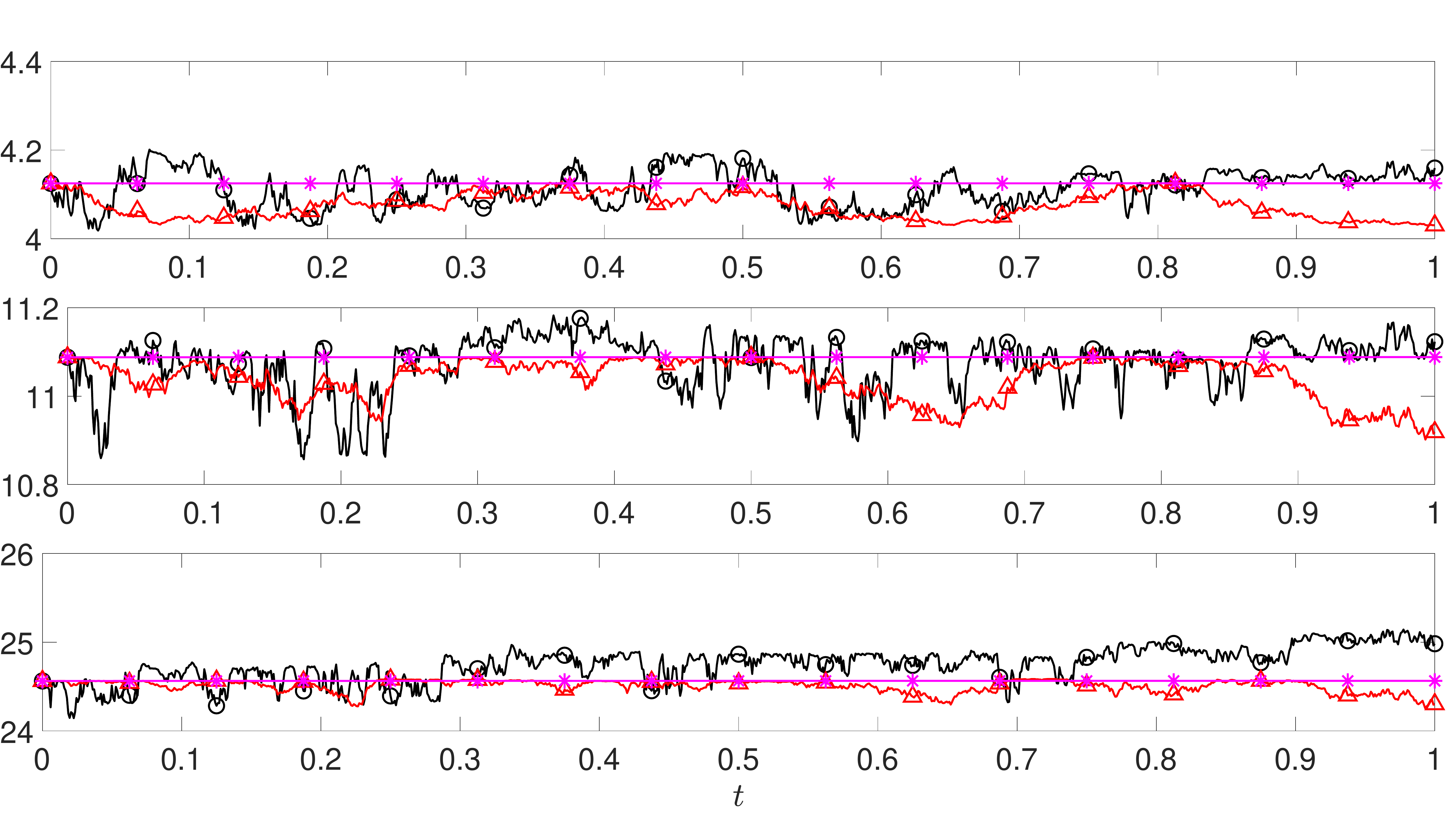}
				\caption{Evolution of Hamiltonian.}
			\end{subfigure}
			
			\begin{subfigure}[t]{0.49\textwidth}
				\centering
				\includegraphics*[width =\textwidth,keepaspectratio,clip]{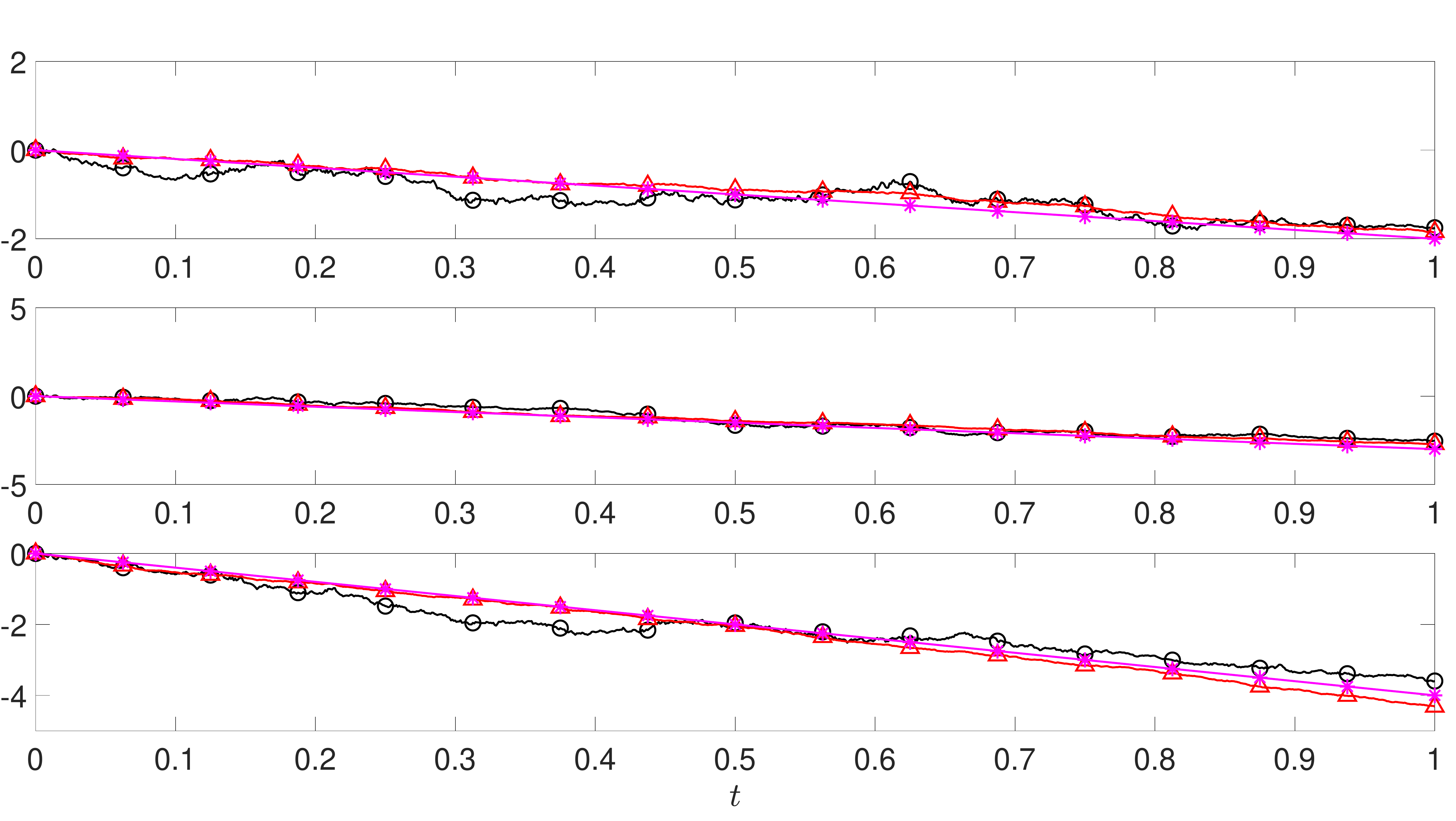}
				\caption{Evolution of mass center.}
			\end{subfigure}
			~ 
			\begin{subfigure}[t]{0.49\textwidth}
				\centering
				\includegraphics*[width =\textwidth,keepaspectratio,clip]{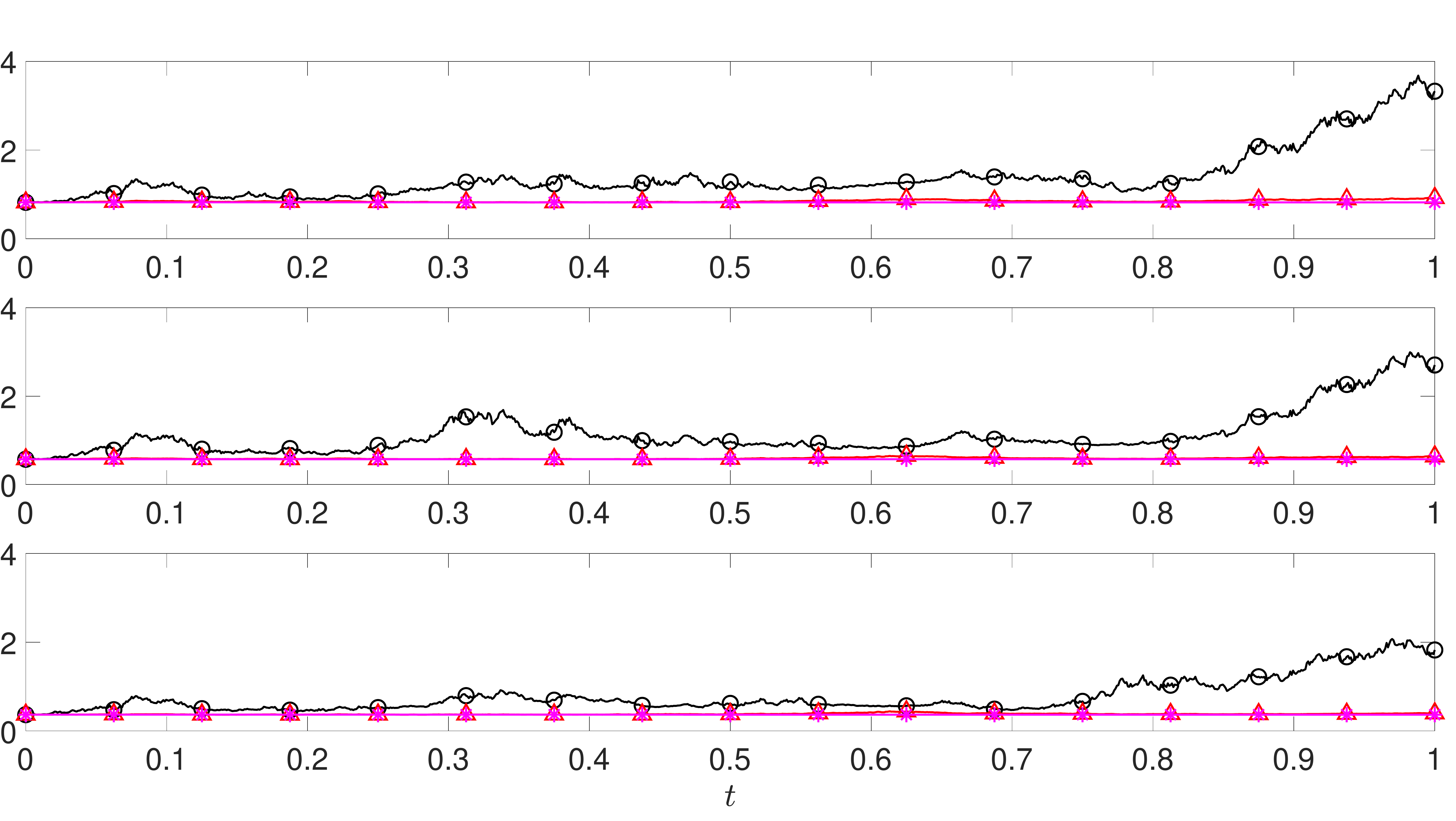}
				\caption{Evolution of pulse width.}
			\end{subfigure}
			
			\caption{
				Time evolution of various quantities of the deterministic ({\color{magenta}\textasteriskcentered{}}) and stochastic ($\circ$ and {\color{red}$\triangle$} for $\gamma=1$ and $\gamma=1/20$ respectively) numerical approximations of equation \eqref{eq:halfManakov}.
				Top to bottom: the initial values are equation \eqref{eq:hasegawa} with coefficient set $1$, $2$, and $3$, from Table~\ref{table:hasegawaCoeff}.
				\label{fig:H1HamilSoliton}
			}
		\end{figure}
		
		When $\gamma=0$, we see how the soliton produces the expected drift of the mass center.
		In addition, hardly visible in these figures, is the fact that the Lie--Trotter splitting scheme 
		does not exactly preserve the $\H^1$-norm, Hamiltonian, or the pulse width of the soliton. Their respective evolutions instead oscillate around their starting values, with the amplitude of the oscillation decreasing as $h$ decreases.
		When $\gamma>0$, it can be clearly observed that the presence of the noise prevents the preservation of the $\H^1$-norm, 
		Hamiltonian and the pulse width of solitons. Furthermore, it can be noted how the pulse width only varies slightly for $\gamma = 1/20$. Finally, the conjecture posed in \cite{Gazeau:13}, stating that the soliton is stable and not strongly destroyed for small noise and short distances, seems to hold.
		%\cite[Ch. 4.1, 5.2, 5.3, 5.4]{GazeauPhd}
		% I don't know how to connect to the conjecture without it sounding weak. We'd have to take a look at the interval [0,02] with the current evolution in order to have a convincing argument that the soliton is not destroyed, given the evolution of the amplitude.
		%		These observations are in line with the observations made by \cite[Ch. 4.1, 5.2, 5.4]{GazeauPhd} and \cite{Gazeau:13}.

		\begin{figure}[h!]
			\centering
			\begin{subfigure}[t]{0.49\textwidth}
				\centering
				\includegraphics*[width =\textwidth,keepaspectratio,clip]{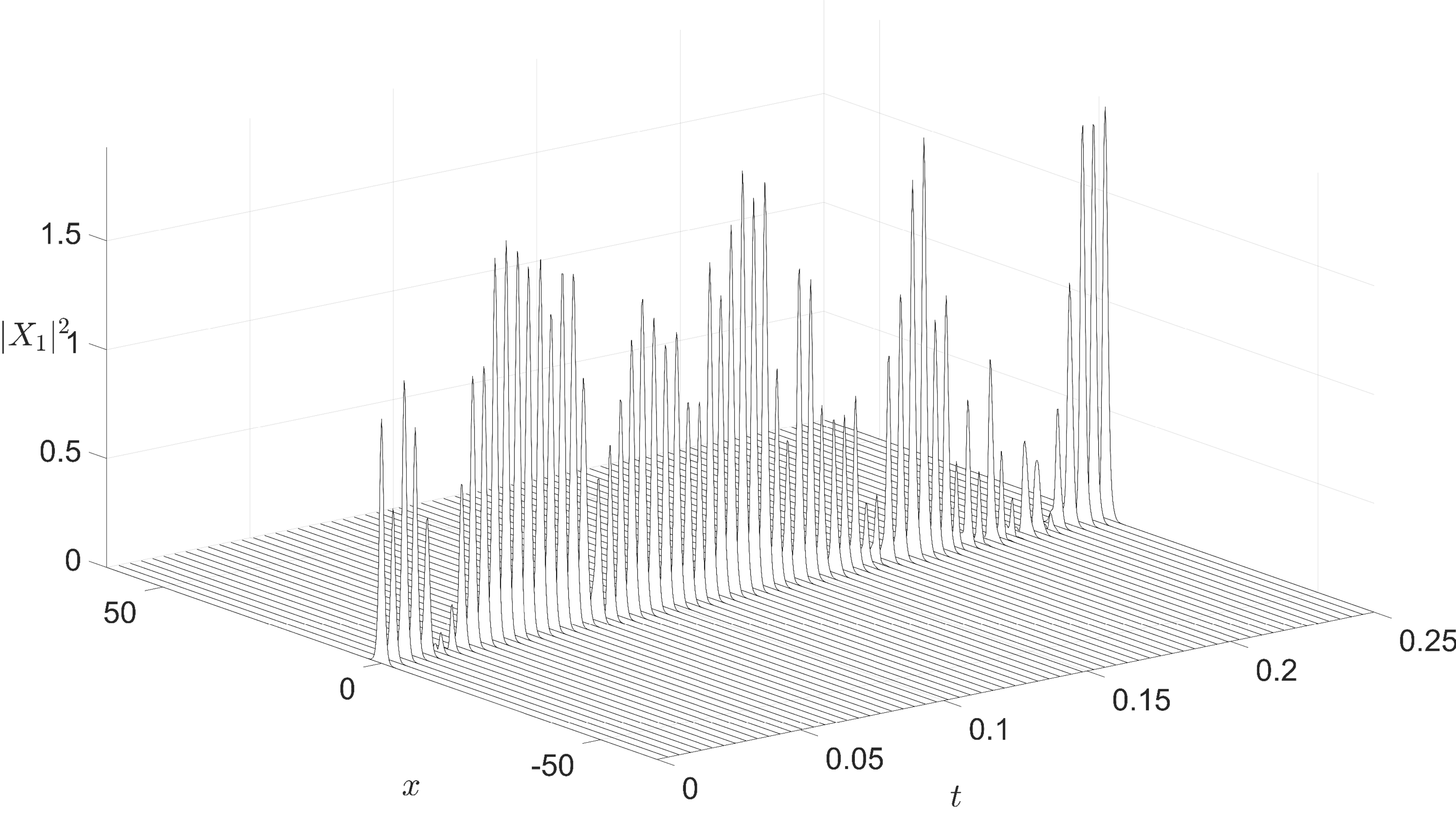}
			\end{subfigure}
		~ 
			\begin{subfigure}[t]{0.49\textwidth}
				\centering
				\includegraphics*[width =\textwidth,keepaspectratio,clip]{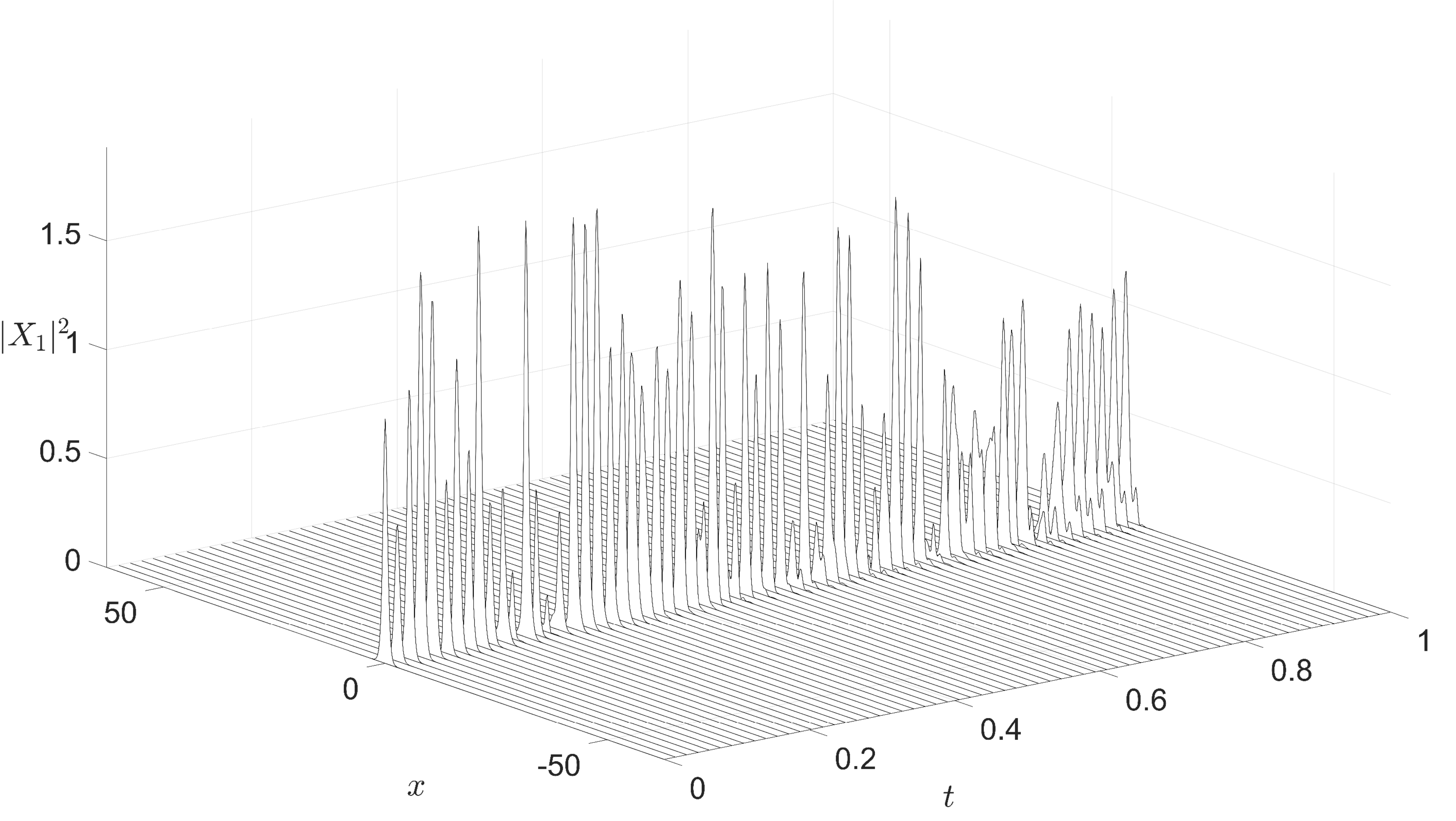}
			\end{subfigure}
			
			\begin{subfigure}[t]{0.49\textwidth}
				\centering
				\includegraphics*[width =\textwidth,keepaspectratio,clip]{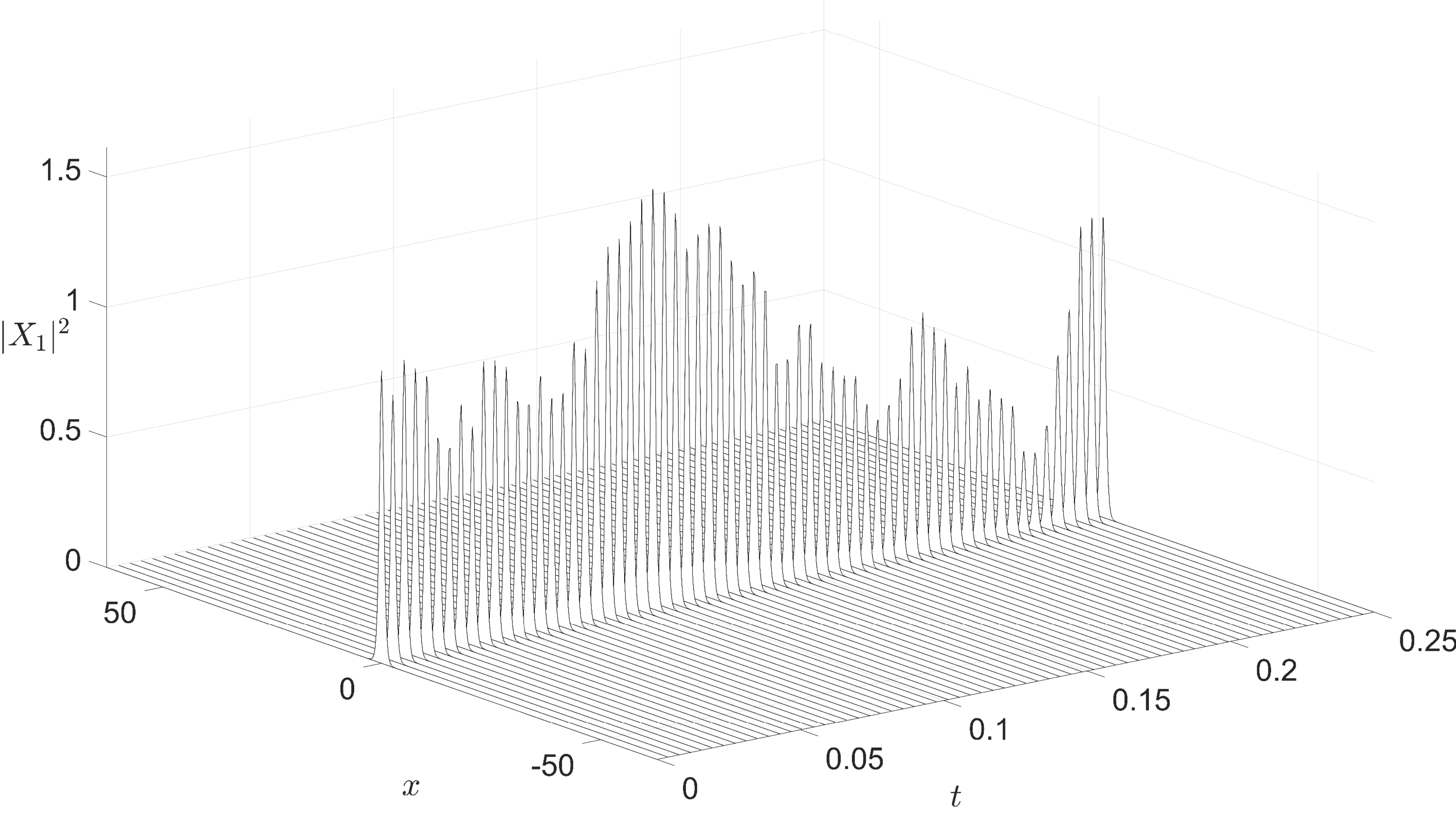}
			\end{subfigure}
		~ 
			\begin{subfigure}[t]{0.49\textwidth}
				\centering
				\includegraphics*[width =\textwidth,keepaspectratio,clip]{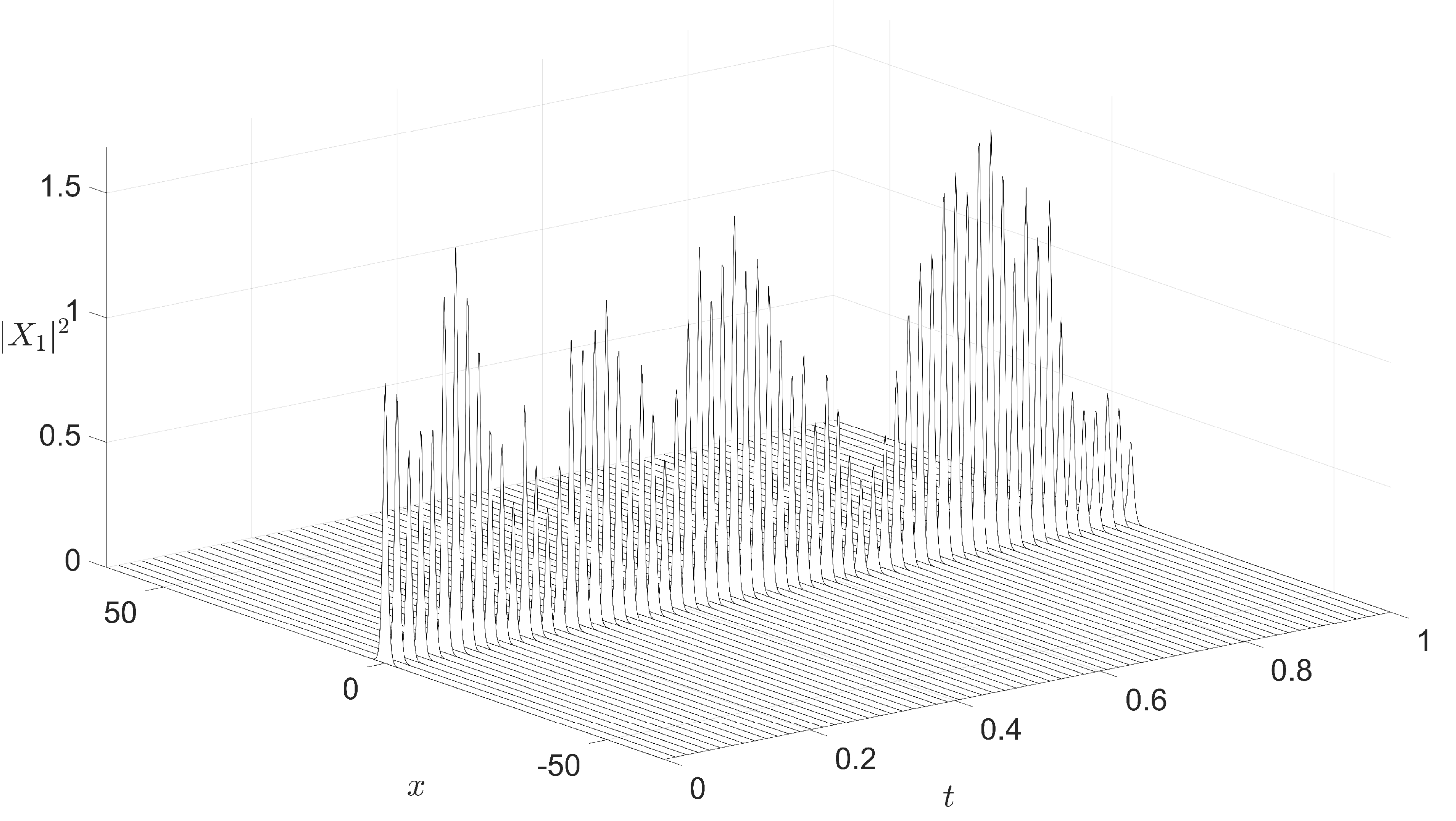}
			\end{subfigure}
			
			\begin{subfigure}[t]{0.49\textwidth}
				\centering
				\includegraphics*[width =\textwidth,keepaspectratio,clip]{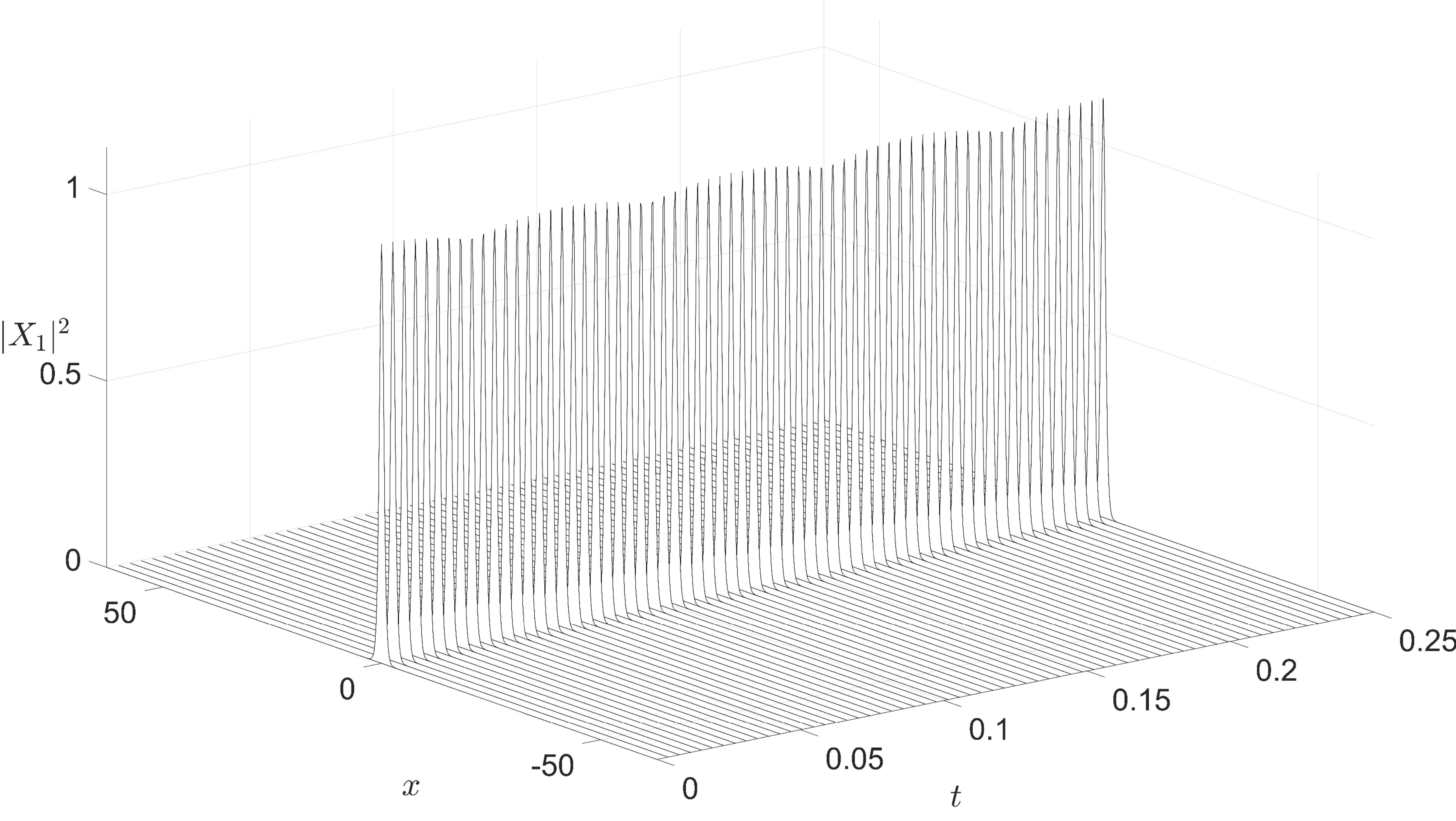}
			\end{subfigure}
		~ 
			\begin{subfigure}[t]{0.49\textwidth}
				\centering
				\includegraphics*[width =\textwidth,keepaspectratio,clip]{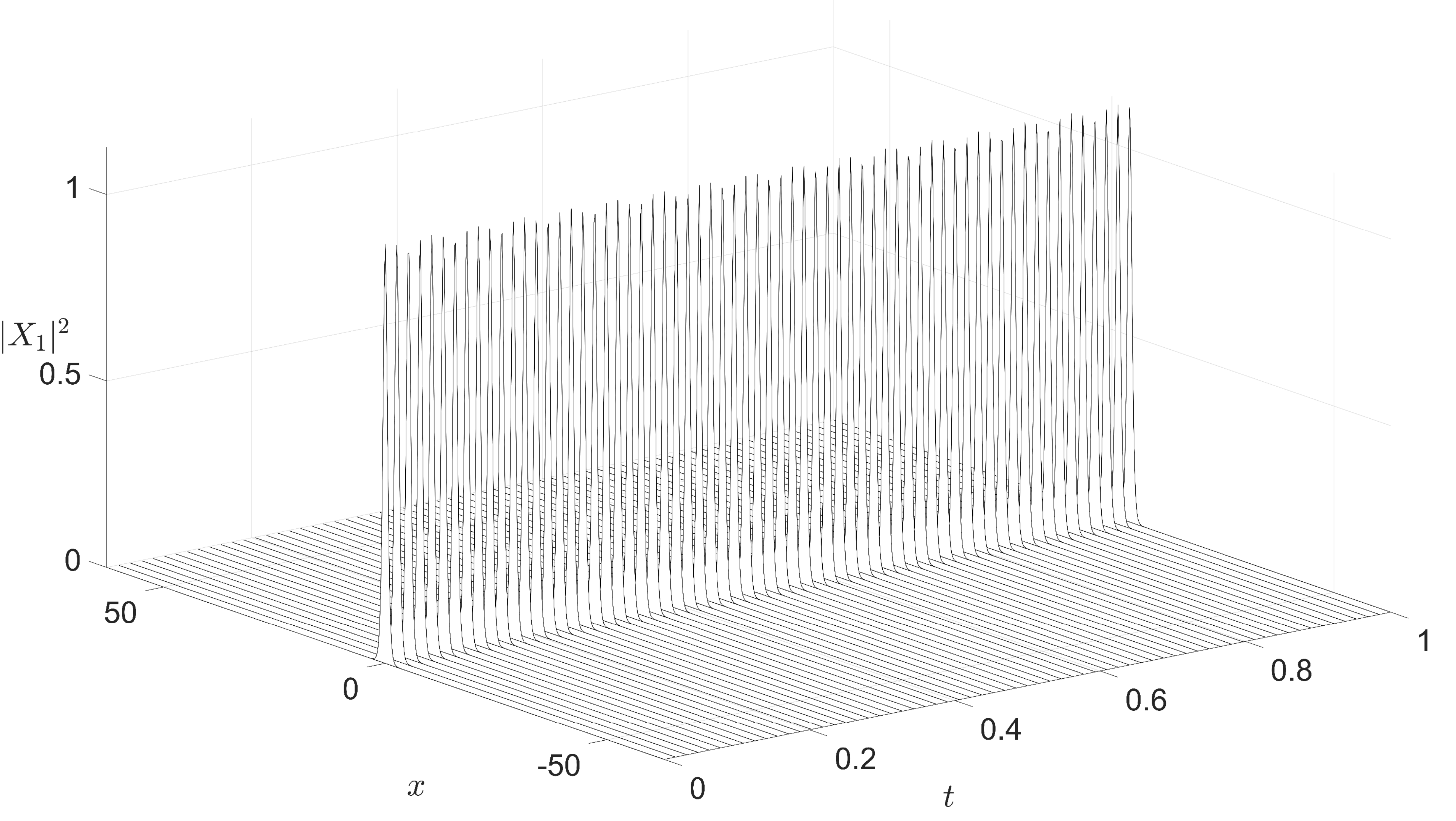}
			\end{subfigure}
			\caption{
				Evolution of the square of the modulus of the first components, $X_1$, up to $t=2.5$ (left) and $t=10$ (right) of the deterministic (bottom) and stochastic ($\gamma = 1$ top and $\gamma=1/20$ middle) numerical approximations of \eqref{eq:halfManakov} with initial value equation \eqref{eq:hasegawa} and coefficient set $3$ from Table~\ref{table:hasegawaCoeff}.
				\label{fig:SolitonWaterfall}
			}
		\end{figure}
		
		\subsection{Conjecture on the critical exponent}
		A stochastic partial differential equation related to the stochastic Manakov equation \eqref{eq:Manakov} 
		or \eqref{eq:halfManakov} is the nonlinear Schr\"odinger equation with white noise dispersion (NLSw) in dimension $d$
		\[
		\begin{cases}
		i\text du + \Delta u\circ \text d\beta + |u|^{2\sigma}u\,\text dt=0\nonumber\\
		u(0)=u_0,
		\end{cases}
		\]
		where $u=u(x,t)$, is a complex valued random process, with $t\geq0$ and $x\in\R^d$, 
		$\Delta u=\displaystyle\sum_{j=1}^d\frac{\partial^2 u}{\partial x_j^2}$ denotes the Laplacian in $\R^d$, 
        $\sigma$ a positive real numbers,
		$\beta=\beta(t)$ is a real valued standard Brownian motion, and
		$u_0$ is a given initial value, see for instance \cite{MR3736655}. Solutions to this SPDE may blowup in finite time, see details below, 
		depending on the choice of the power-law $\sigma$. 
		We therefore introduce a general power-law nonlinearity in the stochastic Manakov equation \eqref{eq:halfManakov} and consider
		\begin{equation}
		\label{eq:ManakovCrit}
		i\text{d}X
		+ \frac{1}{2} \partial^2_x X \,\text{d}t
		+ i \sqrt{\gamma} \sum_{k=1}^{3} \sigma_k \partial_x X \circ \text{d} W_k
		+ |X|^{2\sigma} X \,\text{d}t= 0,
		\end{equation}
		where $\sigma\in\R_+$. The aim of the following numerical experiments is to numerically 
		investigate possible blowup of the stochastic Manakov equation with power-law nonlinearity. 
		
		Following the convention of e.\,g. \cite{MR1886808}, given $u_0 \in \H^1$ and $\omega\in\Omega$, 
		we define the blowup time of the process $X$ by
		$$
		\tau(u_0,\omega) = \inf \{\tau \in[0,+\infty]:
		\lim_{t \to \tau} \norm{X(\cdot,t,\omega)}_1 = +\infty,
		X(0) = u_0
		\}.
		$$
		We say that an exponent $\sigma_\text{crit}$ is critical for equation \eqref{eq:ManakovCrit} if on the one hand
		$\tau(u_0,\omega) = +\infty$ for all $\sigma < \sigma_\text{crit}$ and all $u_0$ and on the other hand $\tau(u_0,\omega) < +\infty$ for all $\sigma > \sigma_\text{crit}$ for some $u_0$. The exponents $\sigma < \sigma_\text{crit}$ and $\sigma > \sigma_\text{crit}$ would be called subcritical exponents and supercritical exponents respectively.
		%% We could also go with the alternative definition. This would require us to properly define the cemetery point, and then how the new function space relates to \H^1.
		%Following the convention of e.\,g. \cite{MR3024974}, we add a cemetery point $\Delta$
		%denote the blowingup time of 
		%$X\in \H^1$
		%by
		%$$
		%\tau(X) = \inf \{
		%t\in[0,T], X(t) = \Delta
		%\}.
		%$$
		It has been shown that the NLSw has solutions in $H^1$ for dimension 
		$d = 1$ and $\sigma = 2$, \cite[Theorem 2.2]{MR2832639}, 
		and for $\sigma < 2/d$ in any dimension, \cite[Theorem 2.3]{MR2652190}.
		It is also conjectured, see \cite{bbd15}, that the critical exponent in the stochastic case is 
		$\sigma_{\text{crit}} = 4/d$, twice that of the deterministic case.
		Extensive numerical experiments on the NLSw are presented in \cite{bbd15,MR3736655,bcd21}.
		Identification of critical exponents for the stochastic Manakov equation \eqref{eq:ManakovCrit} 
		is still an open problem, as discussed in \cite{MR3166967} for the cubic case ($\sigma=1$). 
		
		Let us first investigate possible blowup of solutions to equation \eqref{eq:ManakovCrit} in the cubic case, 
		i.\,e. when $\sigma=1$, for $\gamma=0$ (deterministic case) and $\gamma=1$ by observing how the 
		$\H^1$-norms evolve over a sufficiently long time interval. 
		These simulations use the following four initial values:
		The initial value given by the soliton \eqref{eq:hasegawa} with the usual parameters given at the end of 
		the introduction of Section~\ref{sec-numexp},
		a sum of solitons \eqref{eq:hasegawa} (see \cite{Gazeau:13}) with arbitrarily chosen coefficients
		\begin{equation}
		\label{IV2}
		X_{0,2}
		= X_{0}(5,0,0,0,\pi/4,0,0) 
		+ X_{0}(1,\pi/3,0,0,\pi/4,3,0),
		\end{equation}
		a Gaussian initial value
		\begin{equation}
		\label{IV3}
		X_{0,3} =\begin{pmatrix}
		3\exp\left(-10x^2\right)\\
		2\exp\left(-5x^2\right)
		\end{pmatrix},
		\end{equation}
		and a modification of equation \eqref{eq:hasegawa}
		\begin{equation}
		\label{IV4}
		X_{0,4}
		= X_{0}(1,0,0,0,\pi/4,0,0) + 
		\begin{pmatrix}
		\cos(x)\exp(-x^2)\\
		0
		\end{pmatrix}.
		\end{equation}
		In order to avoid too long computational times, we perform this numerical experiment using a periodic boundary condition and a pseudospectral spatial discretization.
		We consider numerical discretizations with the following parameters: 
		$a = 20\pi$, 
%		$\Delta x = 2a/M$, where 
		$M = 2^{13}, 2^{15}$ Fourier modes,
		$T = 40$,
		$N = 2^k$, $k = 15, 17, 19$,
		and $h = T/N$.
		We then, at the time grids $\left\{t_n\right\}_{n=0}^N = \left\{nh\right\}_{n=0}^N$, 
		compute the mean of the $\H^1$-norms of the numerical solutions,
		$$\mathcal{H}_n
		=
		\E\left[
		\norm{X^n}_{\H^1}
		\right].
		$$
		The expectations are approximated using $48$ samples. The results of these numerical experiments are presented in Figure~\ref{fig:H1DetEvolCubic} (deterministic case) and Figure~\ref{fig:H1EvolCubic}.

		%In order to see to what degree the solution interacts with the boundary, we also inspect the following:
		%The $L^2$-norm of $X^n_i$ over the set $I = [-a,a]\setminus(-63a/64,63a/64)$, denoted as $\norm{\cdot}_{L^2(I)}$, 
		%and the boundary magnitude $|X^n_i|$.
		%We define and track the variables
		%$$\mathcal{E}_n
		%= 
		%\max_{\omega\in\Omega(T)}
		%\max_{t\in[0,t_n]}
		%\max_{i = 1,2}
		%\norm{X^n_i(\Omega)}_{L^2(I)}
		%,$$
		%and
		%$$\mathcal{B}_n
		%= 
		%\max_{\omega\in\Omega(T)}
		%\max_{t\in[0,t_n]}
		%\max_{i = 1,2}
		%|X^n_i(-a,\Omega)|
		%.$$
		%The results can be seen in Figure~\ref{fig:L2EdgeCubicHasegawa} to Figure~\ref{fig:absEdgeCubicGazeau}.
		%%\ref{fig:absEdgeCubicGauss}.
		%From this, we can see that there is little to no interaction with the boundary up until $t=1$ for the observed samples. This means that the observed $\mathcal{H}_n$ and $\mathcal{D}_n$ are reliable at least up until that point.
		
		%It is clear that the $\L^2$-norm only suffers from cumulative floating-point errors, as indicated by $\mathcal{D}_n$ increasing as $M$ and $N$ increases.
		%Further, it is also clear that none of the samples, for any of the chosen discretizations, display any indication of blowup.

		\begin{figure}[h!]
			\centering
			\begin{subfigure}[t]{0.35\textwidth}
				\centering
				\includegraphics*[width =\textwidth,keepaspectratio,clip]{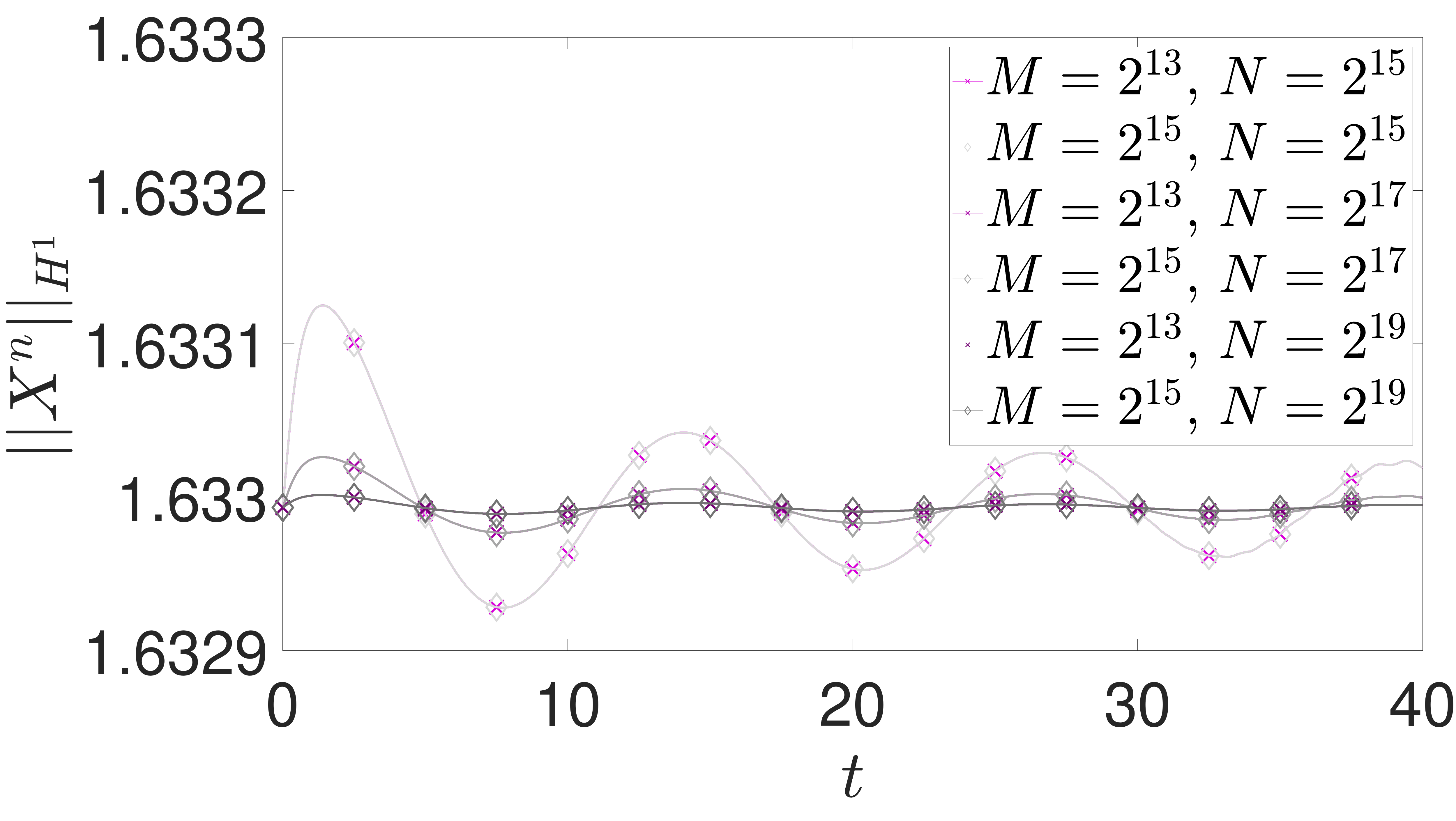}
			\end{subfigure}
			~ 
			\begin{subfigure}[t]{0.35\textwidth}
				\centering
				\includegraphics*[width =\textwidth,keepaspectratio,clip]{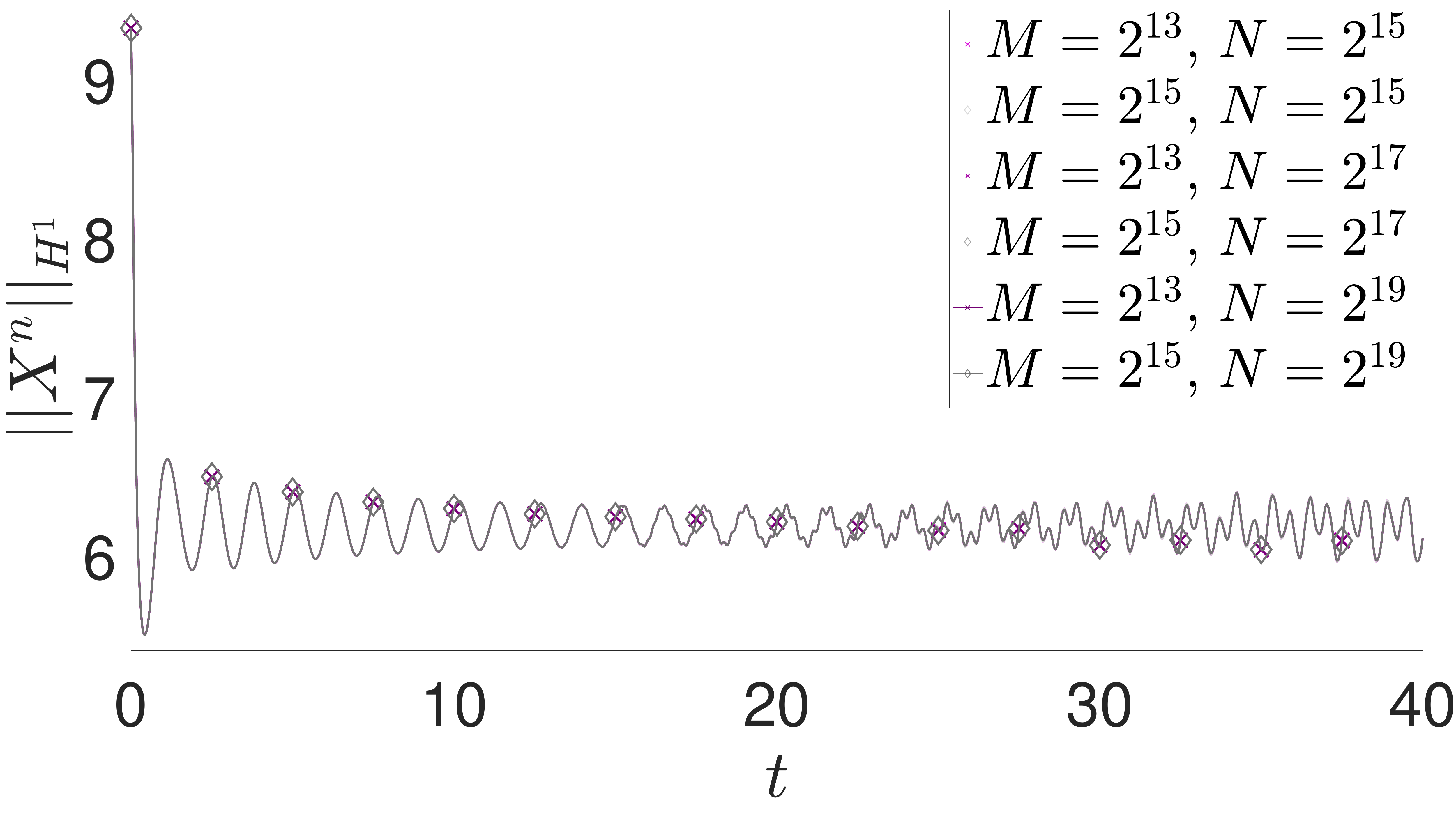}
			\end{subfigure}
			
			\begin{subfigure}[t]{0.35\textwidth}
				\centering
				\includegraphics*[width =\textwidth,keepaspectratio,clip]{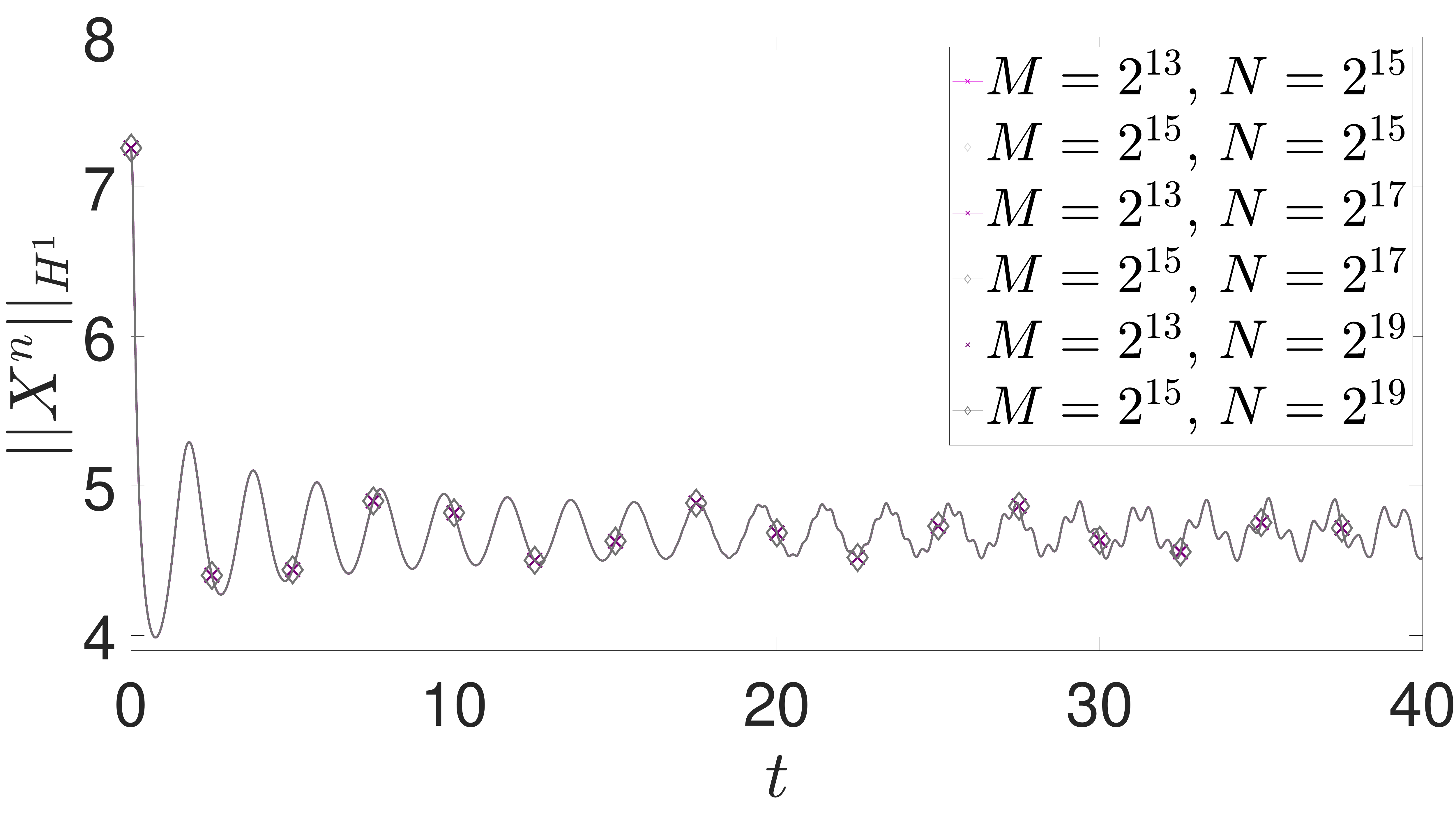}
			\end{subfigure}
			~ 
			\begin{subfigure}[t]{0.35\textwidth}
				\centering
				\includegraphics*[width =\textwidth,keepaspectratio,clip]{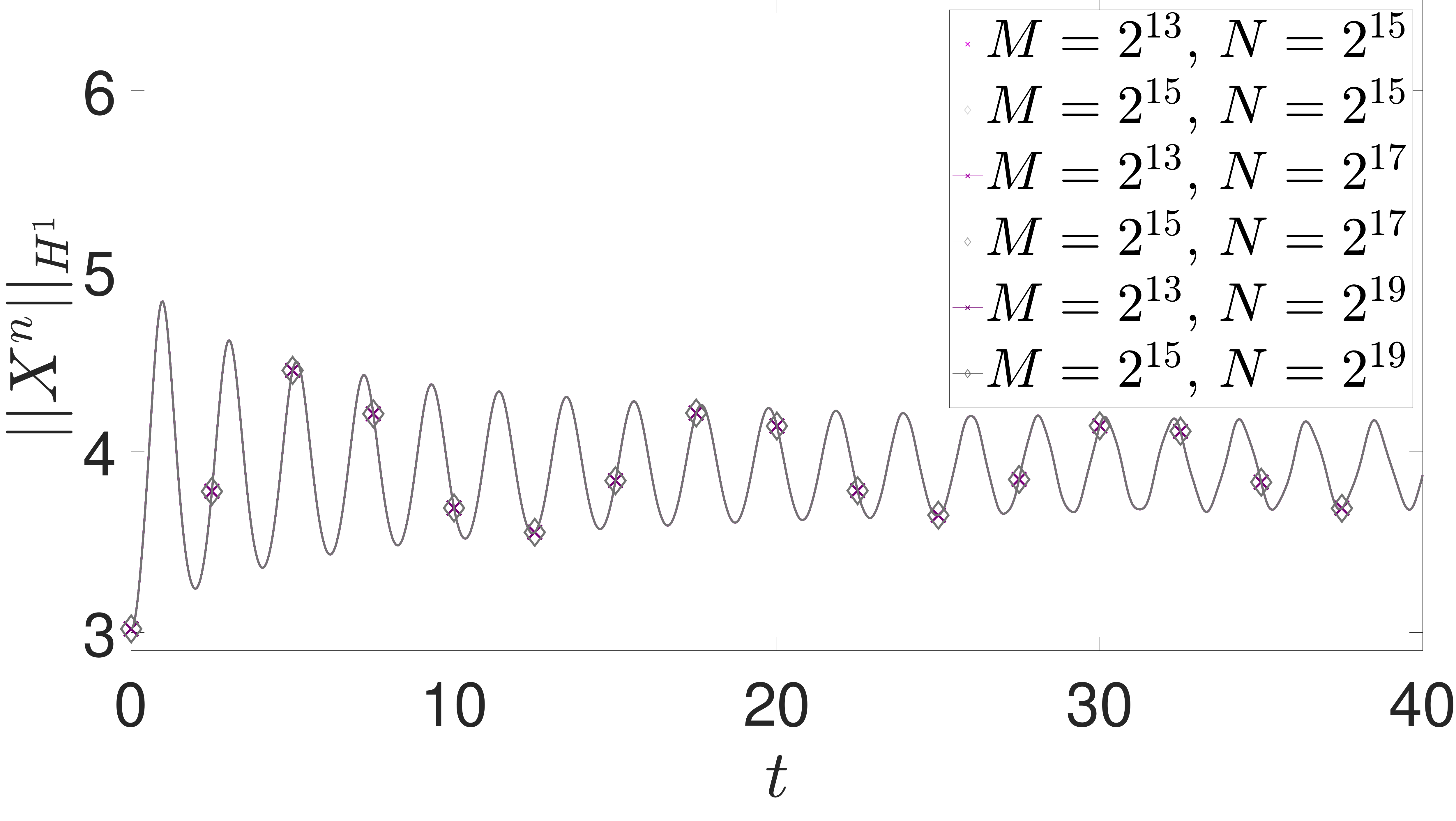}
			\end{subfigure}
			\caption{
				Evolution of $\H^1$-norm of \eqref{eq:ManakovCrit} for $\gamma = 0$ using the four initial values \eqref{eq:hasegawa}, \eqref{IV2}, \eqref{IV3}, and \eqref{IV4} (left to right, top to bottom).  
				Pink $\times$: $M = 2^{13}$, 
				grey $\diamond$: $M = 2^{15}$.
				Darker lines implies larger $N$.
				\label{fig:H1DetEvolCubic}
			}
		\end{figure}

		\begin{figure}[h!]
			\centering
			\begin{subfigure}[t]{0.35\textwidth}
				\centering
				\includegraphics*[width =\textwidth,keepaspectratio,clip]{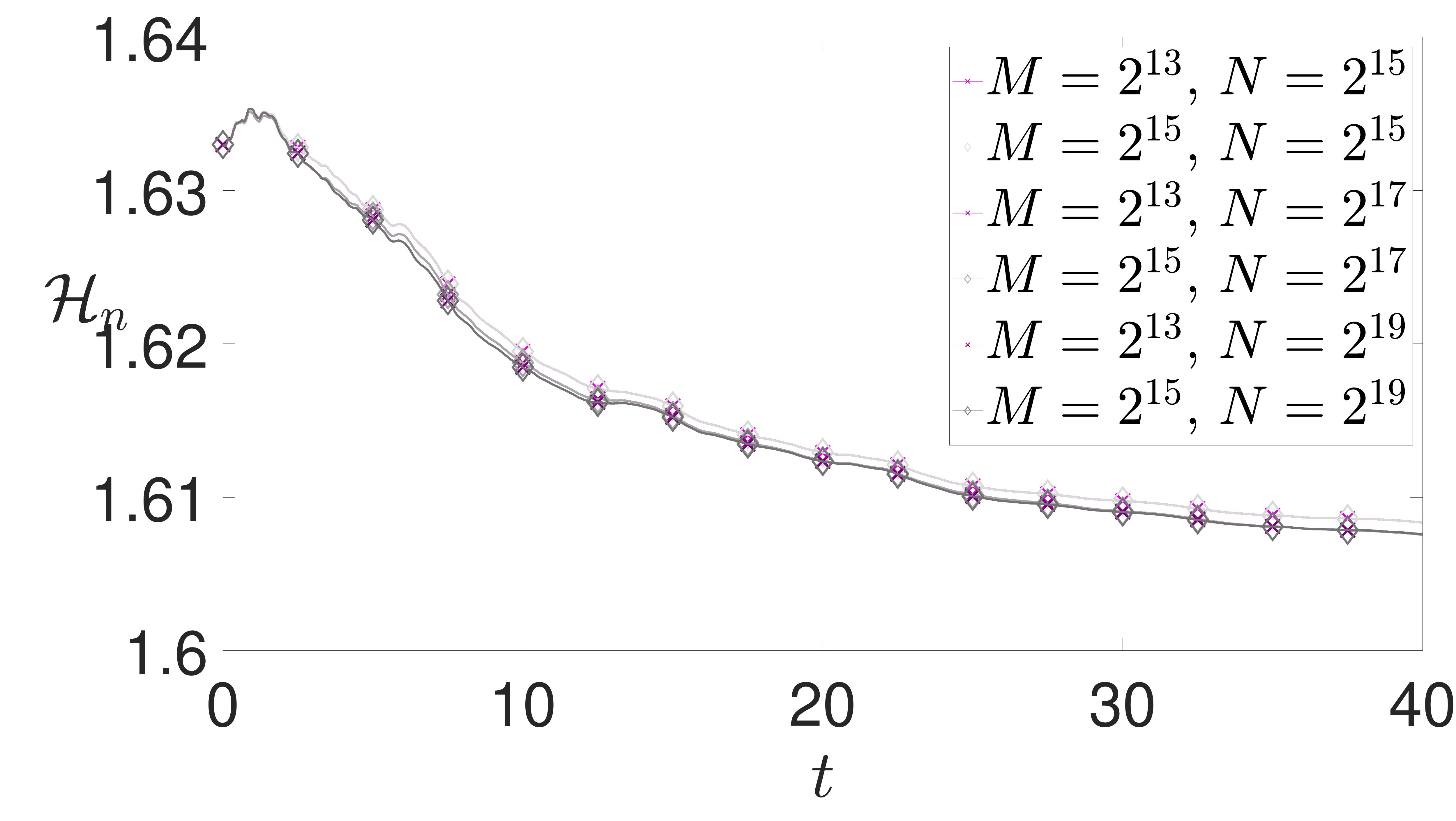}
			\end{subfigure}
			~ 
			\begin{subfigure}[t]{0.35\textwidth}
				\centering
				\includegraphics*[width =\textwidth,keepaspectratio,clip]{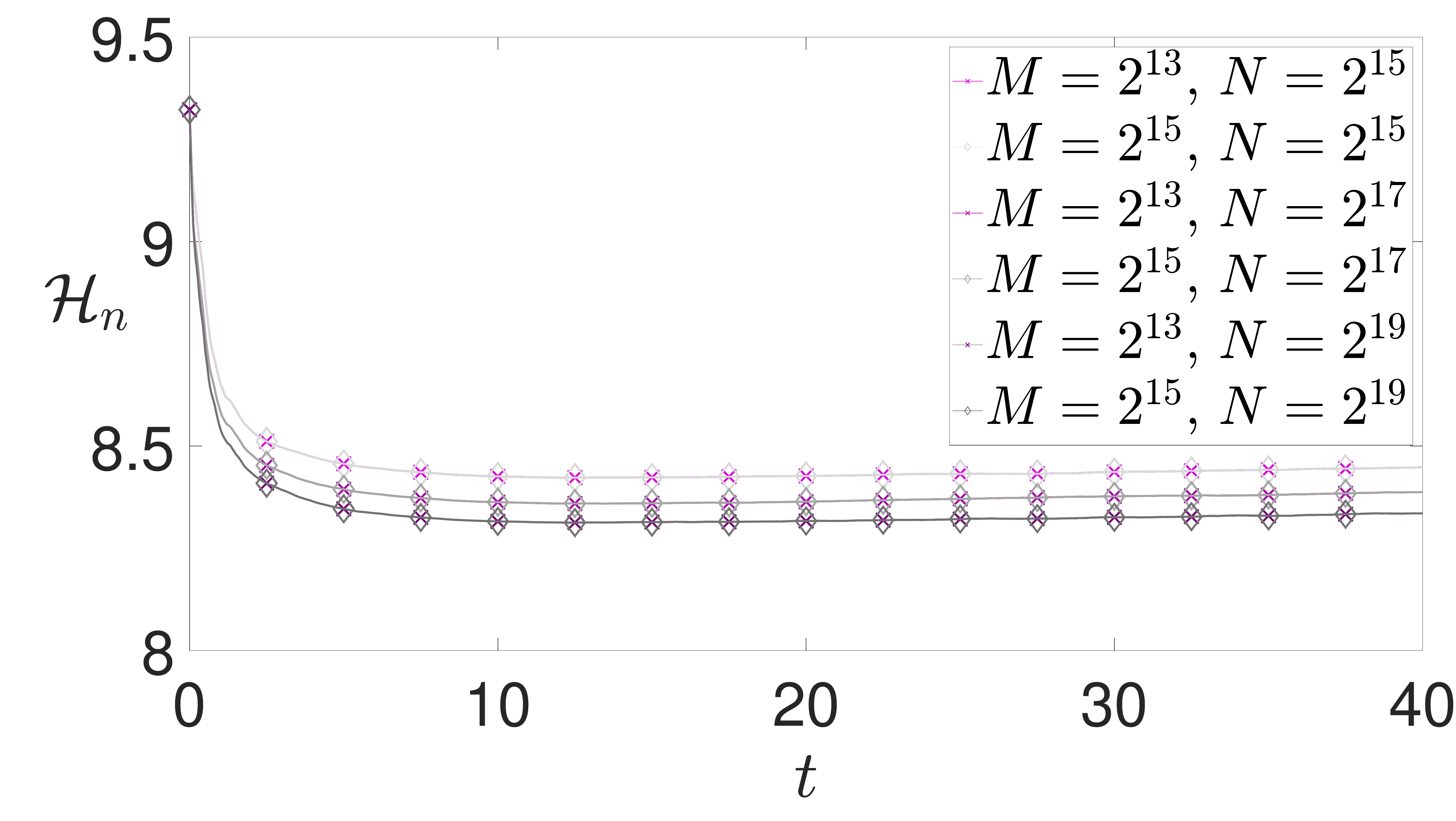}
			\end{subfigure}
			
			\begin{subfigure}[t]{0.35\textwidth}
				\centering
				\includegraphics*[width =\textwidth,keepaspectratio,clip]{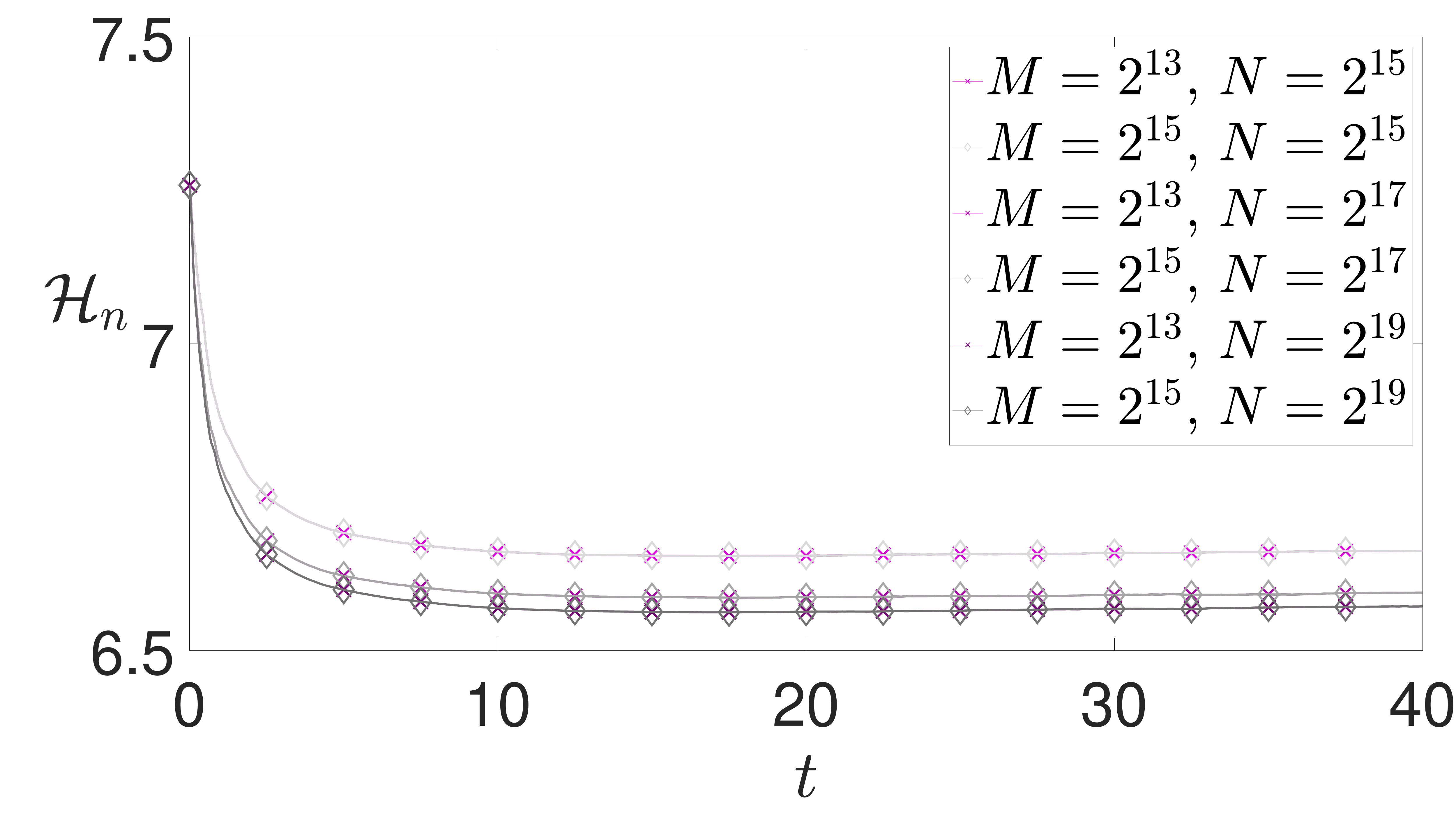}
			\end{subfigure}
			~ 
			\begin{subfigure}[t]{0.35\textwidth}
				\centering
				\includegraphics*[width =\textwidth,keepaspectratio,clip]{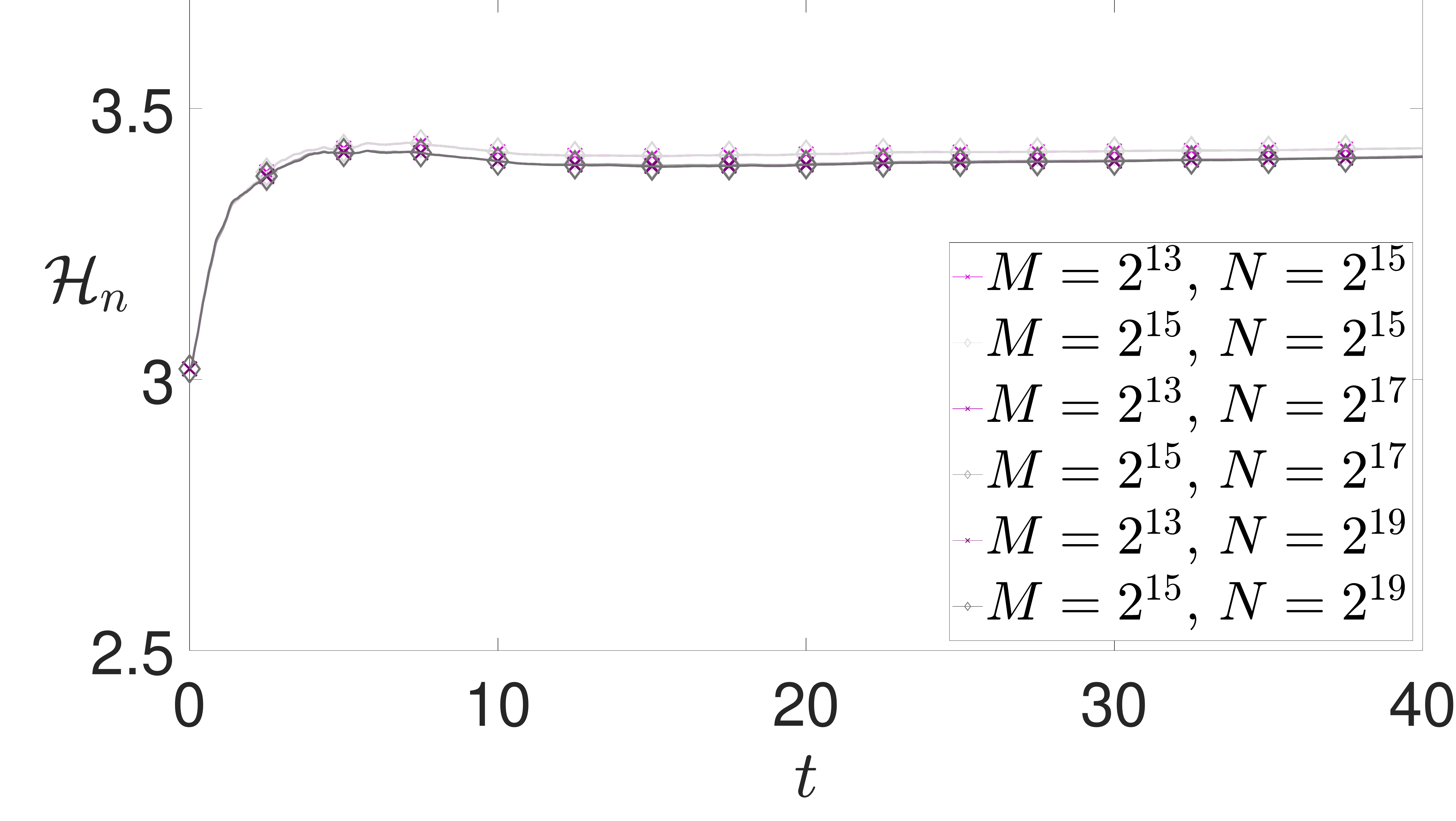}
			\end{subfigure}
			\caption{
				Evolution of the mean of $\H^1$-norms $\gamma = 1$ using the four initial values \eqref{eq:hasegawa}, \eqref{IV2}, \eqref{IV3}, and \eqref{IV4} (left to right, top to bottom). 
				Pink $\times$: $M = 2^{13}$, 
				grey $\diamond$: $M = 2^{15}$.
				Darker lines implies larger $N$.
				\label{fig:H1EvolCubic}
			}
		\end{figure}
		
		From these figures, it is clear that no indication of blowup is present, for any of the chosen discretizations or simulated samples.
		Had $\tau(u_0,\omega)<+\infty$ for any of the samples, we would have expected $\mathcal{H}_n$ to increase as $M$ and $N$ increased.
		
		Expanding on the above results, we perform two numerical experiments where we vary the exponent $\sigma = 2,3,4$ 
		in the stochastic Manakov equation with a power-law nonlinearity \eqref{eq:ManakovCrit} and the coefficient 
		$\gamma= 0$ and $\gamma=1$ (again comparing deterministic and stochastic results).
%		Investigating how the $\H^1$-norm of the numerical solutions evolve for $\sigma = 2,3,4$ we perform two numerical experiments in the deterministic ($\gamma=0$) and a stochastic setting ($\gamma=1$). 
		As in the previous experiment we perform these numerical experiments using a periodic boundary condition and a pseudospectral spatial discretization.
		In order to further limit the computational time, we only consider the initial value \eqref{IV2}, abort the simulations if $\norm{X^n}_{\H^1} > 500$, and simulate only one sample per combination of $\sigma$ and $\gamma$. 
		
		The common parameters for the following two experiments are
		$T = 0.01$,
		$N = 2^k$, $k = 16, 17$,
		and $h = T/N$.
		In order to verify possible blowup we also vary the width of the interval between the two experiments, by using 
		$a = 20\pi$, 
%		$\Delta x = 2a/M$, where 
		and $M = 2^{16}, 2^{17}$ Fourier modes
		in the first experiment and
		$a = 40\pi$, 
%		$\Delta x = 2a/M$, where 
		and $M = 2^{17}, 2^{18}$ Fourier modes
		in the second experiment.
		For the stochastic samples we use one common Brownian motion.
		
		The results can be seen in Figure~\ref{fig:H1HigherSigma} (for the first experiment with $a = 20\pi$) and in Figure~\ref{fig:H1HigherSigma2L} (for the second experiment with $a = 40\pi$). 
%		We see that there is clear evidence of blowup in several cases, and that choice of interval width makes little to no difference. 
		We see that the $\H^1$-norm increases sharply before exceeding $500$ in all cases but for the stochastic ($\gamma=1$) processes with $\sigma=2$.
		Further, we see that taking a finer discretization or wider interval neither prevents nor delays these sharp increases.
		This is a clear evidence of blowup, for $5$ of the $6$ combinations of $\gamma$ and $\sigma$, but we have two notable observations to make. 
		The first is that the presence of noise ($\gamma = 1$) either delays or completely prevents blowup when $\sigma = 2$. 
		The second is that an insufficient number of Fourier modes will fail to properly reflect the rapid increase in the $\H^1$-norm, as seen in both deterministic ($\gamma=0$) cases with $\sigma=4$.
		Other numerical experiments, not shown here, using finer time and spatial discretizations or the initial value \eqref{IV3} produce similar results.
		
		% Should we try to motivate the choice of 500 as a limit?
		
		\begin{figure}[h!]
			\centering
			\begin{subfigure}[t]{0.35\textwidth}
				\centering
				\includegraphics*[width =\textwidth,keepaspectratio,clip]{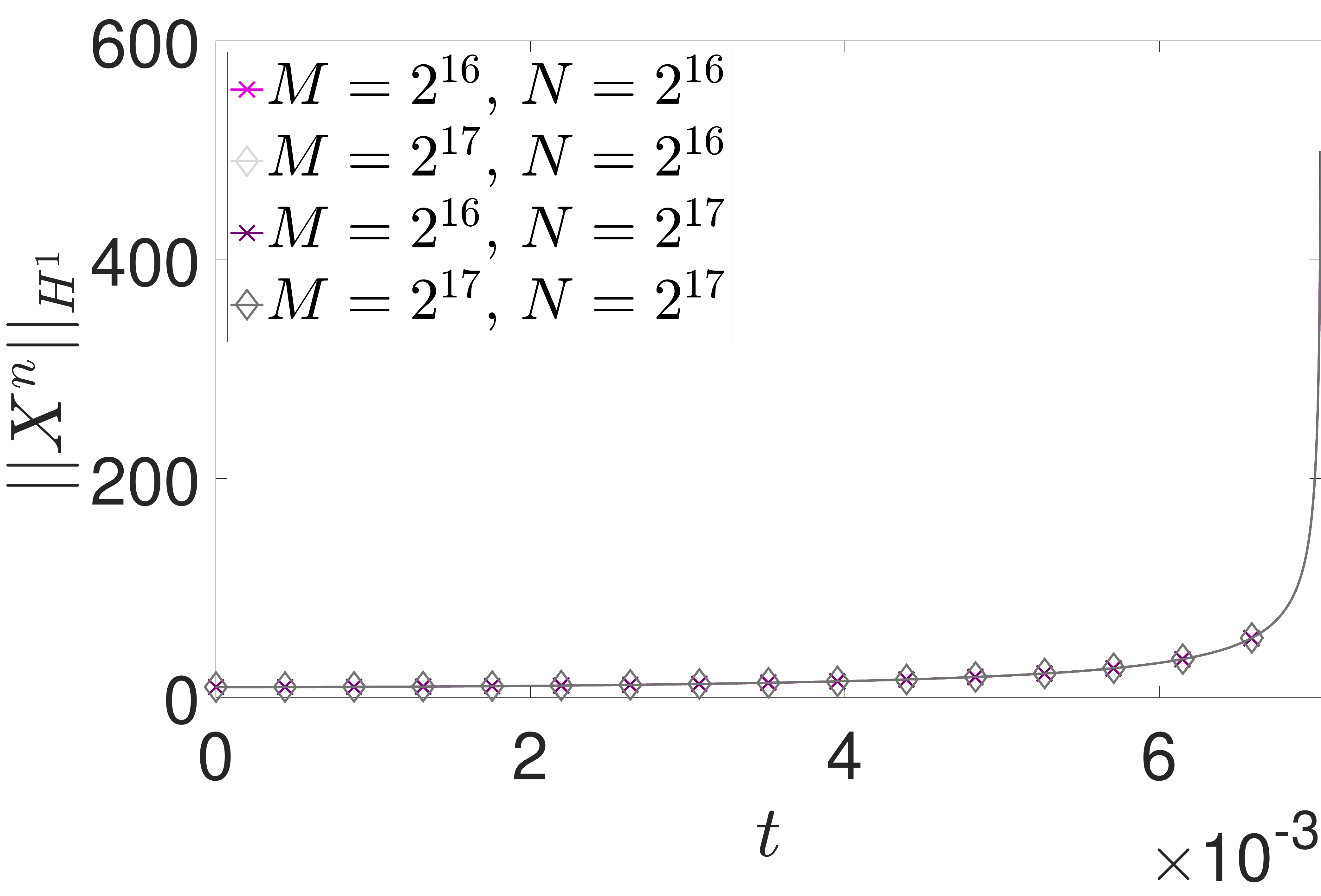}
			\end{subfigure}
			~ 
			\begin{subfigure}[t]{0.35\textwidth}
				\centering
				\includegraphics*[width =\textwidth,keepaspectratio,clip]{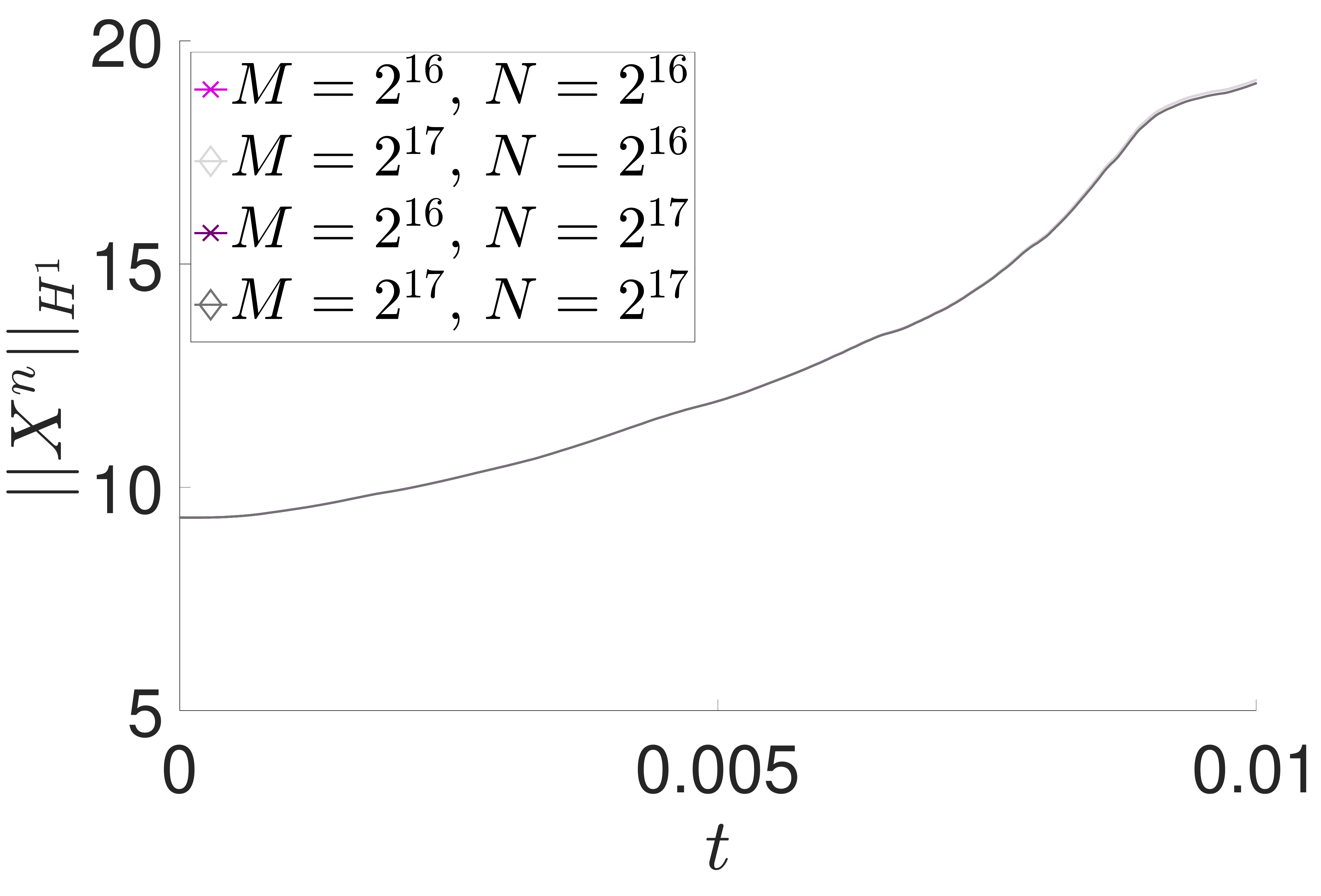}
			\end{subfigure}
			
			\begin{subfigure}[t]{0.35\textwidth}
				\centering
				\includegraphics*[width =\textwidth,keepaspectratio,clip]{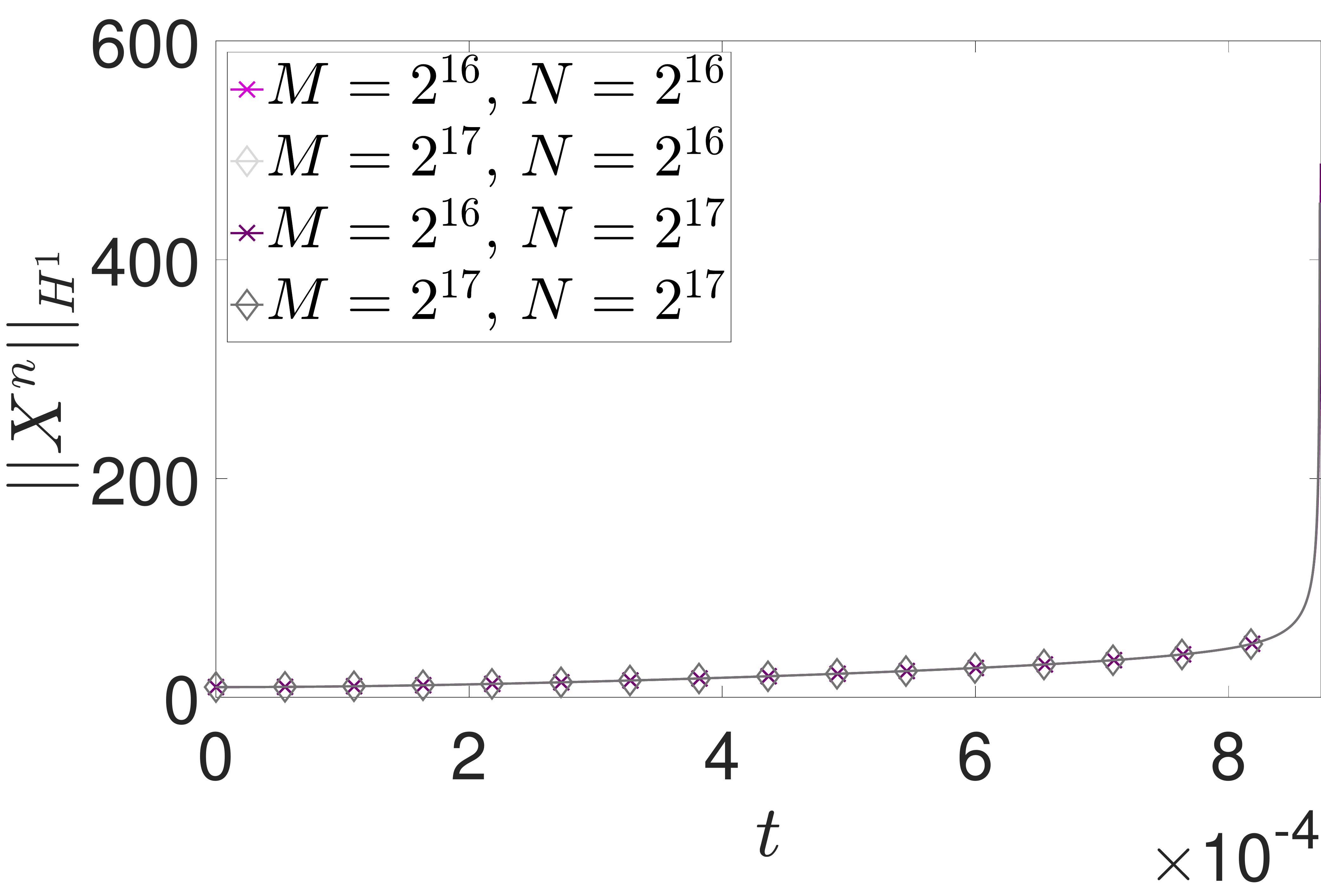}
			\end{subfigure}
			~ 
			\begin{subfigure}[t]{0.35\textwidth}
				\centering
				\includegraphics*[width =\textwidth,keepaspectratio,clip]{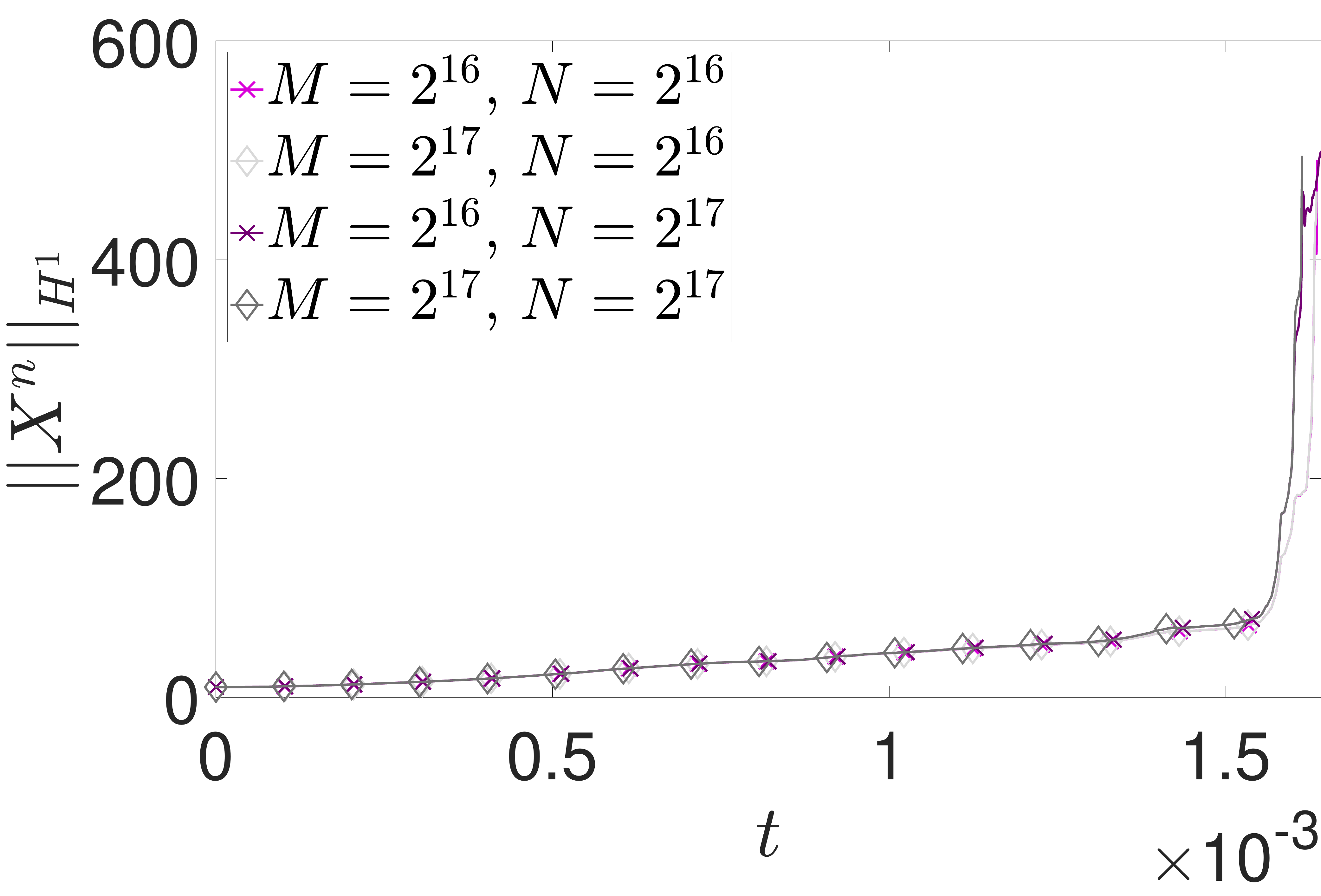}
			\end{subfigure}
		
			\begin{subfigure}[t]{0.35\textwidth}
				\centering
				\includegraphics*[width =\textwidth,keepaspectratio,clip]{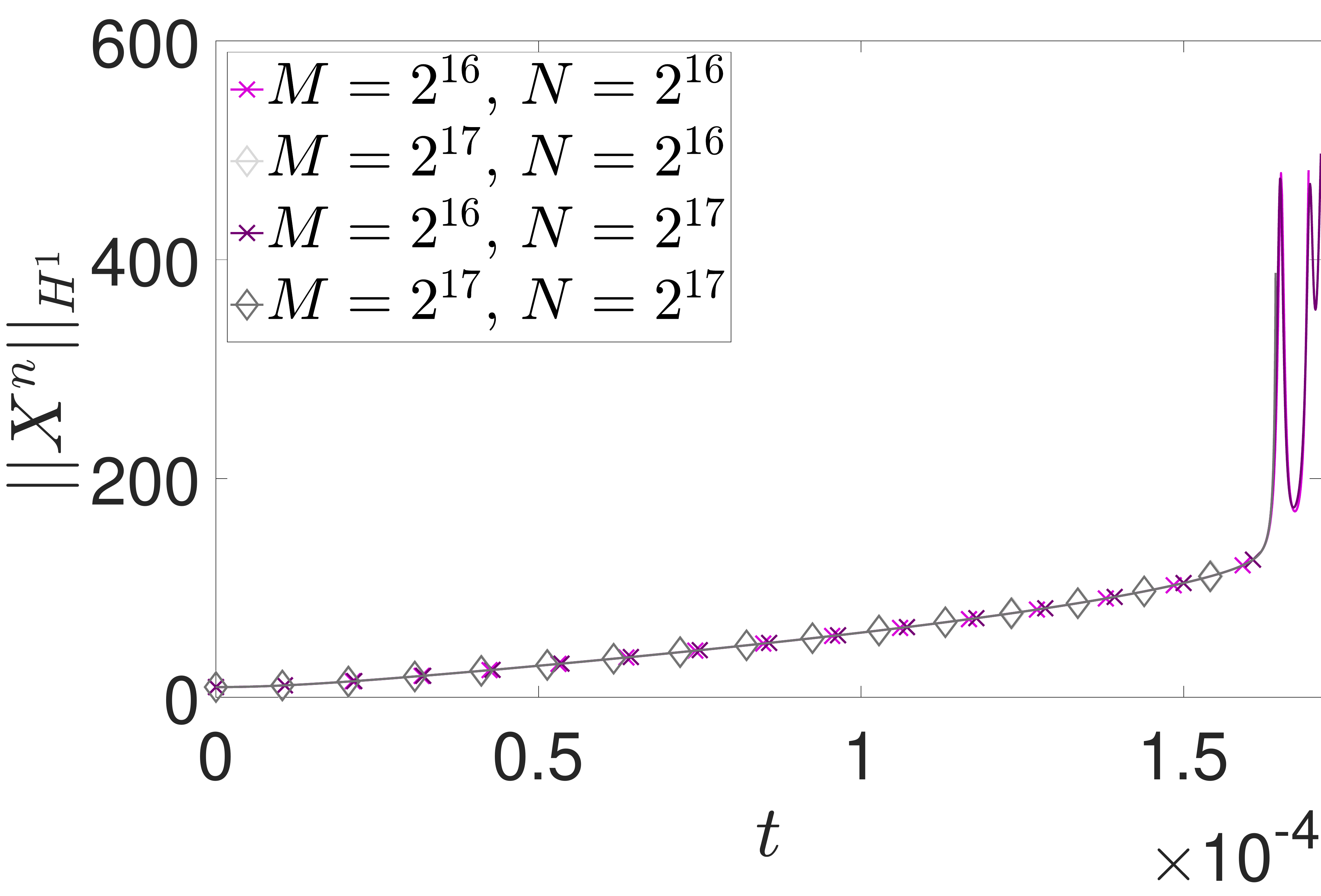}
			\end{subfigure}
			~ 
			\begin{subfigure}[t]{0.35\textwidth}
				\centering
				\includegraphics*[width =\textwidth,keepaspectratio,clip]{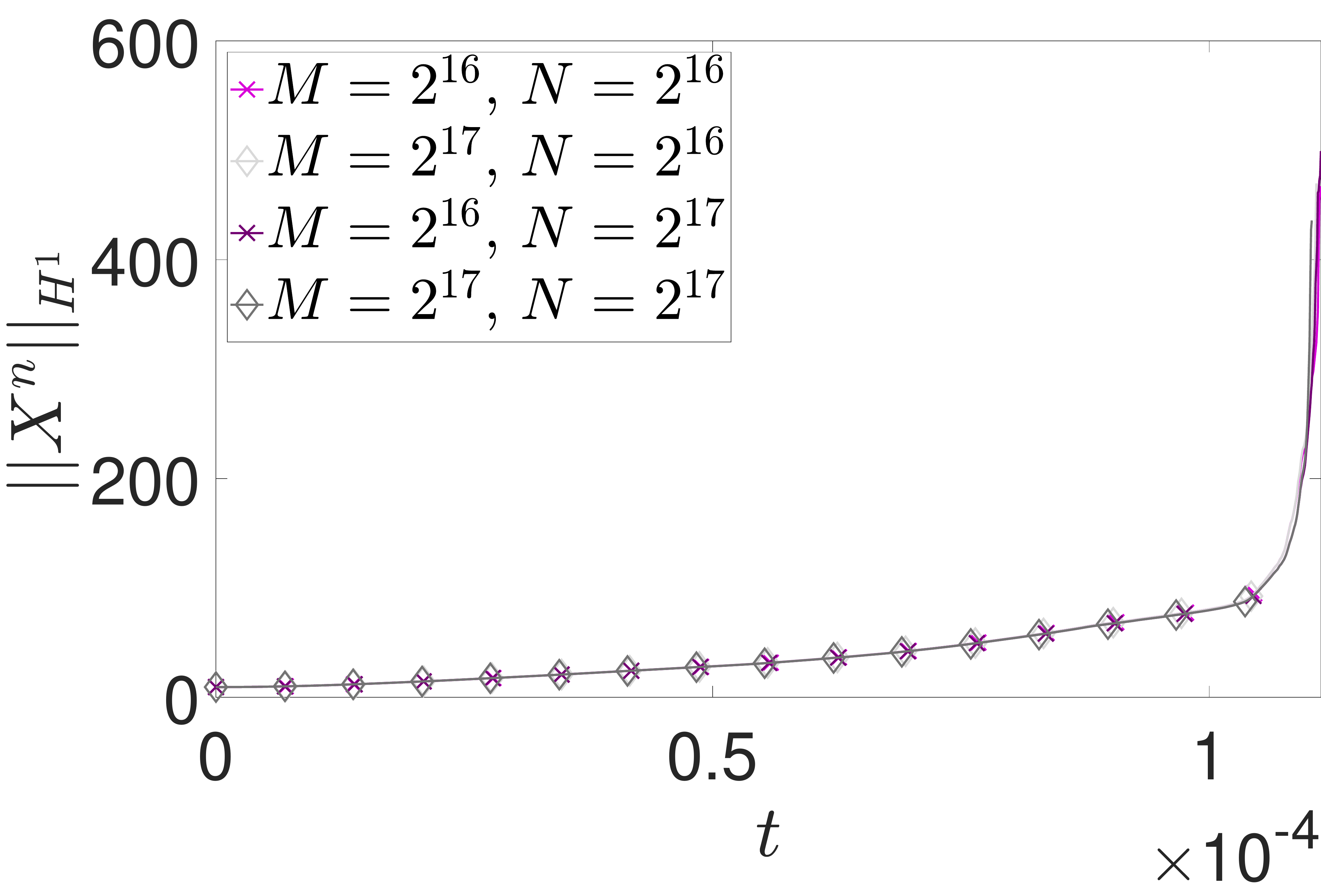}
			\end{subfigure}
			\caption{
				Evolution of $\H^1$-norm of \eqref{eq:ManakovCrit} using the initial value \eqref{IV2},
				$a = 20\pi$,
				$\gamma = 0$ (left) or $\gamma= 1$ (right),
				and $\sigma = 2, 3, 4$ (top, middle, bottom).
				Pink $\times$: $M = 2^{16}$, 
				grey $\diamond$: $M = 2^{17}$.
				Darker lines implies larger $N$.
				\label{fig:H1HigherSigma}
			}
		\end{figure}

		\begin{figure}[h!]
			\centering
			\begin{subfigure}[t]{0.35\textwidth}
				\centering
				\includegraphics*[width =\textwidth,keepaspectratio,clip]{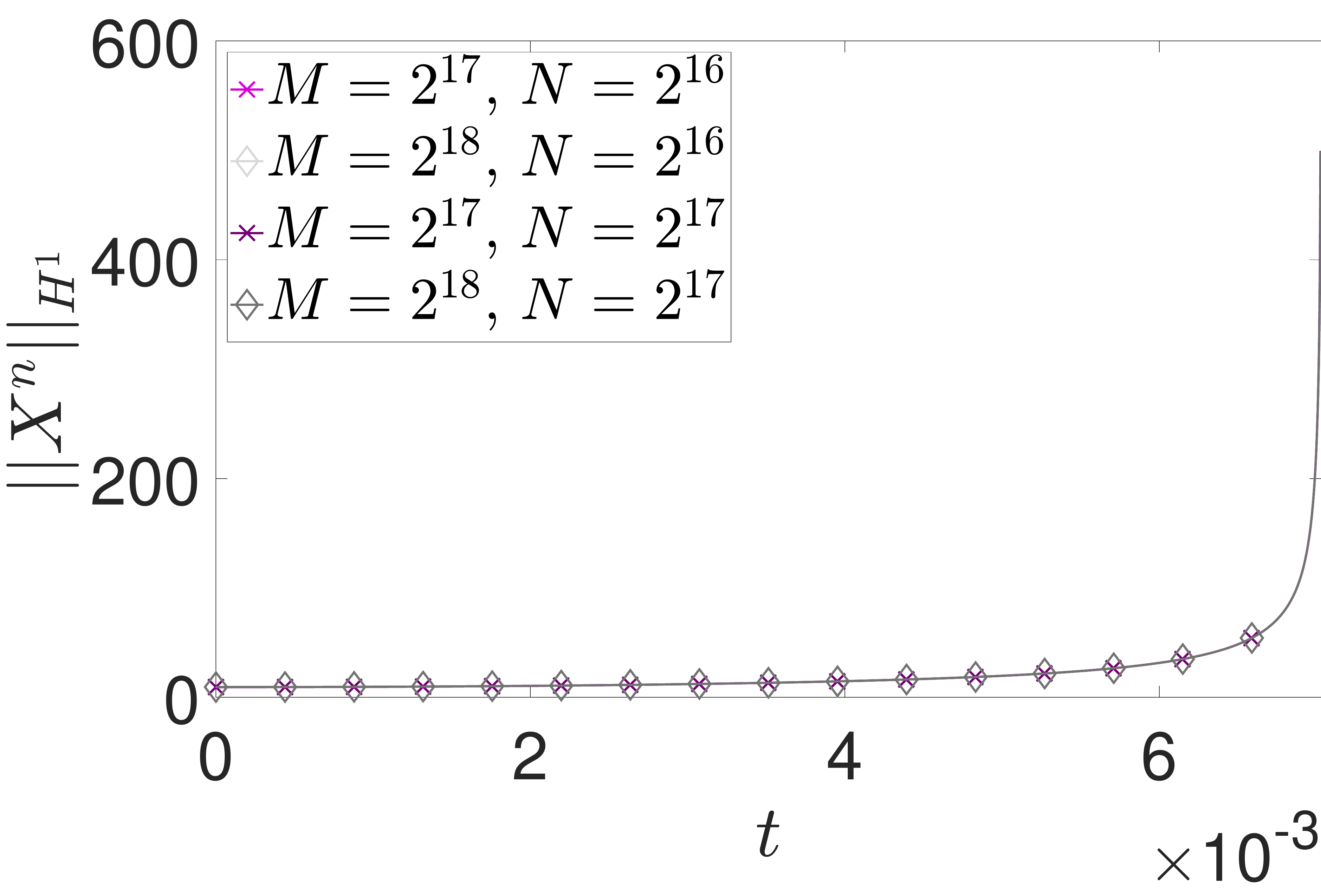}
			\end{subfigure}
			~ 
			\begin{subfigure}[t]{0.35\textwidth}
				\centering
				\includegraphics*[width =\textwidth,keepaspectratio,clip]{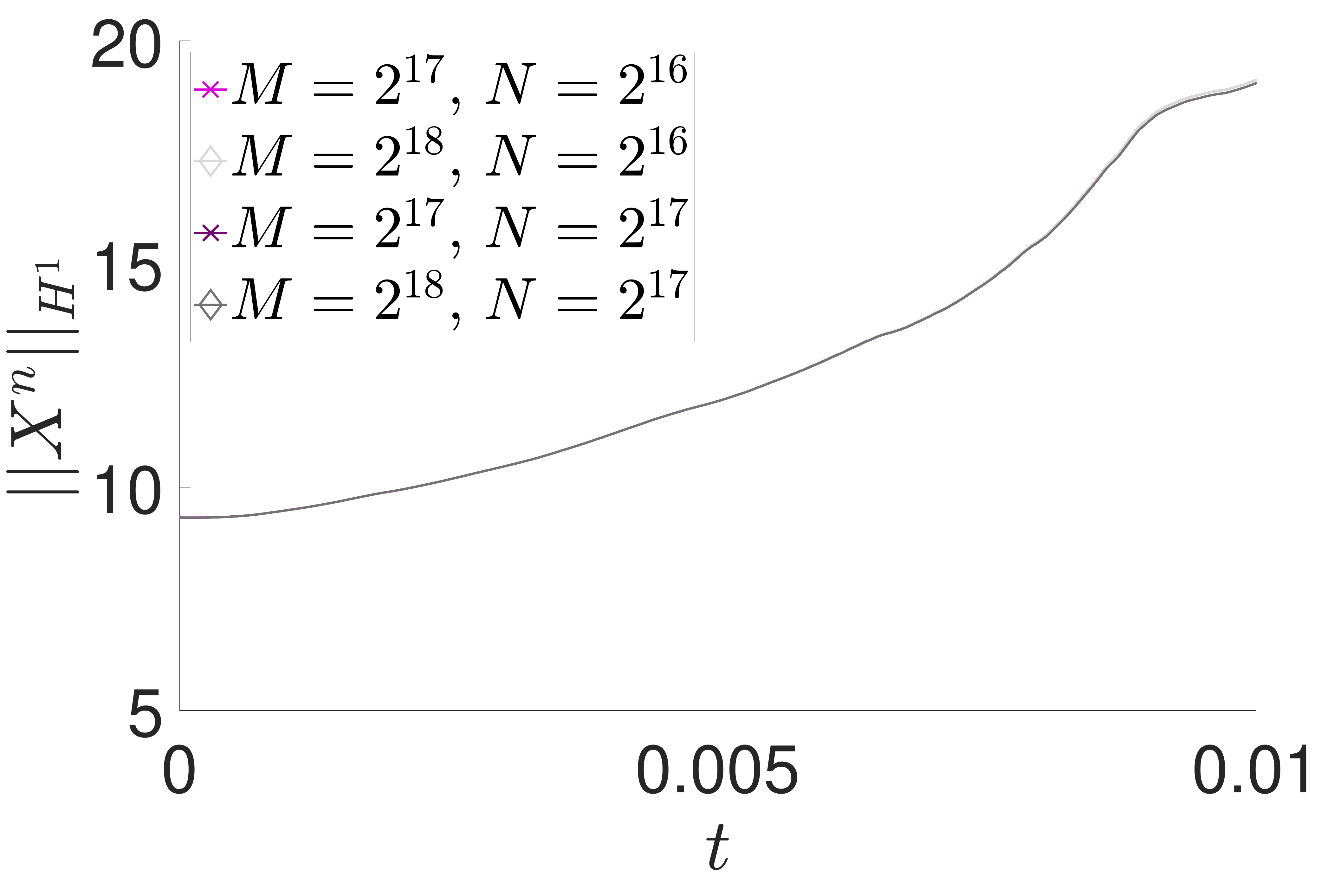}
			\end{subfigure}
			
			\begin{subfigure}[t]{0.35\textwidth}
				\centering
				\includegraphics*[width =\textwidth,keepaspectratio,clip]{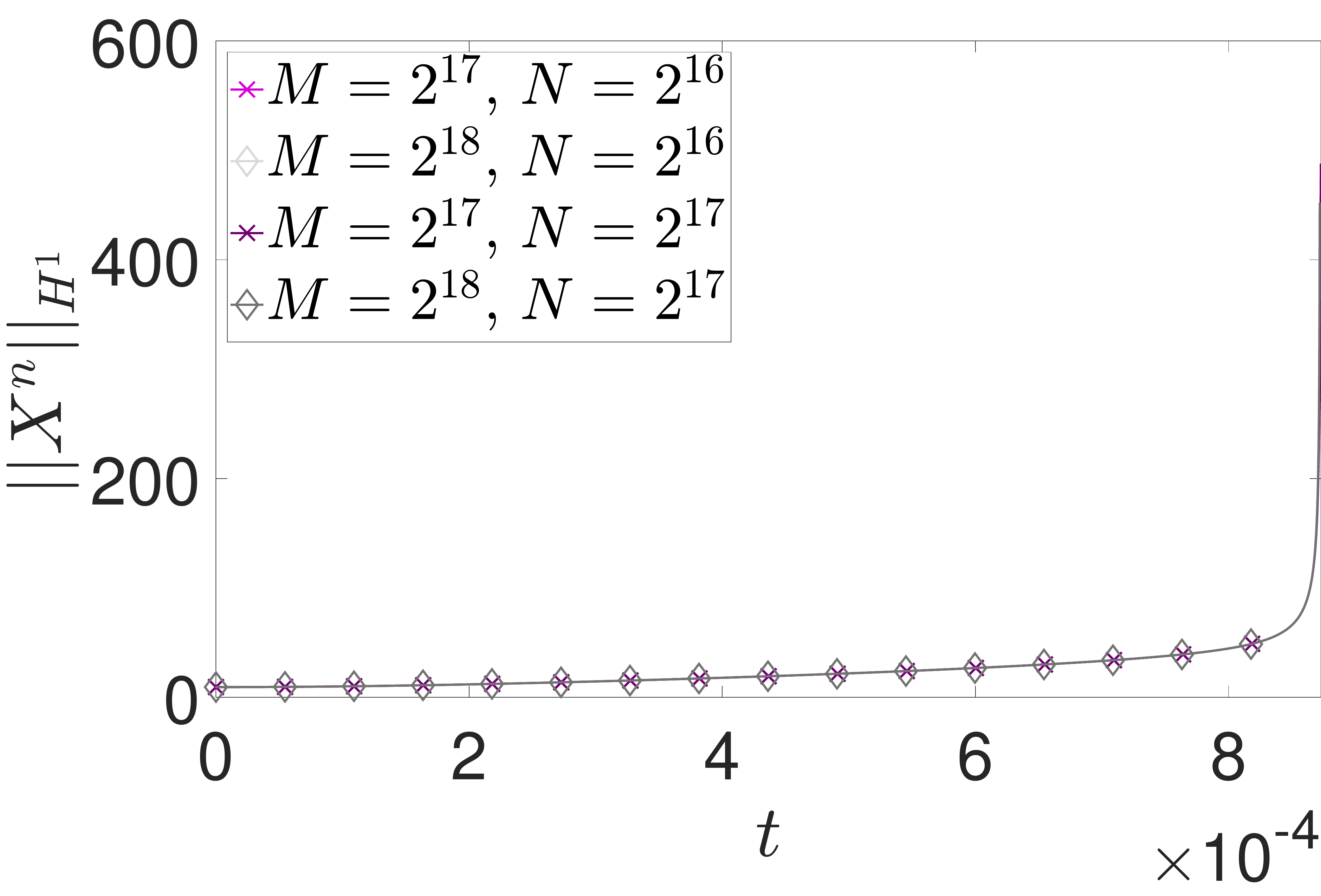}
			\end{subfigure}
			~ 
			\begin{subfigure}[t]{0.35\textwidth}
				\centering
				\includegraphics*[width =\textwidth,keepaspectratio,clip]{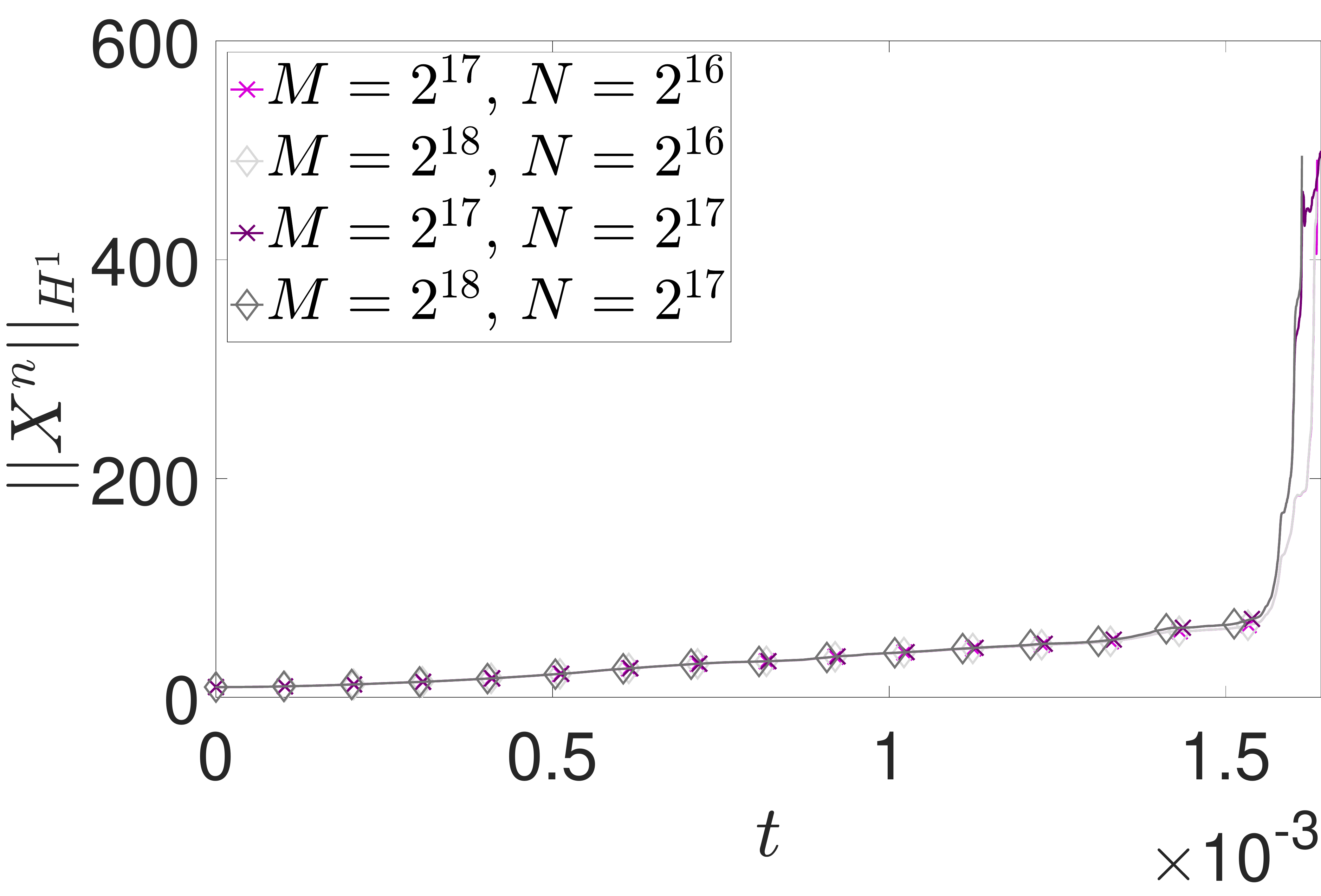}
			\end{subfigure}
			
			\begin{subfigure}[t]{0.35\textwidth}
				\centering
				\includegraphics*[width =\textwidth,keepaspectratio,clip]{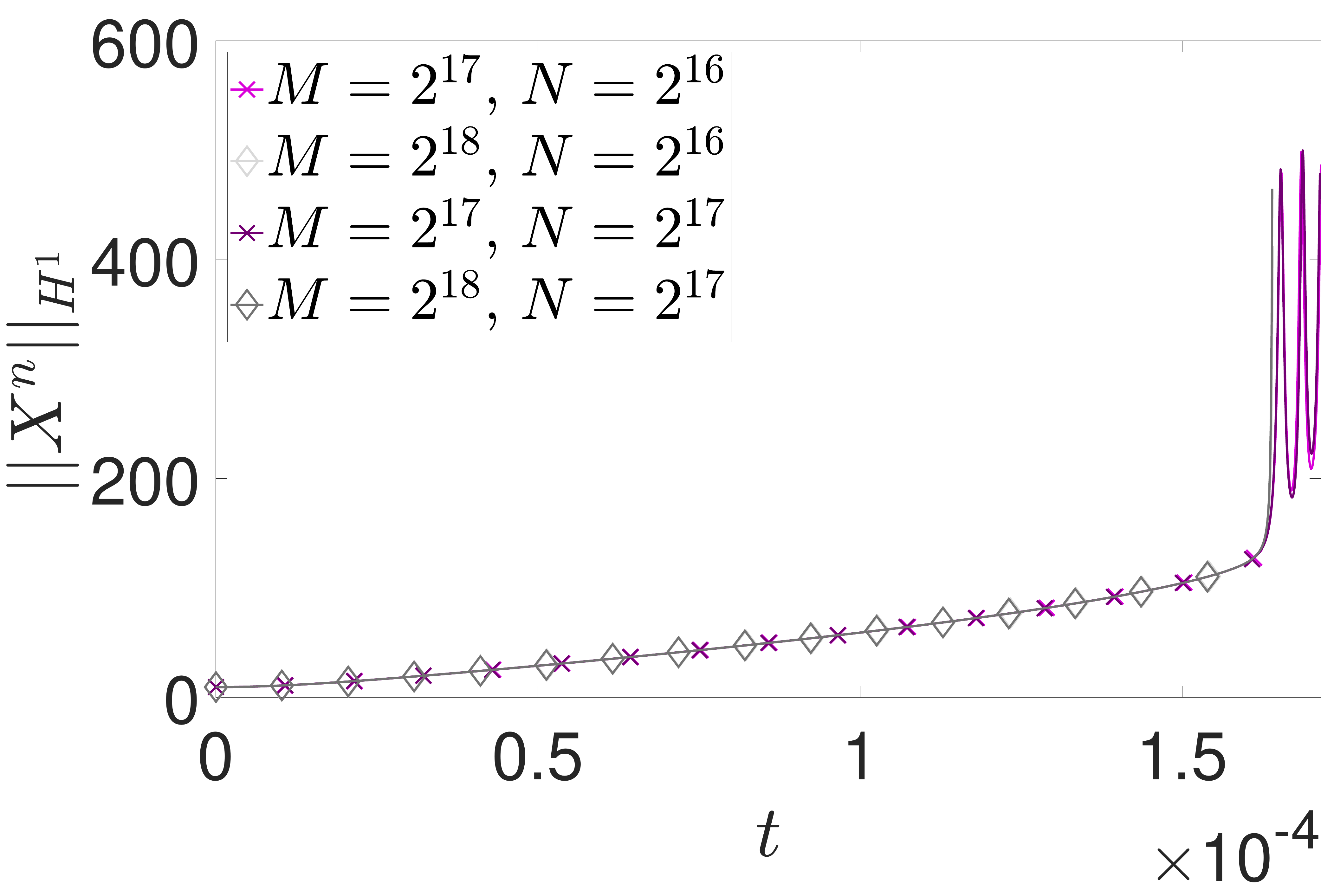}
			\end{subfigure}
			~ 
			\begin{subfigure}[t]{0.35\textwidth}
				\centering
				\includegraphics*[width =\textwidth,keepaspectratio,clip]{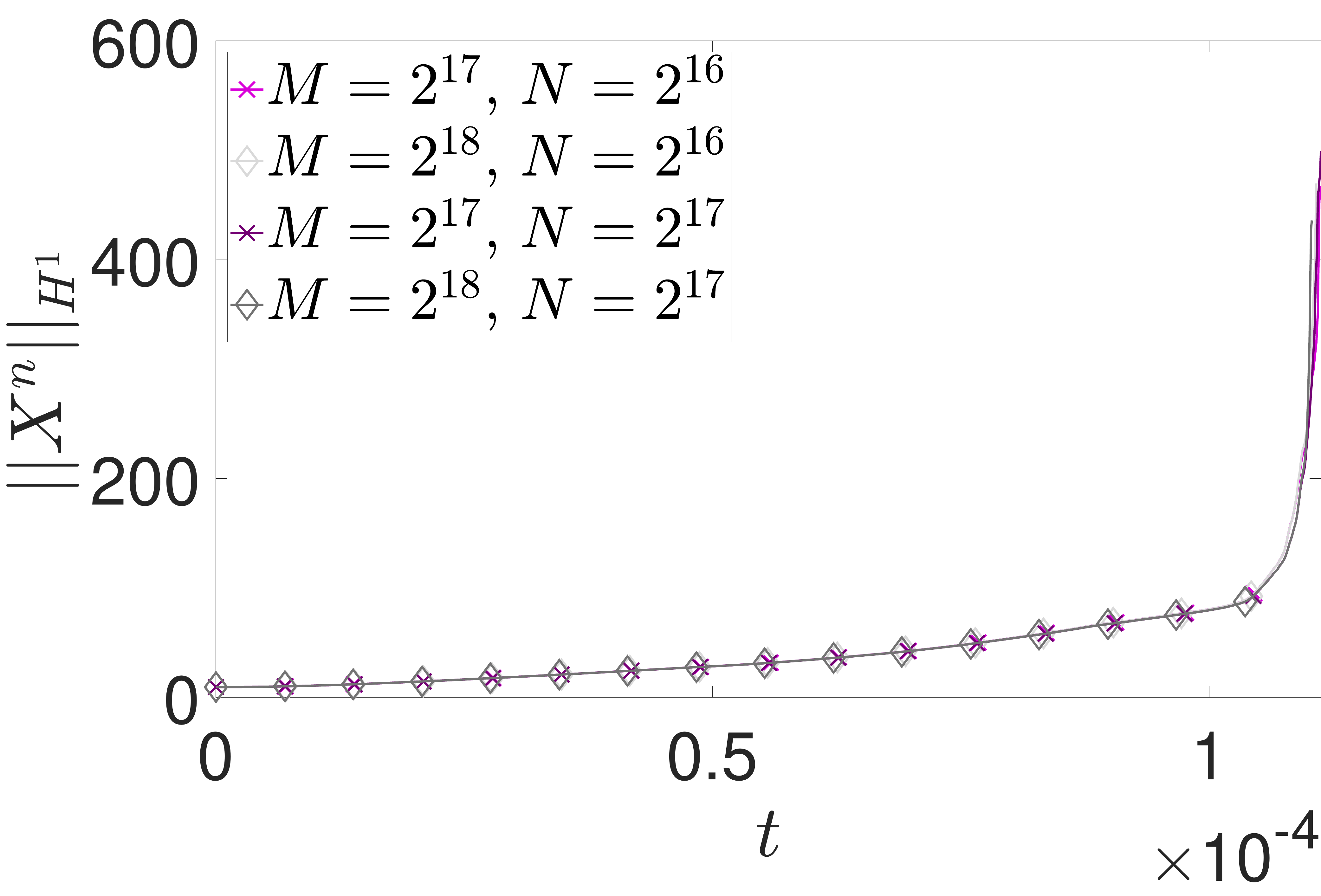}
			\end{subfigure}
			\caption{
				Evolution of $\H^1$-norm of \eqref{eq:ManakovCrit} using the initial value \eqref{IV2},
				$a = 40\pi$,
				$\gamma = 0$ (left) or $\gamma= 1$ (right),
				and $\sigma = 2, 3, 4$ (top, middle, bottom).
				Pink $\times$: $M = 2^{17}$, 
				grey $\diamond$: $M = 2^{18}$.
				Darker lines implies larger $N$.
				\label{fig:H1HigherSigma2L}
			}
		\end{figure}
	
		Having confirmed that blowup may occur, and that the discretization parameters of the previous experiment is sufficient to observe it, we now simulate $48$ samples with $\sigma=1$, $1.25$, $1.5$, $1.75$, $2.$, $2.25$, $2.5$, $2.9$, $3$, $3.5$, $4$ and the parameters
		$T = 0.01$, $N = 2^{17}$, $h = T/N$, $a = 20\pi$, and $M = 2^{17}$ Fourier modes.
		As in the two previous experiments we terminate calculations if $\norm{X^n}_{\H^1} > 500$.
		The results are presented in Figure~\ref{fig:H1new} and Figure~\ref{fig:H1HigherSigmaBatch}. 
		We observe that a rapid increase in the $\H^1$-norm is present for most samples. For $\sigma=2.5$, we suspect that the spatial discretisation may perhaps not be fine enough to reach blowup. 
		We also observe that, as expected, blowup for each sample is reached earlier for larger $\sigma$.

				\begin{figure}[h!]
			\centering
			\begin{subfigure}[t]{0.3\textwidth}
				\centering
				\includegraphics*[width =\textwidth,keepaspectratio,clip]{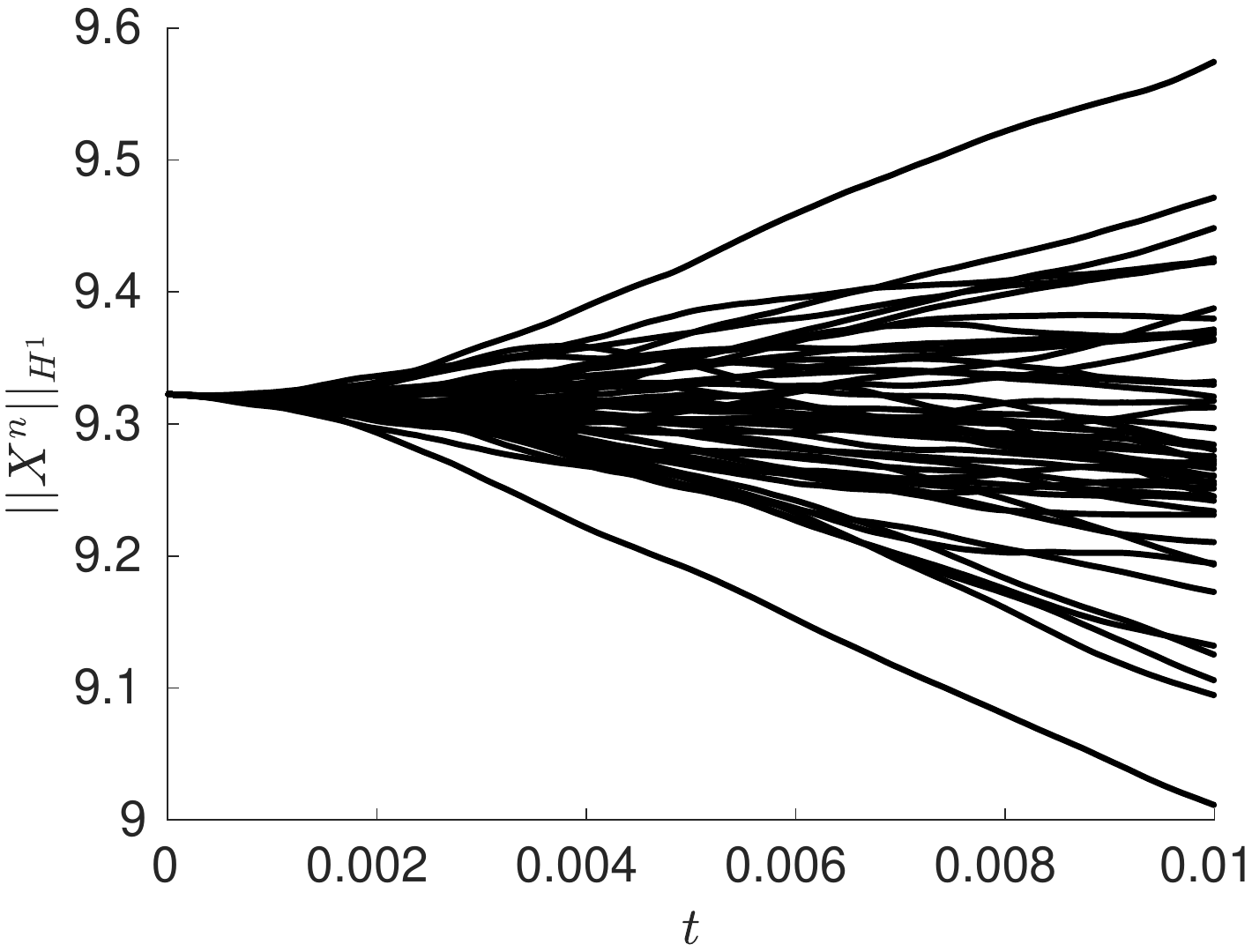}
			\end{subfigure}
			~ 
			\begin{subfigure}[t]{0.3\textwidth}
				\centering
				\includegraphics*[width =\textwidth,keepaspectratio,clip]{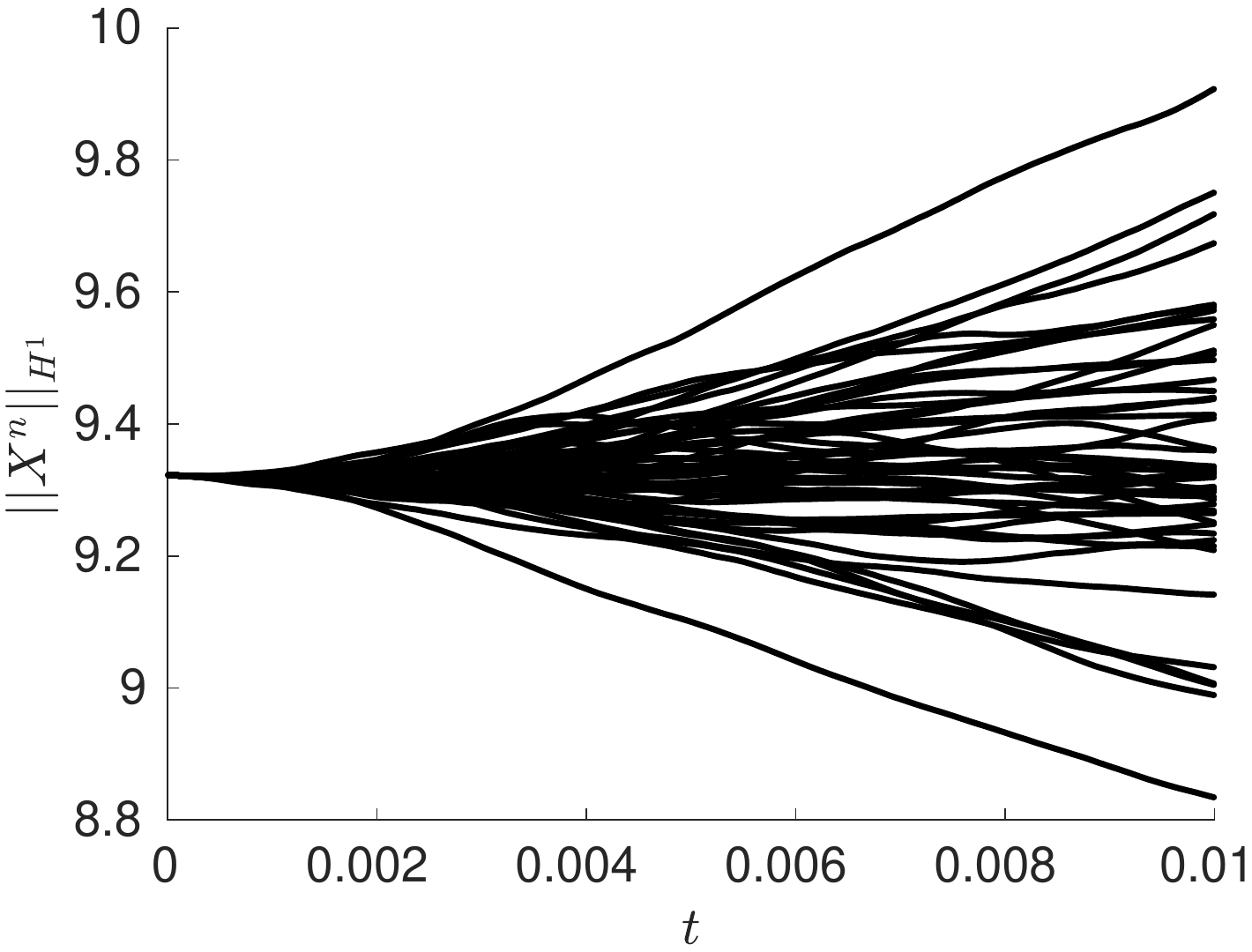}
			\end{subfigure}
			
			\begin{subfigure}[t]{0.3\textwidth}
				\centering
				\includegraphics*[width =\textwidth,keepaspectratio,clip]{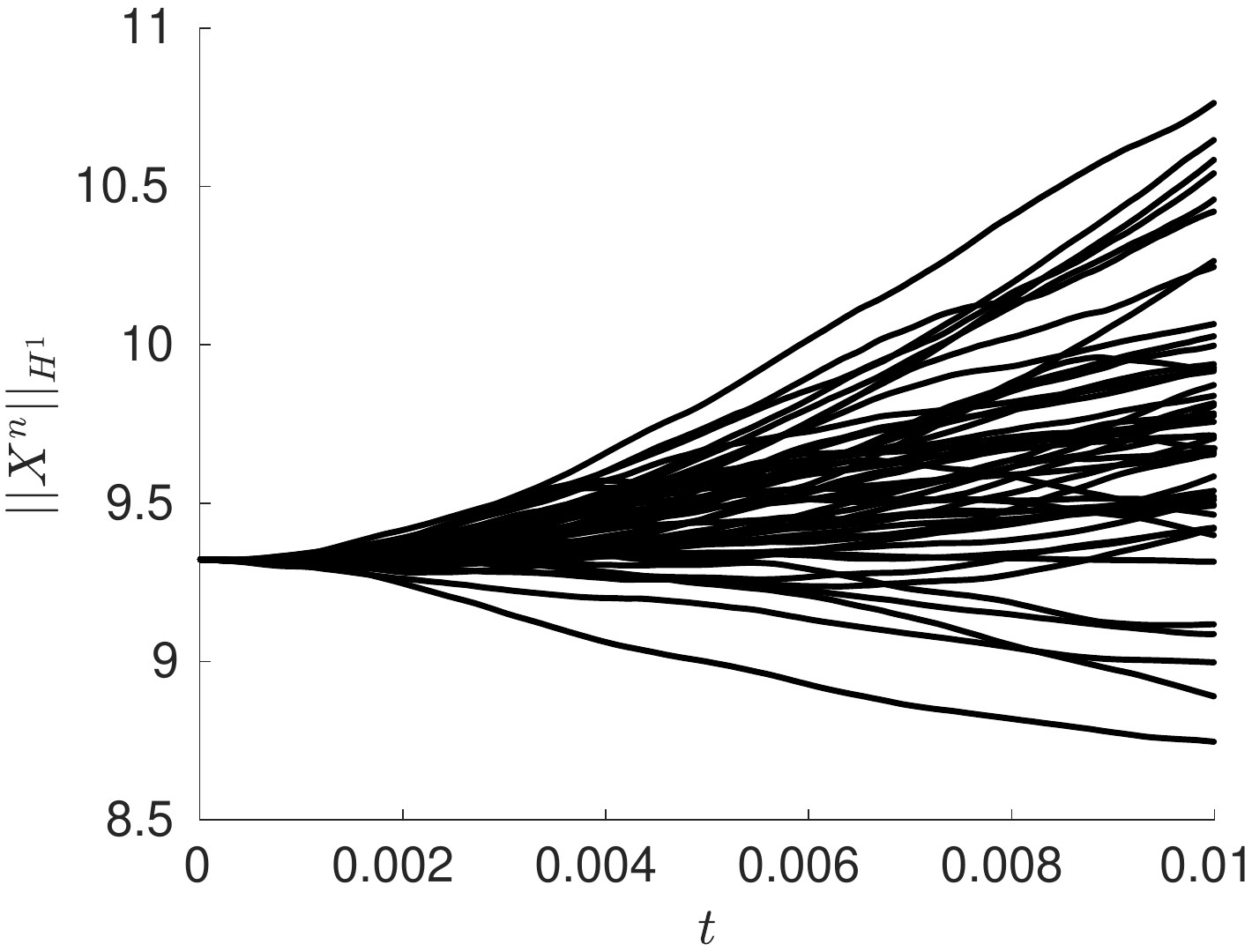}
			\end{subfigure}
			~ 
			\begin{subfigure}[t]{0.3\textwidth}
				\centering
				\includegraphics*[width =\textwidth,keepaspectratio,clip]{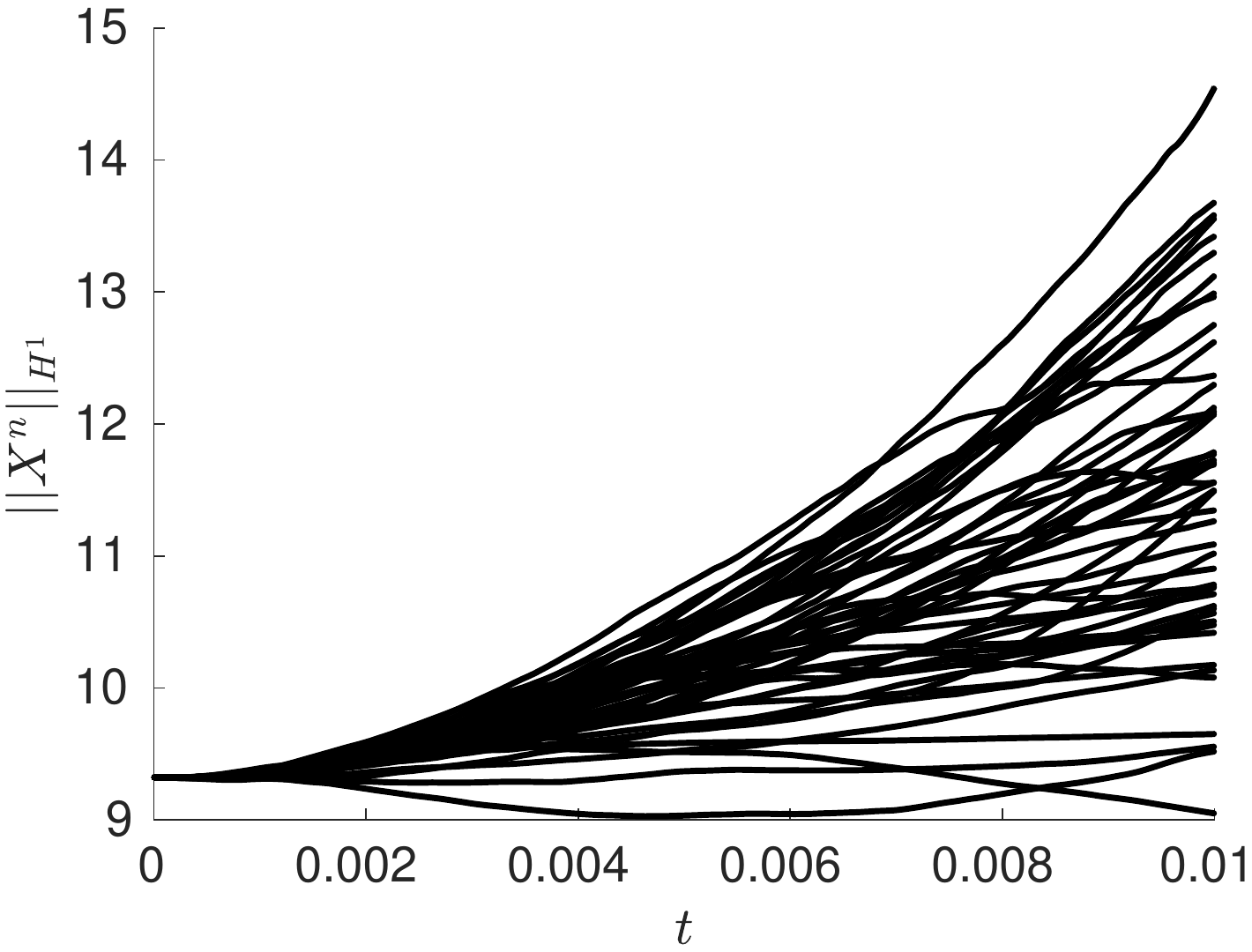}
			\end{subfigure}
			
			\begin{subfigure}[t]{0.3\textwidth}
				\centering
				\includegraphics*[width =\textwidth,keepaspectratio,clip]{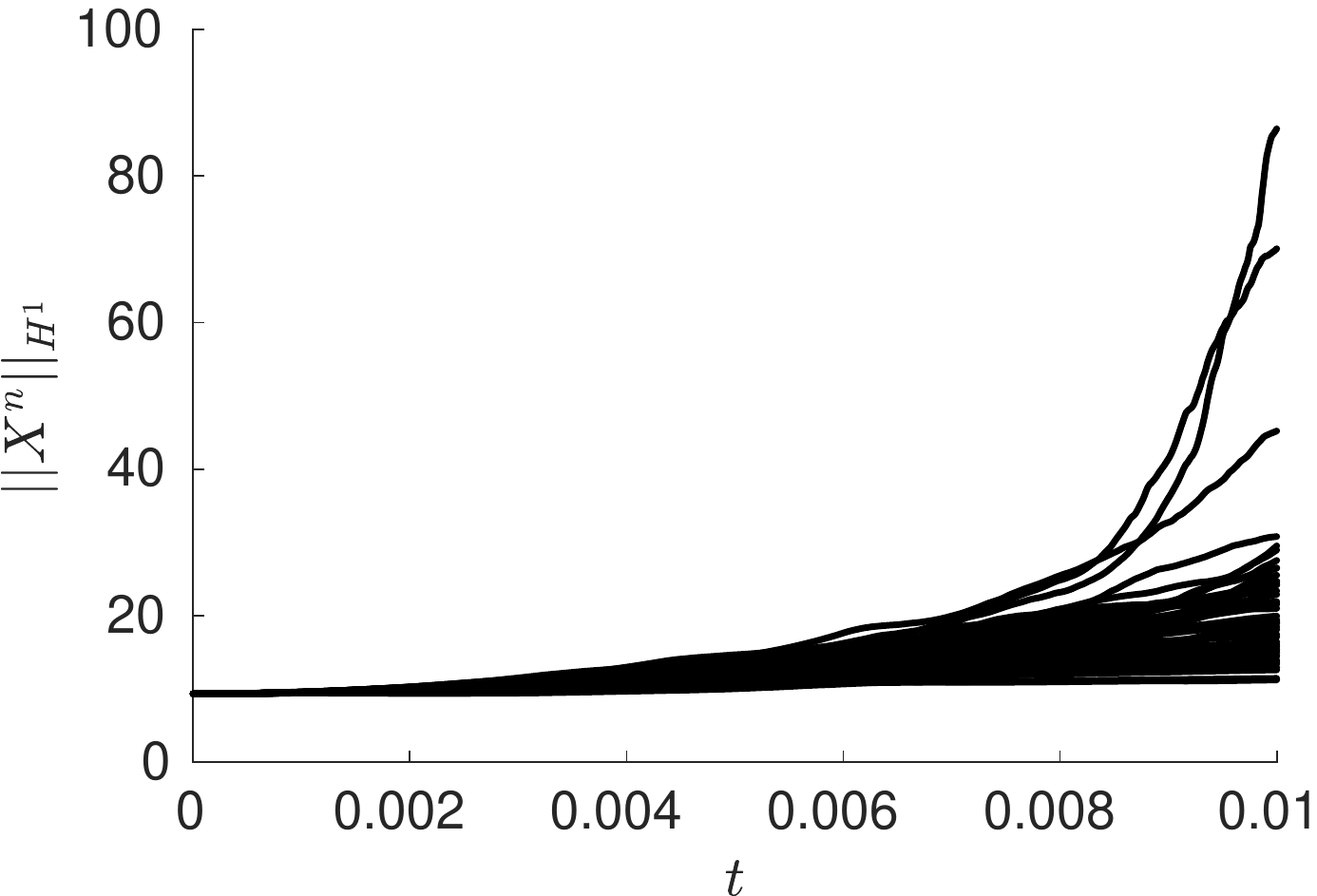}
			\end{subfigure}
			~ 
			\begin{subfigure}[t]{0.3\textwidth}
				\centering
				\includegraphics*[width =\textwidth,keepaspectratio,clip]{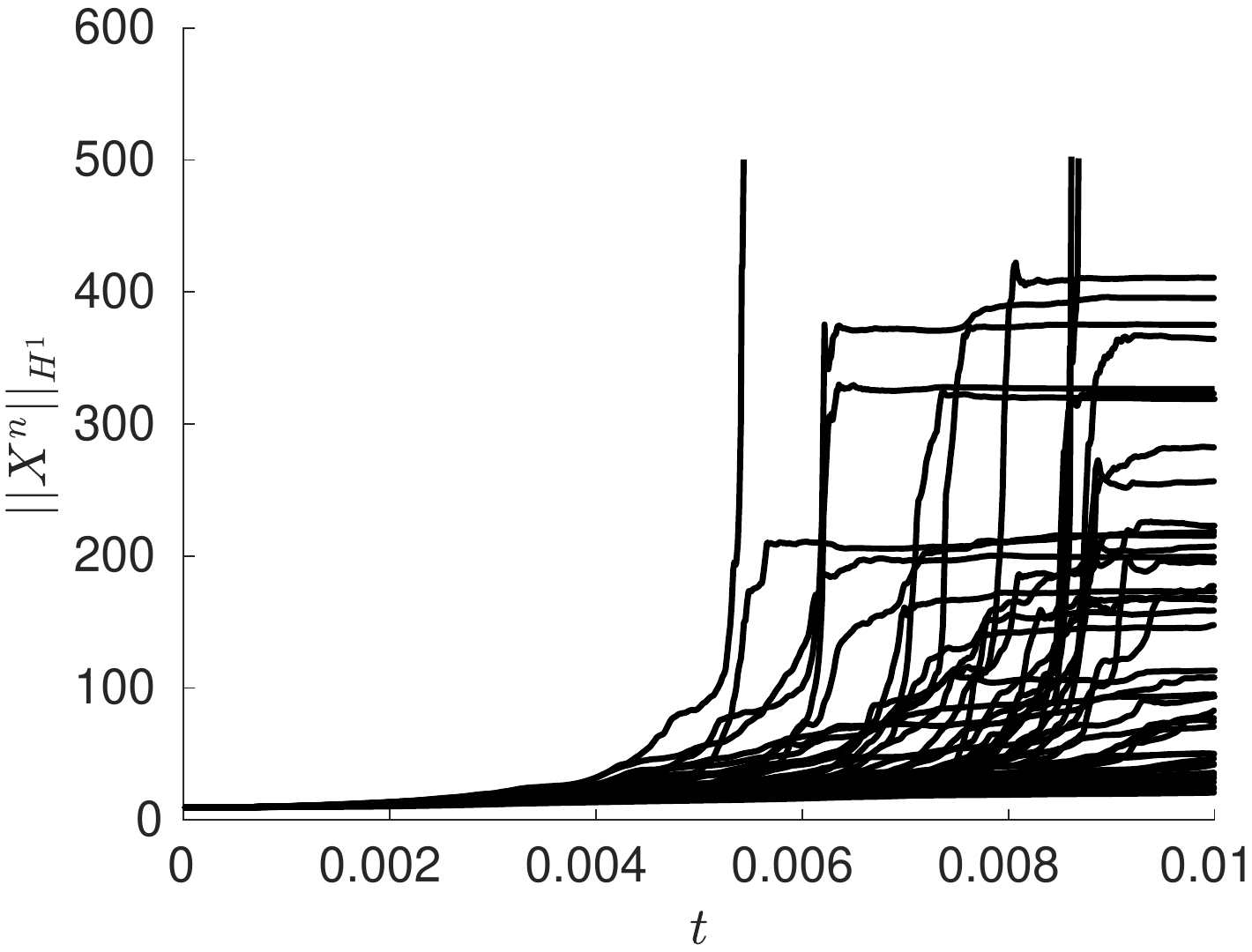}
			\end{subfigure}
			\caption{
				Evolution of $\H^1$-norms of \eqref{eq:ManakovCrit} with $\sigma = 1, 1.25, 1.5, 1.75, 2., 2.25$ (in order left to right, top to bottom) using the initial value \eqref{IV2}.
				\label{fig:H1new}
			}
		\end{figure}

				\begin{figure}[h!]
			\centering
			\begin{subfigure}[t]{0.35\textwidth}
				\centering
				\includegraphics*[width =\textwidth,keepaspectratio,clip]{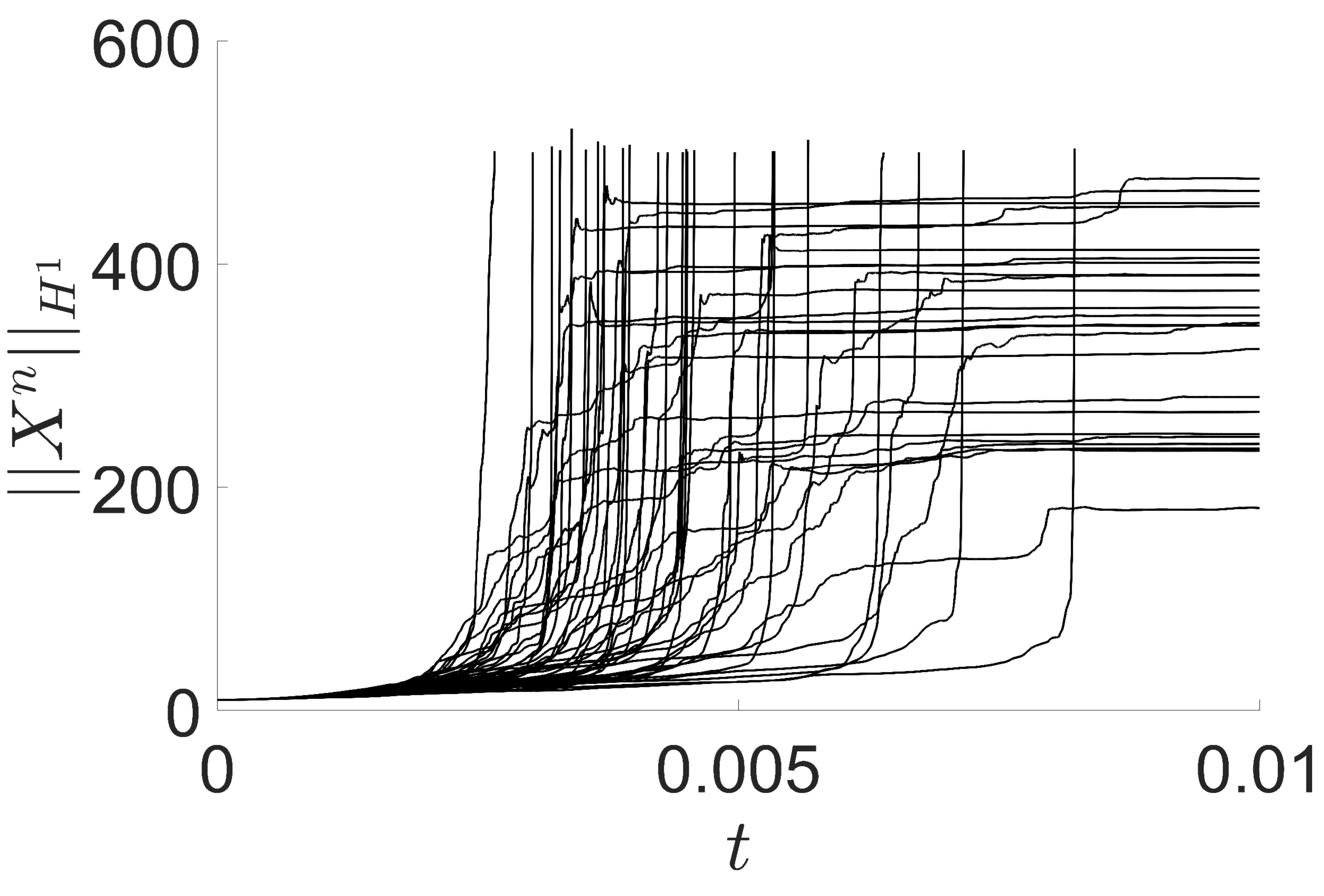}
			\end{subfigure}
			~ 
			\begin{subfigure}[t]{0.35\textwidth}
				\centering
				\includegraphics*[width =\textwidth,keepaspectratio,clip]{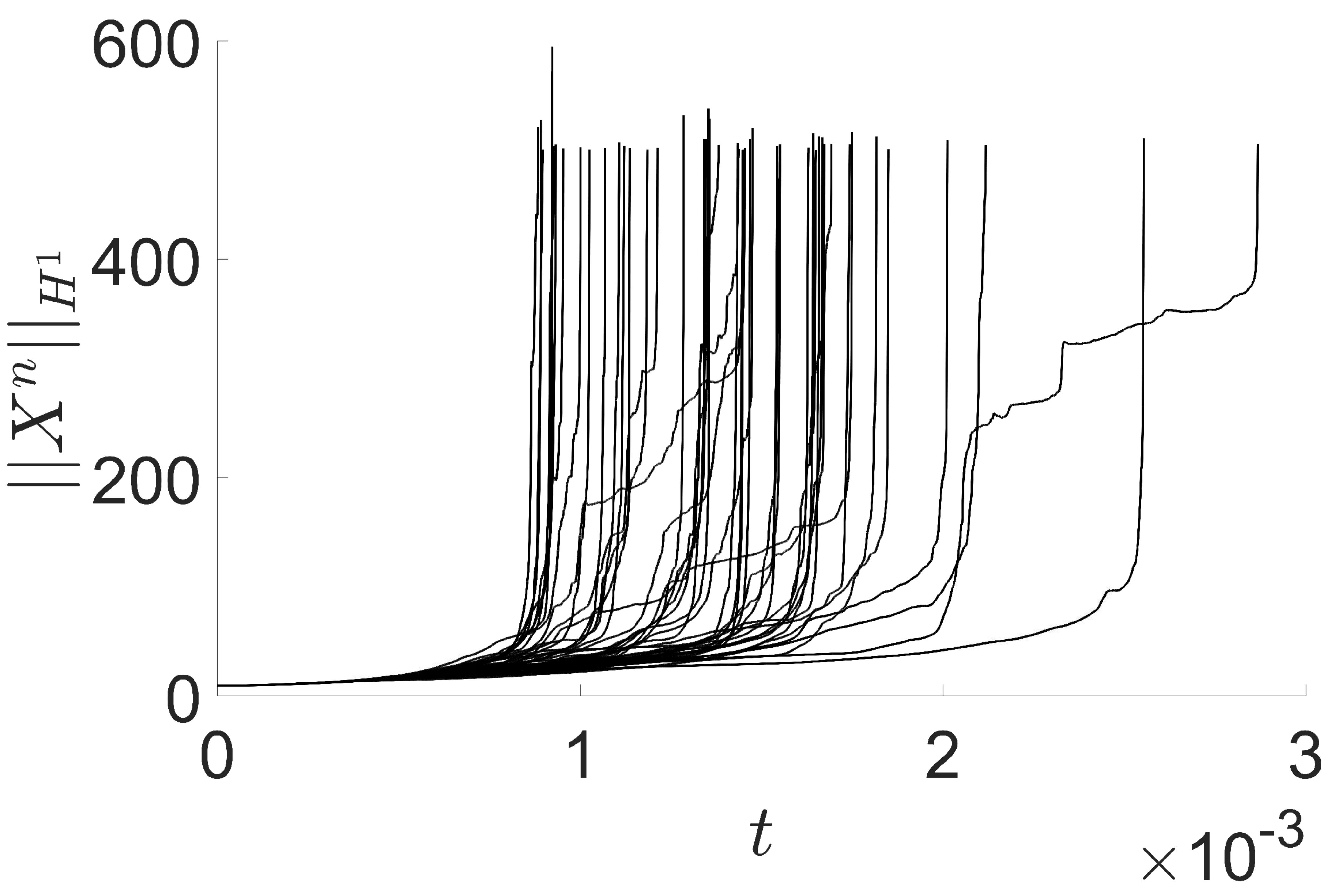}
			\end{subfigure}
			
			\begin{subfigure}[t]{0.35\textwidth}
				\centering
				\includegraphics*[width =\textwidth,keepaspectratio,clip]{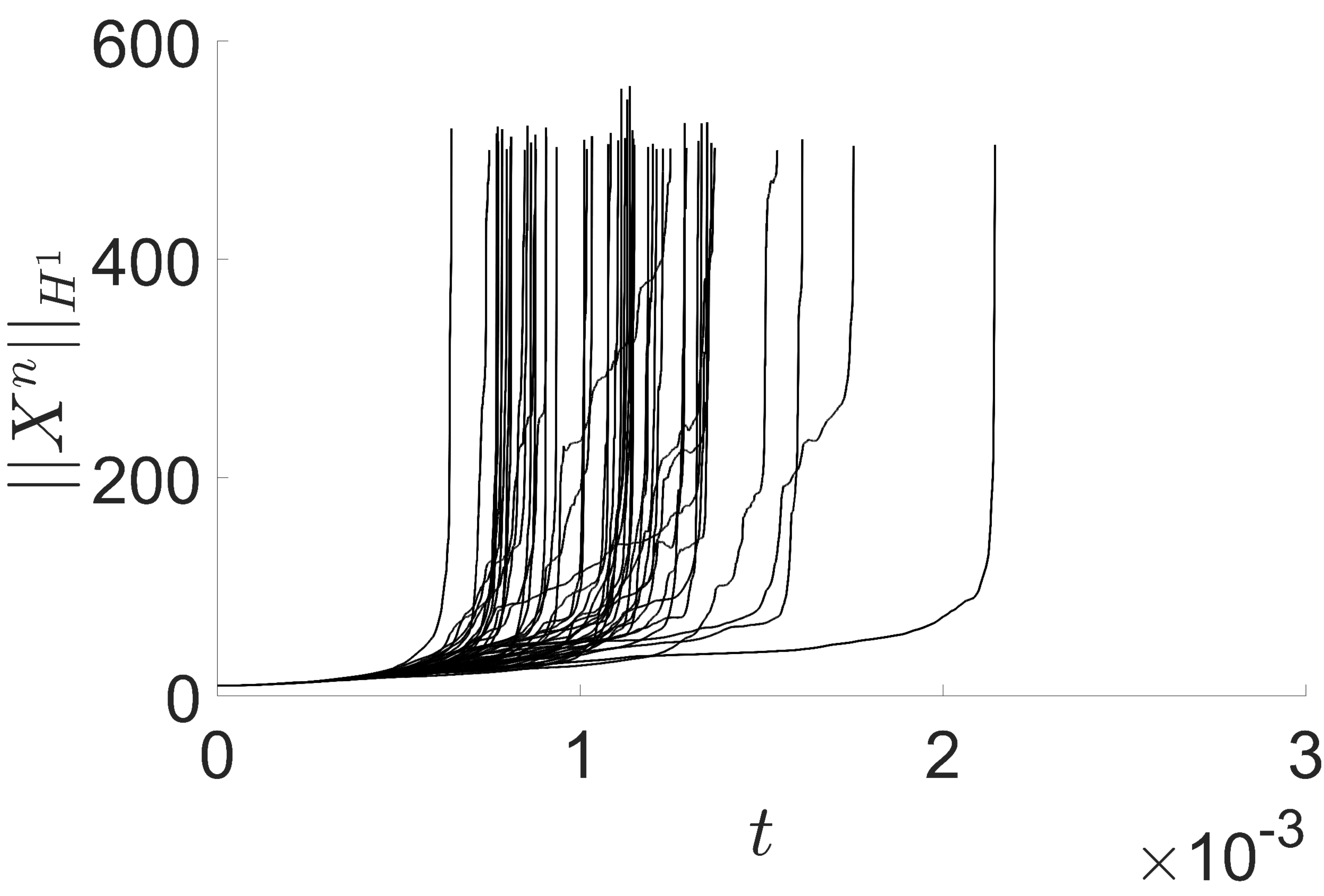}
			\end{subfigure}
			~ 
			\begin{subfigure}[t]{0.35\textwidth}
				\centering
				\includegraphics*[width =\textwidth,keepaspectratio,clip]{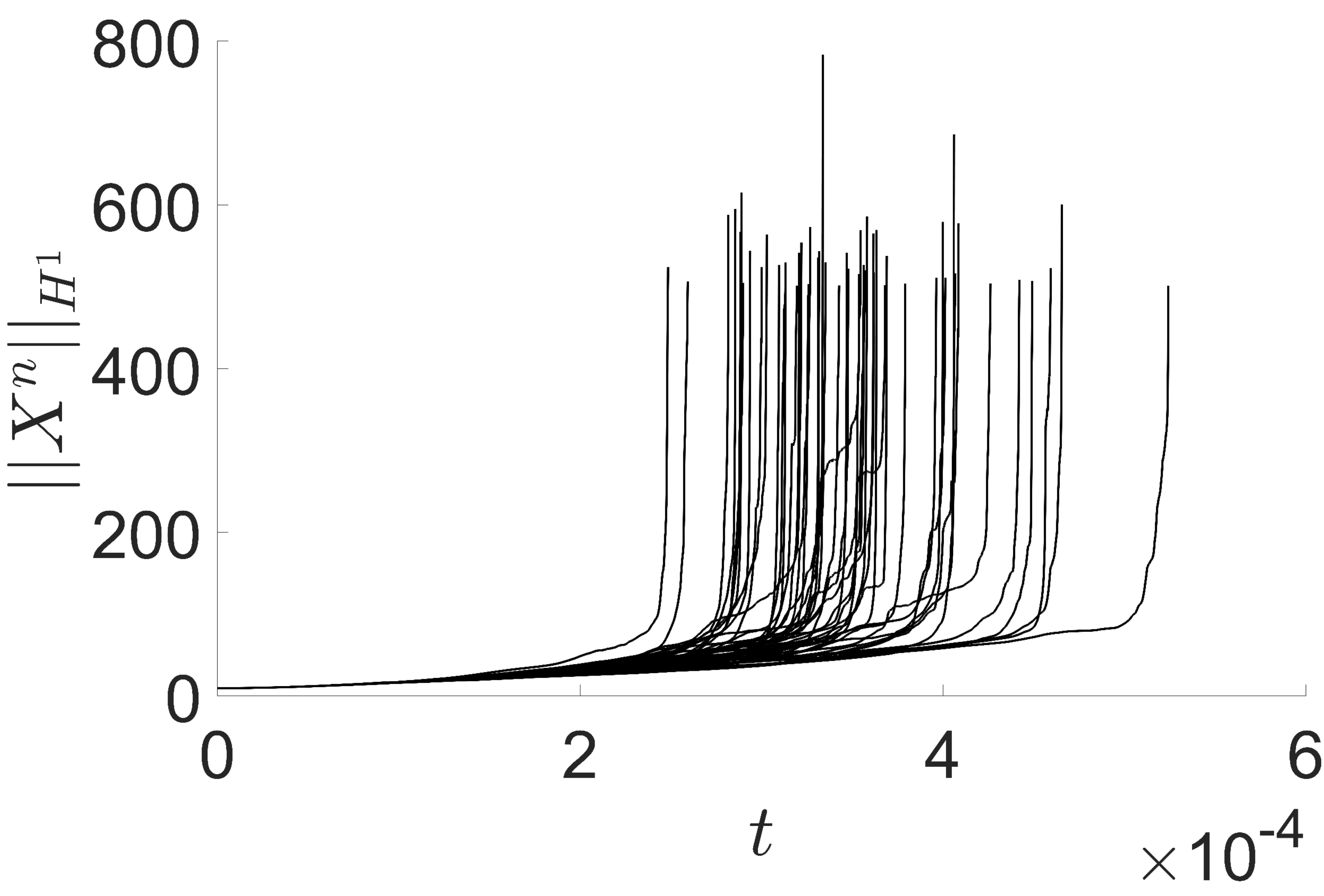}
			\end{subfigure}
			
			\begin{subfigure}[t]{0.35\textwidth}
				\centering
				\includegraphics*[width =\textwidth,keepaspectratio,clip]{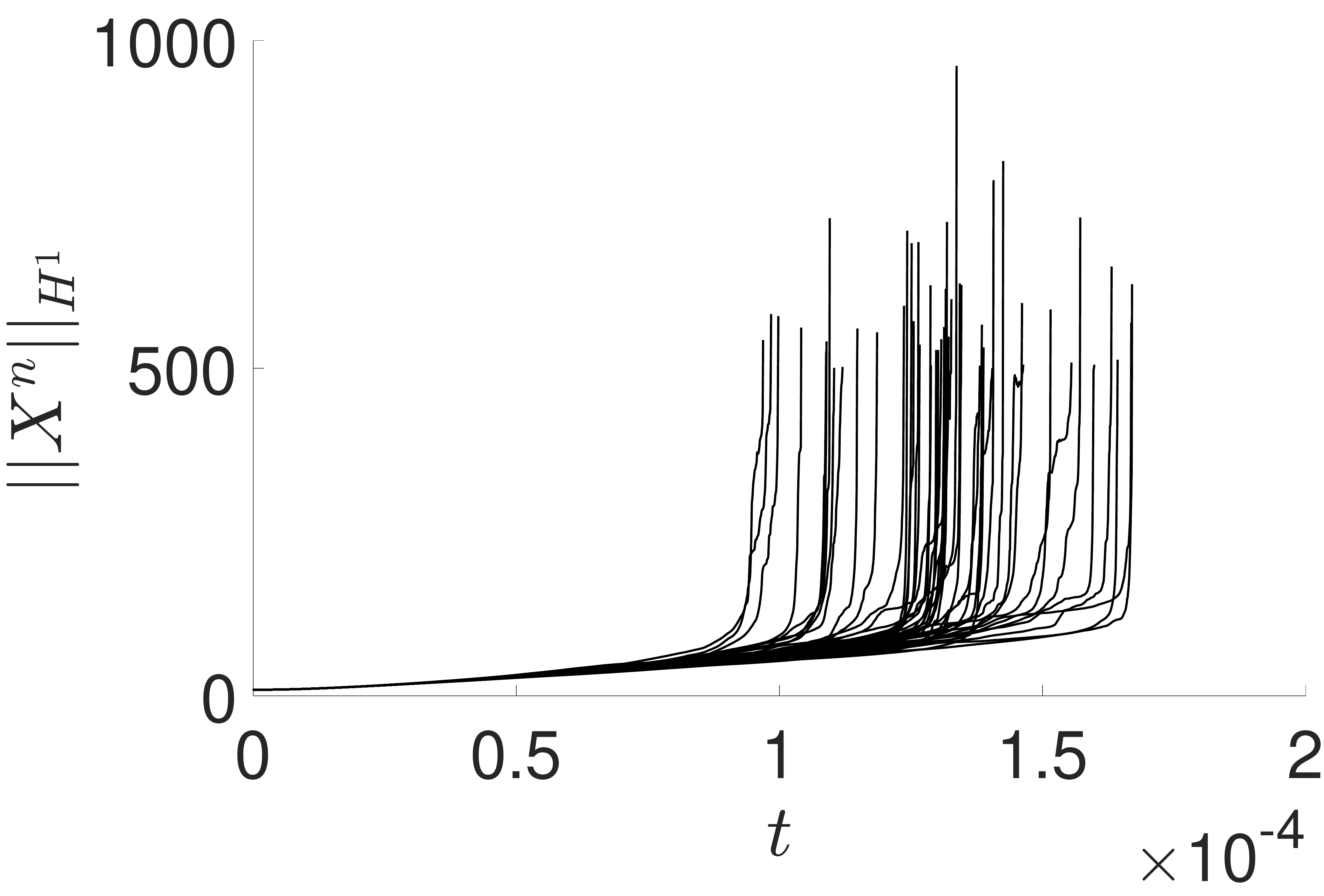}
			\end{subfigure}
			\caption{
				Evolution of $\H^1$-norms of \eqref{eq:ManakovCrit} with $\sigma = 2.5, 2.9, 3, 3.5, 4$ (in order left to right, top to bottom) using the initial value \eqref{IV2}.
				\label{fig:H1HigherSigmaBatch}
			}
		\end{figure}
		
		The results for the deterministic problem ($\gamma=0$) are presented in Figure~\ref{fig:detH1new} (with a focus on values 
		of $\sigma$ between $1$ and $2.25$).

            \begin{figure}[h!]
			\centering
			\begin{subfigure}[t]{0.3\textwidth}
				\centering
				\includegraphics*[width =\textwidth,keepaspectratio,clip]{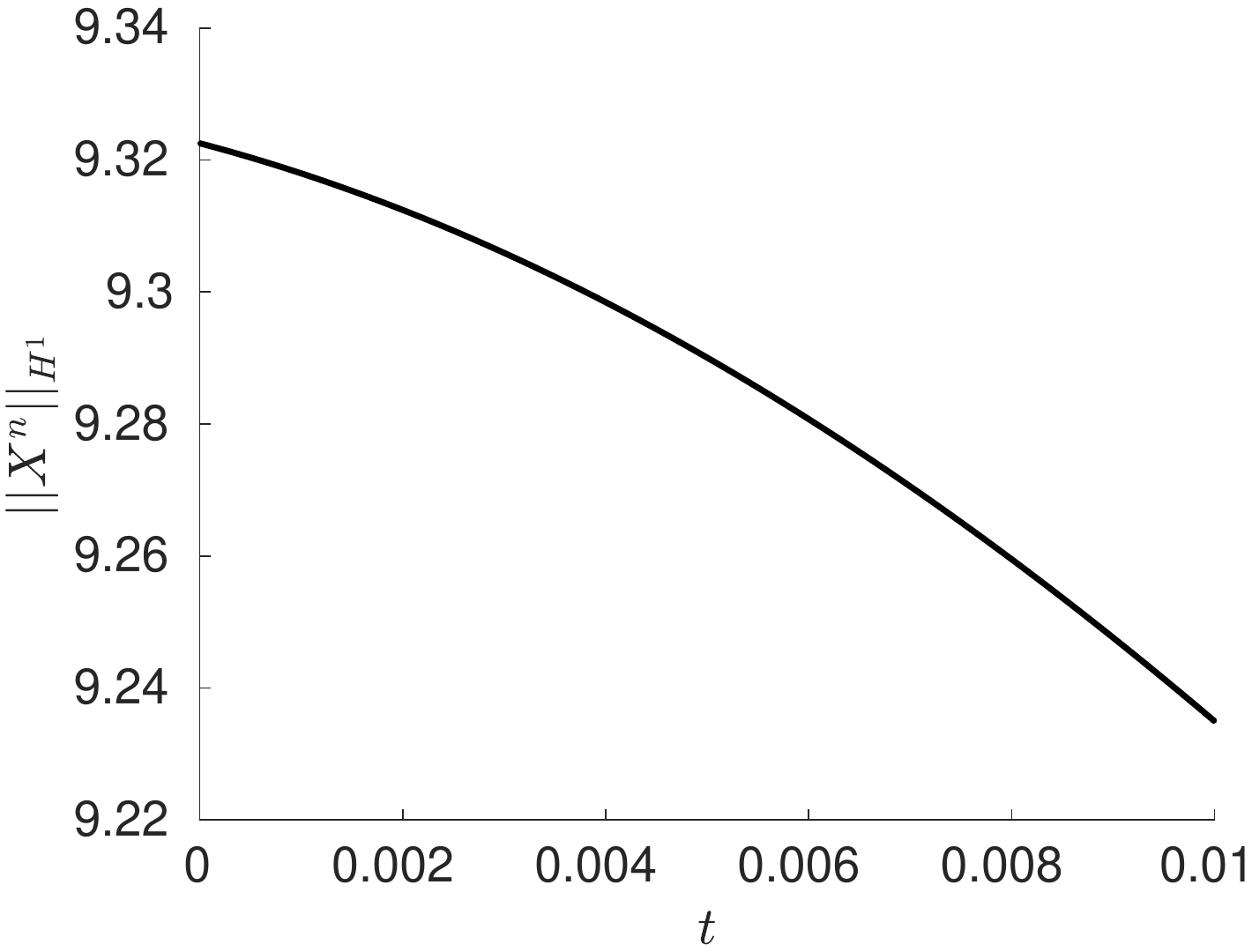}
			\end{subfigure}
			~ 
			\begin{subfigure}[t]{0.3\textwidth}
				\centering
				\includegraphics*[width =\textwidth,keepaspectratio,clip]{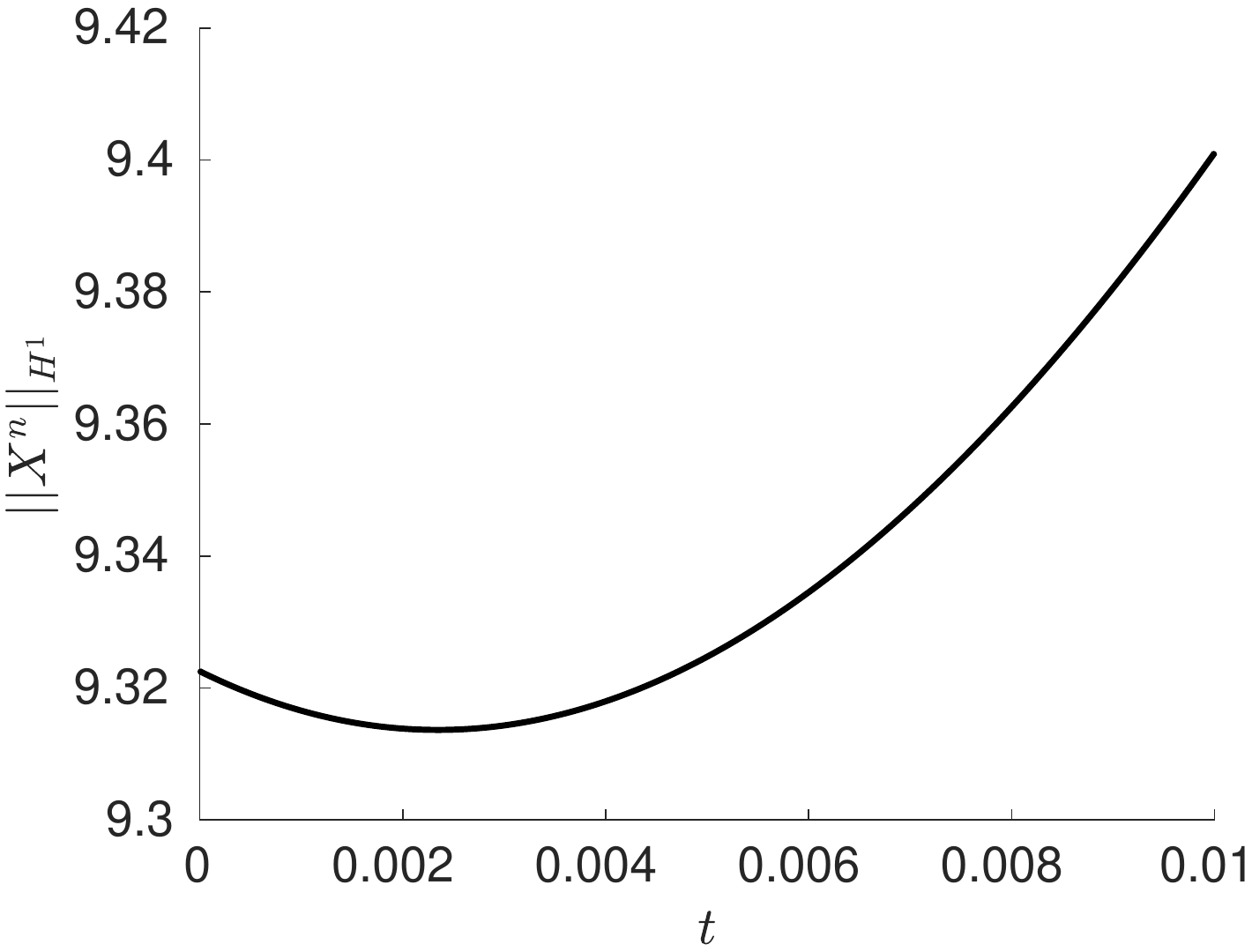}
			\end{subfigure}
			
			\begin{subfigure}[t]{0.3\textwidth}
				\centering
				\includegraphics*[width =\textwidth,keepaspectratio,clip]{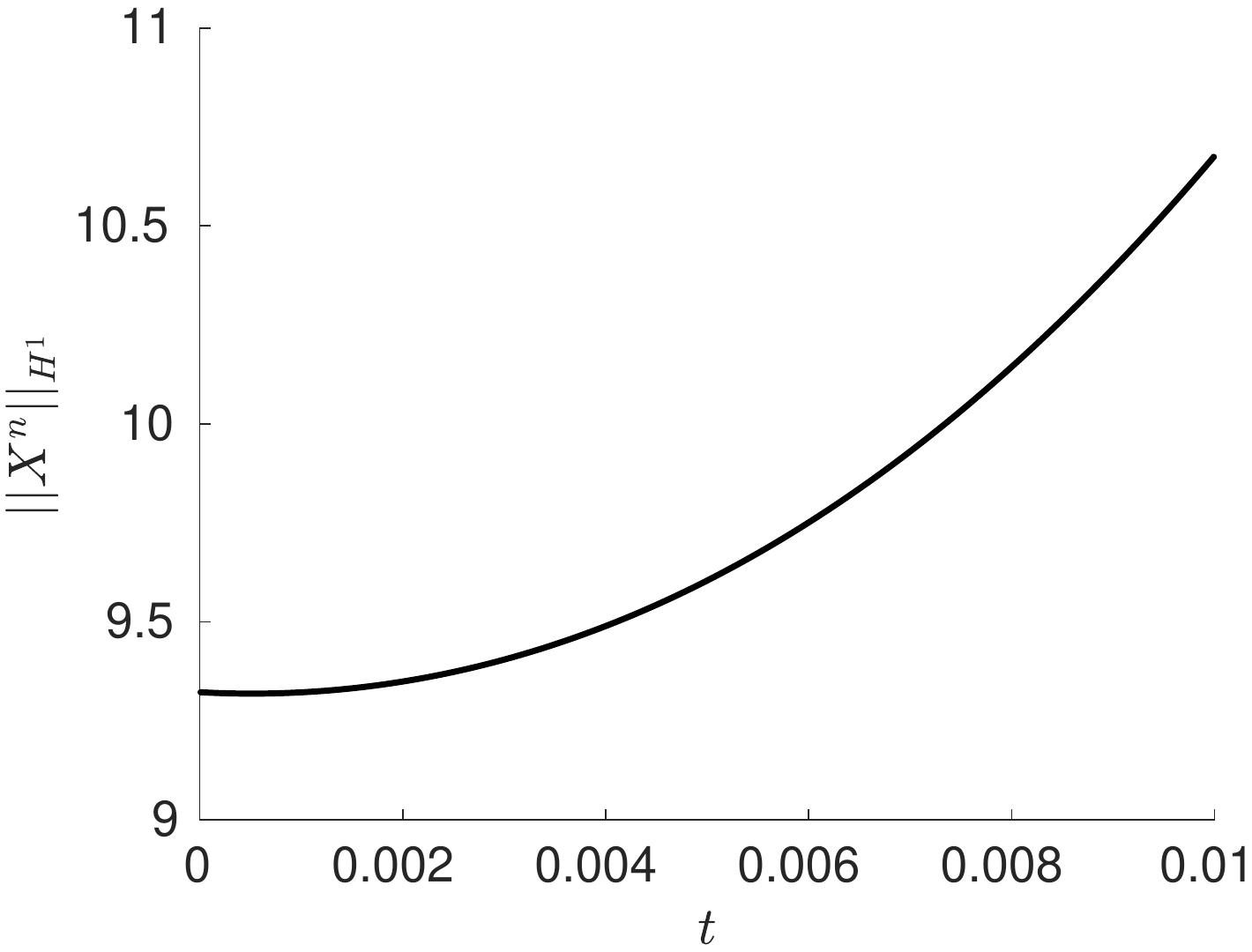}
			\end{subfigure}
			~ 
			\begin{subfigure}[t]{0.3\textwidth}
				\centering
				\includegraphics*[width =\textwidth,keepaspectratio,clip]{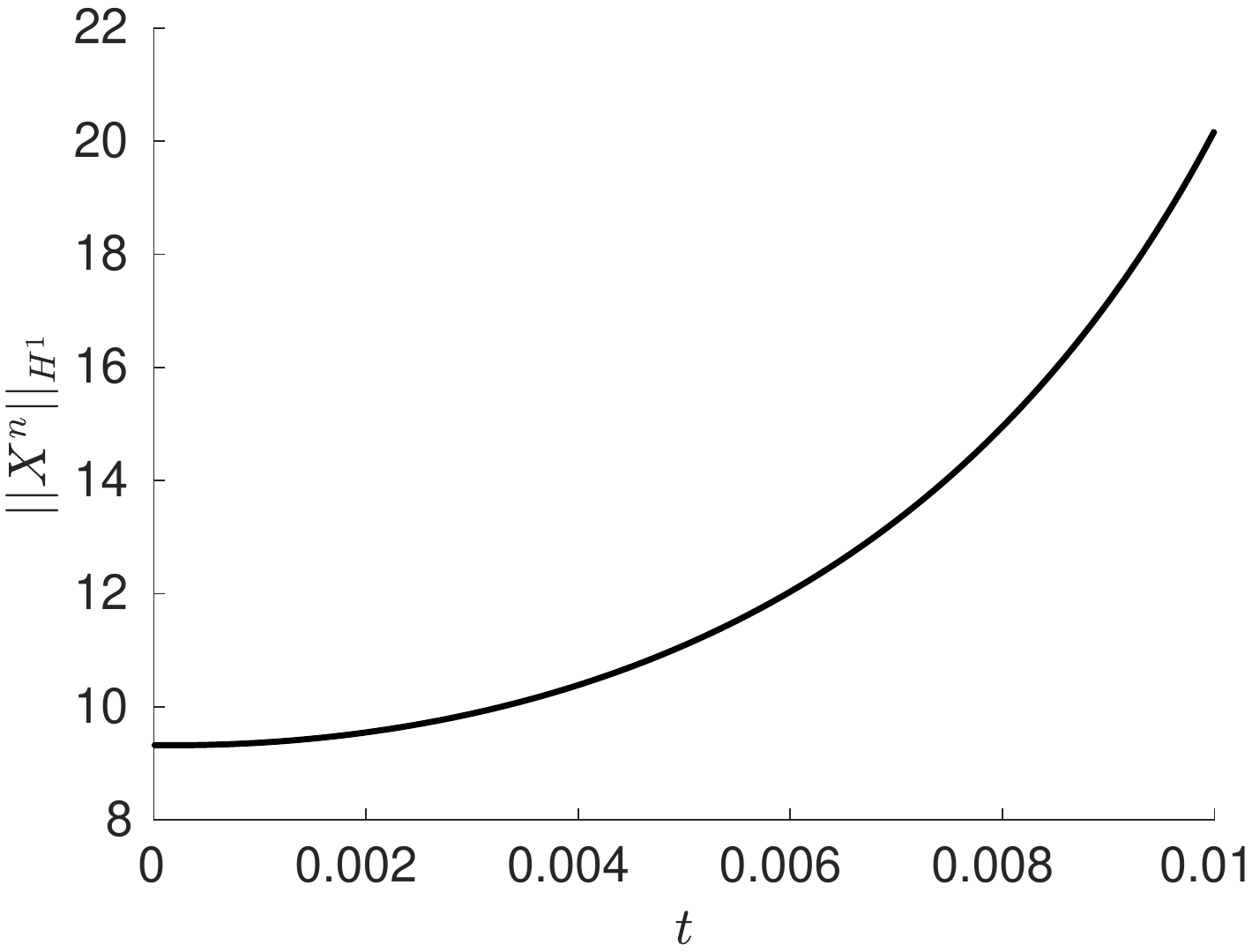}
			\end{subfigure}
			
			\begin{subfigure}[t]{0.3\textwidth}
				\centering
				\includegraphics*[width =\textwidth,keepaspectratio,clip]{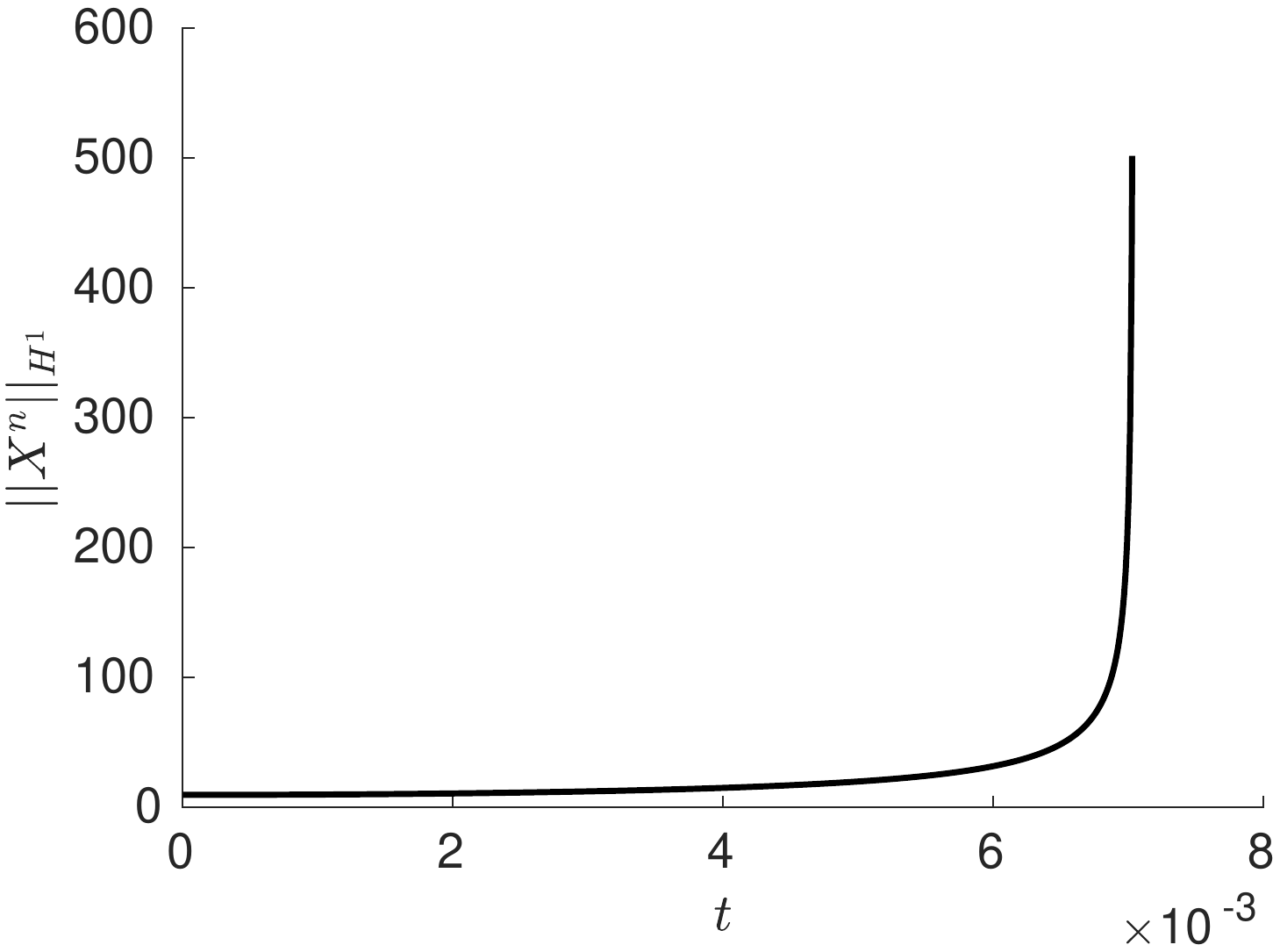}
			\end{subfigure}
			~ 
			\begin{subfigure}[t]{0.3\textwidth}
				\centering
				\includegraphics*[width =\textwidth,keepaspectratio,clip]{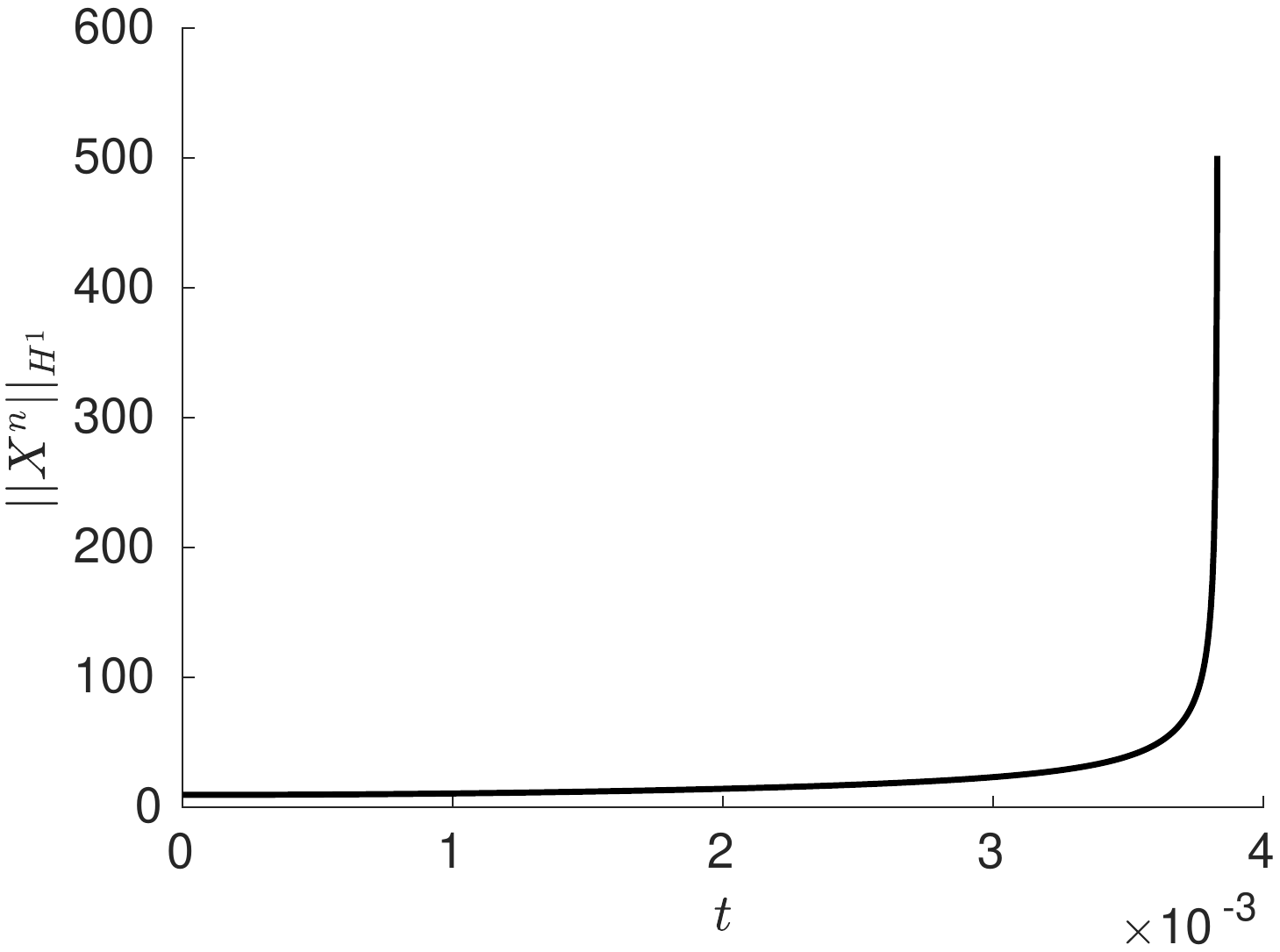}
			\end{subfigure}
			\caption{
				Evolution of $\H^1$-norms of \eqref{eq:ManakovCrit} with $\gamma=0$ and 
				$\sigma = 1, 1.25, 1.5, 1.75, 2., 2.25$ (in order left to right, top to bottom) using the initial value \eqref{IV2}.
				\label{fig:detH1new}
			}
		\end{figure}
		
		% ---- Conclusions / conjecture ----
		% Want to mention:
		% sigma \ge 2 produces blowup in the deterministic case
		% Conjecture (deterministic): sigma_crit = 2
		% sigma \ge 3 produces blowup in the stochastic case
		% Conjecture (stochastic): Noise delays/prevents blowup for sigma=2
		% Conjecture (stochastic): sigma \ge 3 produces blowup
		
		% Conclusions already present in paragraph above, therefore we only list our conjectures
		
		%		We can conclude that $\sigma\ge 2$ yields blowup for the deterministic Manakov equation, equation \eqref{eq:halfManakov} with $\gamma = 0$ (for certain initial values).
		%		Similarly, we can conclude that $\sigma\ge 3$ yields blowup for the stochastic Manakov equation \eqref{eq:halfManakov} (for certain initial values). 
		With these preliminary numerical experiments, we formulate the following two conjectures 
		regarding the Manakov equation \eqref{eq:halfManakov} in dimension $d=1$: 
		First, in the deterministic case ($\gamma=0$), blowup occurs when $\sigma\ge 2$, see left columns of 
		Figure~\ref{fig:H1HigherSigma}, Figure~\ref{fig:H1HigherSigma2L}, and Figure~\ref{fig:detH1new}. 
		Second, in the stochastic case ($\gamma>0$), blowup occurs when $\sigma>2$, see right columns 
		of Figure~\ref{fig:H1HigherSigma}, Figure~\ref{fig:H1HigherSigma2L}, and Figure~\ref{fig:H1new} 
		and Figure~\ref{fig:H1HigherSigmaBatch}.
		
		%		We conjecture that the deterministic Manakov equation (equation \eqref{eq:halfManakov} with $\gamma = 0$) in dimension $d=1$ has the critical exponent $\sigma_{\text{crit}}=2$, similar to the nonlinear Schrödinger equation, with blowup at $\sigma_{\text{crit}}$.
		%		Further, we also conjecture that \cite[Theorem 2.2]{MR2832639} extends to the stochastic Manakov equation, in that solutions in $\H^1$ exists for equation \eqref{eq:halfManakov} in dimension $d=1$ with $\sigma=2$.
		%		Lastly, we stress that the conjecture posed by \cite{bbd15}, that $\sigma_{\text{crit}} = 4/d$, does not seem to extend to the stochastic Manakov equation, given the observed blowup for $\sigma=3$.
		
		%		Similarly to the NLS, we conjecture that for the deterministic Manakov equation (equation \eqref{eq:halfManakov} with $\gamma = 0$) in dimension $d=1$ has the critical exponent $\sigma_{\text{crit}}=2$, with blowup at $\sigma_{\text{crit}}$.
		%		We also conjecture that \cite[Theorem 2.2]{MR2832639} extends to the stochastic Manakov equation, in that solutions in $\H^1$ exists for $\sigma=2$ and $d=1$.
		%		Lastly, we conjecture that the conjecture posed by \cite{bbd15}, that $\sigma_{\text{crit}} = 4/d$, does not extend to the stochastic Manakov equation, given the observed blowup for $\sigma=3$.

		\section{Acknowledgement}
		% We appreciate the referees' comments on an earlier version of the paper. 
		This work was partially supported by the Swedish Research Council (VR) (project nr. $2018-04443$), 
		FR\"O the mobility programs of the French Embassy/Institut fran\c{c}ais de Su\`ede, and INRIA Lille Nord-Europe. 
        G. Dujardin was partially supported by the Labex CEMPI (ANR-11-LABX-0007-01).
		The computations were performed on resources provided by 
		the Swedish National Infrastructure for Computing (SNIC) 
		at HPC2N, Ume{\aa} University.
		
		\bibliographystyle{plain}
		\bibliography{labib}
		
	\end{document}